\documentclass[a4paper, 11pt]{amsart}   
\usepackage{mathptmx, amssymb,amscd,latexsym, eulervm}   
\usepackage{amsmath}
\usepackage{amsthm}
\usepackage{mathdots}
\usepackage[T1]{fontenc}
\usepackage[utf8]{inputenc}
\usepackage{mathtools}
\usepackage[pagebackref,colorlinks=true,linkcolor=blue,urlcolor=blue]{hyperref}
\usepackage[initials, lite]{amsrefs}
\usepackage{color}
\usepackage{comment}
\usepackage[onehalfspacing]{setspace}
\usepackage{wasysym, stackengine, makebox, graphicx}
\newcommand\isom{\mathrel{\stackon[-0.1ex]{\makebox*{\scalebox{1.08}{\AC}}{=\hfill\llap{=}}}{{\AC}}}}
\newcommand\nvisom{\rotatebox[origin=cc] {-90}{$ \isom $}}
\newcommand\visom{\rotatebox[origin=cc] {90} {$ \isom $}}
\usepackage{tabularx}
\usepackage{amsfonts} 
\usepackage{paralist}
\usepackage{aliascnt}
\usepackage{amscd}
\usepackage{blkarray}
\usepackage{mathbbol}
\usepackage{setspace}
\usepackage[inner=2.5cm,outer=2.5cm, bottom=3.2cm]{geometry}
\usepackage{tikz, tikz-cd}
\usepackage{calligra,mathrsfs}

\usepackage{tikz}
\usetikzlibrary{matrix}
\usetikzlibrary{arrows,calc}
\allowdisplaybreaks

\BibSpec{collection.article}{%
	+{}  {\PrintAuthors}                {author}
	+{,} { \textit}                     {title}
	+{.} { }                            {part}
	+{:} { \textit}                     {subtitle}
	+{,} { \PrintContributions}         {contribution}
	+{,} { \PrintConference}            {conference}
	+{}  {\PrintBook}                   {book}
	+{,} { }                            {booktitle}
	+{,} { }                            {series}
	+{, vol.} { }                            {volume}
	+{,} { }                            {publisher}
	+{,} { \PrintDateB}                 {date}
	+{,} { pp.~}                        {pages}
	+{,} { }                            {status}
	+{,} { \PrintDOI}                   {doi}
	+{,} { available at \eprint}        {eprint}
	+{}  { \parenthesize}               {language}
	+{}  { \PrintTranslation}           {translation}
	+{;} { \PrintReprint}               {reprint}
	+{.} { }                            {note}
	+{.} {}                             {transition}
	+{}  {\SentenceSpace \PrintReviews} {review}
}
\AtBeginDocument{%
	\def\MR#1{}
}

\makeatletter
\@namedef{subjclassname@2020}{%
	\textup{2020} Mathematics Subject Classification}
\makeatother

\newcommand{\kk}{\mathbb{k}}

\newcommand{\GG}{\mathcal{G}}

\newcommand{\syl}{{\normalfont\text{syl}}}
\newcommand{\rsyl}{{\normalfont\text{rsyl}}}
\newcommand{\syz}{{\normalfont\text{syz}}}
\newcommand{\NN}{\normalfont\mathbb{N}}
\newcommand{\ZZ}{{\normalfont\mathbb{Z}}}

\newcommand{\PP}{{\normalfont\mathbb{P}}}

\newcommand{\xx}{\normalfont\mathbf{x}}
\newcommand{\yy}{\normalfont\mathbf{y}}

\newcommand{\SD}{{\normalfont\text{SD}}}

\newcommand{\mm}{{\normalfont\mathfrak{m}}}

\newcommand{\pp}{\mathfrak{p}}
\newcommand{\aaa}{\mathfrak{a}}

\newcommand{\qqq}{\mathfrak{q}}
\newcommand{\Dij}{D_{i,\alpha}}
\newcommand{\Fij}{F_{i,\alpha}}
\newcommand{\Pij}{\psi_{i,\alpha}}

\newcommand{\nnn}{\mathfrak{n}}
\newcommand{\m}{\mathfrak{m}}

\newcommand{\rank}{\normalfont\text{rank}}

\newcommand{\depth}{\normalfont\text{depth}}
\newcommand{\grade}{\normalfont\text{grade}}
\newcommand{\Tor}{\normalfont\text{Tor}}
\newcommand{\End}{\normalfont\text{End}}
\newcommand{\Ext}{\normalfont\text{Ext}}

\newcommand{\Ker}{\normalfont\text{Ker}}
\newcommand{\Coker}{{\normalfont\text{Coker}}}
\newcommand{\morl}{{\normalfont\text{morl}}}
\newcommand{\Quot}{{\normalfont\text{Quot}}}

\newcommand{\IM}{\normalfont\text{Im}}
\newcommand{\HT}{\normalfont\text{ht}}

\newcommand{\Ann}{\normalfont\text{Ann}}
\newcommand{\Supp}{\normalfont\text{Supp}}
\newcommand{\Ass}{\normalfont\text{Ass}}
\newcommand{\Sym}{\normalfont\text{Sym}}
\newcommand{\Rees}{\mathcal{R}}

\newcommand{\Hom}{{\normalfont\text{Hom}}}

\newcommand{\JJ}{\mathcal{J}}

\newcommand{\DD}{\mathbb{D}}

\newcommand{\OO}{\mathcal{O}}

\newcommand{\LL}{\mathcal{L}}

\newcommand{\HL}{\normalfont\text{H}_{\mm}}
\newcommand{\HH}{\normalfont\text{H}}

\newcommand{\gr}{\normalfont\text{gr}}

\newcommand{\AAA}{\mathcal{A}}

\newcommand{\bideg}{\normalfont\text{bideg}}
\newcommand{\Proj}{\normalfont\text{Proj}}
\newcommand{\Spec}{{\normalfont\text{Spec}}}

\newcommand{\biProj}{{\normalfont\text{BiProj}}}

\newtheorem{theorem}{Theorem}[section]

\newaliascnt{headcor}{headthm}

\aliascntresetthe{headcor}

\newaliascnt{headconj}{headthm}

\aliascntresetthe{headconj}

\newaliascnt{corollary}{theorem}
\newtheorem{corollary}[corollary]{Corollary}
\aliascntresetthe{corollary}

\newtheorem{claim}{Claim}[theorem]

\newaliascnt{lemma}{theorem}
\newtheorem{lemma}[lemma]{Lemma}
\aliascntresetthe{lemma}

\newaliascnt{conjecture}{theorem}
\newtheorem{conjecture}[conjecture]{Conjecture}
\aliascntresetthe{conjecture}

\newaliascnt{proposition}{theorem}
\newtheorem{proposition}[proposition]{Proposition}
\aliascntresetthe{proposition}

\theoremstyle{definition}
\newaliascnt{definition}{theorem}
\newtheorem{definition}[definition]{Definition}
\aliascntresetthe{definition}

\newaliascnt{notation}{theorem}

\aliascntresetthe{notation}

\newaliascnt{example}{theorem}
\newtheorem{example}[example]{Example}
\aliascntresetthe{example}

\newaliascnt{examples}{theorem}

\aliascntresetthe{examples}

\newaliascnt{remark}{theorem}
\newtheorem{remark}[remark]{Remark}
\aliascntresetthe{remark}

\newaliascnt{question}{theorem}

\aliascntresetthe{question}

\newaliascnt{questions}{theorem}

\aliascntresetthe{questions}

\newaliascnt{problem}{theorem}

\aliascntresetthe{problem}

\newaliascnt{construction}{theorem}

\aliascntresetthe{construction}

\newaliascnt{setup}{theorem}
\newtheorem{setup}[setup]{Setup}
\aliascntresetthe{setup}

\newaliascnt{algorithm}{theorem}

\aliascntresetthe{algorithm}

\newaliascnt{observation}{theorem}
\newtheorem{observation}[observation]{Observation}
\aliascntresetthe{observation}

\newaliascnt{defprop}{theorem}

\aliascntresetthe{defprop}

\DeclareFontFamily{OT1}{pzc}{}
\DeclareFontShape{OT1}{pzc}{m}{it}{<-> s * [1.100] pzcmi7t}{}
\DeclareMathAlphabet{\mathchanc}{OT1}{pzc}{m}{it}

\def\equationautorefname~#1\null{(#1)\null}
\def\sectionautorefname~#1\null{Section #1\null}
\def\subsectionautorefname~#1\null{\S #1\null}

\def\surjects{\twoheadrightarrow}
\begin{document}
	
	% ---------------Tittle and presentation----------------------------
	\title{Generalized Jouanolou duality, weakly Gorenstein rings, and applications to blowup algebras}

	\author{Yairon Cid-Ruiz}
	\address{Department of Mathematics, North Carolina State University, Raleigh, NC 27695, USA}
	\email{ycidrui@ncsu.edu}
	
	\author{Claudia Polini}
	\address{Department of Mathematics, University of Notre Dame, Notre Dame, IN 46556, USA}
	\email{cpolini@nd.edu}
	
	\author{Bernd Ulrich}
	\address{Department of Mathematics, Purdue University, West Lafayette, IN 47907, USA}
	\email{bulrich@purdue.edu}

	\date{\today}
	
	\keywords{blowup algebras, Rees algebra, symmetric algebra, Jouanolou duality, weakly Gorenstein rings, perfect pairing, Morley forms, canonical module, determinantal rings, rational maps, approximation complex.}
	\subjclass[2020]{13A30, 13H10, 13D45, 13D07, 14E05.}

	\begin{abstract}
		We provide a generalization of Jouanolou duality that is applicable to a plethora of situations. 
		The environment where this generalized duality takes place is a new class of rings, that we introduce and call weakly Gorenstein rings.
		As a consequence, we obtain a new general framework
		%of dualities 
		to investigate blowup algebras.
		We use our results to study and determine the defining equations of the Rees algebra for certain families of ideals.
	\end{abstract}

	\maketitle

	\section{Introduction}

	\noindent
	1.1.
	\textbf{Generalized Jouanolou duality and weakly Gorenstein rings.}
	\smallskip

	In a series of seminal papers Jouanolou studied elimination theory through the lens of modern algebraic geometry (\cites{Jo, Jo2, Jo3, Jo4, Jo5, Jo6}).
	One of his main tools was a new duality, nowadays dubbed \emph{Jouanolou duality}.
	Let $S$ be a positively graded  Noetherian $T$-algebra, with graded irrelevant ideal $\mm$. 
	When $S$ is a (not necessarily flat) complete intersection over $T$  and $\dim(S) = \dim(T)$, Jouanolou proved that, up to shift in degree, there are graded $S$-isomorphisms 
	$
	\HL^i(S)  \cong {}^*\Ext_T^i(S, T)
	$
	for all $i \ge 0$, where ${}^*\Ext$ denotes the graded $\Ext$ functor.
	
	The fundamental problem in elimination theory is to compute the image of a projection (see \cite{E}*{\S 14.1}), like the natural projection $\pi : X = \Proj(S) \rightarrow \Spec(T)$.
	The scheme-theoretic image of $\pi$ is given by the closed subscheme $\Spec(T/\aaa)$ with $\aaa = \Ker(T \xrightarrow{\; \rm nat \;} \HH^0(X, \OO_X))$.
	We say that $\aaa$ is the \emph{resultant ideal}.
	We have the four-term exact sequence 
	$
	0 \rightarrow \HL^0(S)_0 \rightarrow T = S_0 \xrightarrow{\; \rm nat \;} \HH^0(X, \OO_X) \rightarrow \HL^1(S)_0 \rightarrow 0.
	$
	This shows that the scheme-theoretic image of $\pi$ is given by 
	$$
	\IM(\pi) \;=\; \Spec(T/\aaa) \quad \text{with} \quad \aaa \;=\; \HL^0(S)_0.
	$$
	Therefore, Jouanolou duality gives an effective method to compute the image of a projection. 
	Indeed, if we have the isomorphism $\HL^0(S) \cong {}^*\Hom_T(S, T)(-\delta)$, we can compute the torsion part $\HL^0(S)_0$ (which involves the multiplicative structure of $S$) via the $T$-module $\Hom_T(S_\delta, T)$ (which can be computed as the kernel of the transpose of a presentation matrix of $S_\delta$ as a $T$-module).
	Using his duality, Jouanolou proved many beautiful formulas involving resultants (\cites{Jo, Jo2, Jo3, Jo4, Jo5, Jo6}).
	The expository references \cites{buse2005elimination, BUSE_CATANESE_POSTINGHEL,BC_FIB_ELIM,COX_ELIM} show how Jouanolou's work (in particular, his duality) is still relevant in modern elimination theory.

	\smallskip
	
	Our goal is to extend this duality to algebras that are not complete intersections, in fact not even Gorenstein.
	To this end, we introduce a new generalization of  Gorenstein rings.
	We call these rings \emph{weakly Gorenstein} rings.
	As in the classical case of Gorenstein rings where local duality takes place, our new notion of weakly Gorenstein rings is the natural environment where Jouanolou duality exists and can be generalized.
	Assume $S$ is Cohen-Macaulay and  $\omega_S$ is a graded canonical module of $S$.
	Let $\aaa \subset S$ be a homogeneous ideal and $i \ge 0$ be an integer.
	We say that $S$ is \emph{$i$-weakly Gorenstein with respect to $\aaa$} if there exists a fixed homogeneous element $y \in \omega_S$ that generates $\omega_S$ generically and locally in codimension at most $i$ at every prime that contains $\aaa$.
	More precisely, we require that 
	$$
	\dim\big(\Supp\left(\omega_S/Sy\right)\big) \,<\, \dim(S)  \quad \text{ and } \quad 
	\dim\big(\Supp\left(\omega_S/Sy\right) \cap V(\aaa)\big) \,<\, \dim(S) - i. 
	$$	
	
	Our first main result  says that a generalization of Jouanolou duality holds for weakly Gorenstein rings.
	%(see \autoref{thm_gen_Jou})
	Indeed, let $B=T[x_1,\ldots,x_d]$ be a positively graded polynomial ring mapping onto $S$ and assume that $S$ is a perfect $B$-module of codimension $c$. 
	If $S$ is $(i+1)$-weakly Gorenstein with respect to $\mm$, then there is a graded $S$-isomorphism
	$$
	\HL^i(S) \, \cong \,  {}^*\Ext_T^{i+c-d}(S, T)\left(\deg(y)\right)
	$$
	(see \autoref{thm_gen_Jou}).
	When $i=0$ and $S$ is standard graded with $c=d$,
	%If in addition $S$ is standard graded with $c=d$, 
	and $\delta := - \deg(y)$, 
	we prove that there
	is an isomorphism $\omicron: \, \HL^0(S)_{\delta} \xrightarrow{\; \cong \;} T$ and that the multiplication maps
	$\, \HL^0(S)_j\otimes _T S_{\delta -j} \rightarrow \HL^0(S)_{\delta} \,$ followed by $\omicron$ are perfect pairings 
	inducing isomorphisms $\, \HL^0(S)_{j} \xrightarrow{\; \cong \;}\; \Hom_T\left(S_{\delta-j}, T\right) \, $ as above (see \autoref{Gen_perfectpairing}). We provide explicit inverses of these isomorphisms, which we construct from any suitable element
	$\Delta$ in the annihilator of the diagonal ideal, the kernel
	of the multiplication map $S\otimes_T S \rightarrow S$ (see \autoref{thm_Morley_forms1}). The significance of such explicit
	inverses is that they reduce the computation of $\HL^0(S)$, as an ideal in $S$,
	to the computation of $ {}^*\Hom_T(S,T)$. Jouanolou achieved this by means of Morley forms, which our construction generalizes. We will elaborate on 
	this for 
	the special case of symmetric algebras, in the second part of the Introduction.

	\smallskip
	
	Surprisingly, many classes of algebras and ideals of interest satisfy the weakly Gorenstein condition. 
	We prove that determinantal rings tend to have this property (see \autoref{thm_det_weak_Gor}), and so do symmetric algebras  as long as they are Cohen-Macaulay (see \autoref{thm_Sym_weak_Gor}).
	Our approach to showing that symmetric algebras have the weakly Gorenstein property is by computing explicitly the canonical module. 
	The formula for the canonical module is interesting in its own right, and yields several applications.
	
	\smallskip

	We now describe a family of symmetric algebras that satisfy the weakly Gorenstein property, their canonical modules, and some related results that we obtain along the way.
	Let $(R, \mm)$ be a $d$-dimensional Cohen-Macaulay local ring, and $I = (f_1,\ldots,f_n)\subset R$ be an ideal minimally generated by $n$ elements and of codimension $g = \HT(I) \ge 2$.
	One says that $I$ has the property $F_0$ if for any $\pp \in V(I)$ the minimal number of generators satisfies $\mu(I_\pp) \le \dim(R_\pp) + 1$.
	For any $k \ge 0$, one says that $I$ has the sliding depth property $\SD_k$ if the depth of the $i$-th Koszul homology with respect to the sequence $f_1,\ldots,f_n$ is at least $d- n+ i + k$ for all $i \ge 0$.
	The condition $F_0$ is necessary for the Cohen-Macaulayness of $\Sym(I)$, and thus unavoidable to show the weakly Gorenstein condition. 
	Under the additional assumptions that $R$ is Gorenstein with infinite residue field and that $I$ satisfies $\SD_1$, in \autoref{thm_Sym_weak_Gor}, we compute explicitly the canonical module of $\Sym(I)$ and prove that $\Sym(I)$ is $d$-weakly Gorenstein with respect to the ideal $\mm \, \Sym(I).$
	%\begin{enumerate}[\rm (i)]	
	%		\item the canonical module $\omega_{\Sym(I)}$ is isomorphic to an ideal of the form $\left(f_1,y_1\right)^{g-2}\Sym(I)$.
	%		\item $\Sym(I)$ is $d$-weakly Gorenstein with respect to the ideal $\mm \, \Sym(I)$.
	%	\end{enumerate}

These results are applicable in a wide range of situations since the condition $\SD_1$ is satisfied by several classes of ideals (for instance, strongly Cohen-Macaulay ideals, and so in particular, perfect ideals of codimension two and perfect Gorenstein ideals codimension three).
Our formula for the canonical module of symmetric algebras coincides with a formula for the canonical module of certain Rees algebras (see \cite{HSV2}), which leads to  interesting consequences in \autoref{cor_beautiful}.

Our main tool to compute $\omega_{\Sym(I)}$ is a new complex that mends one of the main drawbacks of the approximation complex $\mathcal{Z}_\bullet$.
The approximation complex $\mathcal{Z}_\bullet$ is ubiquitous in the study of blowup algebras, and it provides a resolution of the symmetric algebra in many cases of relevance. 
However, the fact that is made up of Koszul syzygies, which are typically not free, can be a non trivial obstacle.  
To remedy this problem we introduce a halfway resolution that refines $\mathcal{Z}_\bullet$.
In \autoref{prop_modified_Z_complex}, we introduce the new complex $\mathcal{L}_\bullet$ that consists of free modules in the last $g-1$ positions and that coincides with  $\mathcal{Z}_\bullet$ in the remaining positions. 
The complex $\mathcal{L}_\bullet$ is acyclic when $\mathcal{Z}_\bullet$ is.

Furthermore, these halfway free resolutions lead to actual free resolutions of the symmetric algebra for special families of ideals such as almost complete intersections and perfect ideals of deviation two (see \autoref{thm_res_Sym}). 
%The usefulness of the complex $\mathcal{L}_\bullet$ came as surprise, 
The construction of the complex $\mathcal{L}_\bullet$ is notable since computing free resolutions of symmetric algebras is a problem of tall order.
%, it is somehow surprising the utility of the complex $\mathcal{L}_\bullet$. 

\medskip

\noindent
1.2.
\textbf{Applications to blowup algebras.}
\smallskip

Our generalization of Jouanolou duality and the definition of weakly Gorenstein rings provide a general  framework to study blowup algebras. 
We are particularly interested in finding the defining equations of the Rees algebra. 
Since Rees algebras appear as the coordinate ring of the blowup of a variety along a subvariety, the significance of finding their defining equations becomes apparent. 
This problem has been extensively studied by algebraic geometers, commutative algebraists, and researchers in applied areas like geometric modeling (see, e.g., \cites{KPU3,KPU4,KPU5,Buse,KM,CBD14,V_REES_EQ,COX_MOV_CURV,COX_HOFFMAN_WANG,HONG_SV,KPU_RAT_SC,CBD13,Morey, MU}). 

\smallskip

When the symmetric algebra is a complete intersection, the classical Jouanolou duality  is the standard (and the most forceful) tool to compute the defining equations of the Rees algebra (see \cites{Buse,KPU3,KM}).
We now describe and justify how our generalized Jouanolou duality can and does play a similar role when the symmetric algebra is only Cohen-Macaulay.

\smallskip

Let $\kk$ be a field, $R = \kk[x_1,\ldots,x_d]$ be a standard graded polynomial ring, $\mm = (x_1,\ldots,x_d) \subset R$ be the graded irrelevant ideal, and $I = (f_1,\ldots,f_n) \subset R$ be an ideal minimally generated by $n$ forms of degree $D \ge 1$. 
Let $T = \kk[y_1,\ldots,y_n]$ be a standard graded polynomial ring and $\nnn = (y_1,\ldots,y_n) \subset T$ be the graded irrelevant ideal.
Let $B = R \otimes_\kk T$ be a standard bigraded polynomial ring, and consider the bihomogenous epimorphism 
$$
\Phi : B \twoheadrightarrow \Rees(I) = R[It] = \bigoplus_{j = 0}^\infty I^jt^j \subset R[t], \quad x_i \mapsto x_i \text{ and } y_i \mapsto f_it.
$$
Then $\JJ = \Ker(\Phi) \subset B$ is the defining ideal of $\Rees(I)$.
The graph of the rational map 
$$
\GG : \PP_\kk^{d-1} \dashrightarrow \PP_\kk^{n-1}
$$ 
determined by the forms $f_1,\ldots,f_n$ is naturally given as $\Gamma = \biProj(B/\JJ) \subset \biProj(B) = \PP_\kk^{d-1} \times_\kk \PP_\kk^{n-1}$. 

\smallskip

Traditionally one considers the Rees algebra as a natural epimorphic image of the symmetric algebra $\Sym(I)$ of $I$ and one studies the kernel of this map,
$$
0 \rightarrow \AAA \rightarrow \Sym(I) \rightarrow \Rees(I) \rightarrow 0.
$$
The kernel $\AAA$ is the $R$-torsion submodule of $\Sym(I)$.
The defining equations of $\Sym(I)$ are easily read from a minimal presentation matrix of $I$.
An almost unavoidable constraint in the study of blowup algebras is the $G_d$ condition.
The ideal $I$ satisfies $G_d$ if  $\mu(I_\pp) \le \dim(R_\pp)$ for all $\pp \in V(I)$ such that $\HT(\pp) < d$.
Furthermore, we need the conditions $F_0$ and $\SD_1$ to show that $\Sym(I)$ has the weakly Gorenstein property with respect to $\mm \,\Sym(I)$.
Under these conditions, in \autoref{thm_Jou_dual_Sym}, we provide a general framework of dualities to study blowup algebras. 
Let $\delta = (g-1)D-d$ and $\beta = d-g+2$.
Suppose that $n = d+1$, $g = \HT(I) \ge 2$, and $I \subset R$ satisfies the conditions $G_d$ and $\SD_1$. 
We then show that the following six statements hold: 
\begin{enumerate}[(i)]
	\item $\mathcal{A} = \HL^0(\Sym(I))$.\smallskip
	
	\item For all $0 \le i \le d-1$, there is an isomorphism of bigraded $B$-modules 
	$$
	\HL^i(\Sym(I)) \, \cong \, {}^*\Ext_T^i\left(\Sym(I), T\right)\left(-\delta, -\beta\right).
	$$
	In particular, $\AAA = \HL^0(\Sym(I)) \cong {}^*\Hom_T\left(\Sym(I), T\right)\left(-\delta, -\beta\right)$.
	\smallskip
	
	\item 
	For all $i < 0$ and $i > \delta$, we have $\AAA_{(i,*)} = 0$. 
	There is an isomorphism $\AAA_{(\delta,*)} \cong T(-\beta)$ of graded $T$-modules.
	For all $0 \le i \le \delta$, we have the equality of $\Sym(I)$-ideals
	$$
	\AAA_{(\ge i,*)} \; = \; 0:_{\Sym(I)} \mm^{\delta+1-i}.
	$$
	
	\item $\AAA$ is minimally generated in $\xx$-degree at most $(g-2)D-d+1$.
	
	\item 
	Let $0 \le i \le \delta$. 
	The natural multiplication map 
	$
	\mu : \AAA_{(i,*)} \; \otimes_T \; \Sym(I)_{(\delta-i,*)} \;\rightarrow\;  \AAA_{(\delta,*)}, \; a \otimes b \mapsto a\cdot b
	$
	is a perfect pairing that induces the abstract isomorphism 
	$$
	\nu : \AAA_{(i,*)} \;\xrightarrow{\;\cong\;}\; \Hom_T\left(\Sym(I)_{(\delta-i,*)}, \AAA_{(\delta,*)}\right)
	$$
	seen in part {\rm(ii)}.
	\smallskip
	
	\item For all $2 \le i \le d+1$, there is an isomorphism of bigraded $B$-modules 
	$$
	\HH_\nnn^i(\Sym(I)) \, \cong \, {}^*\Ext_R^{i-1}(\Sym(I), R)\left(-(g-1)D, g-1\right).
	$$
\end{enumerate}
A couple of words regarding the results of \autoref{thm_Jou_dual_Sym} are in place. 
Part (i) simply means that $I$ is an ideal of linear type on the punctured spectrum on $R$.
Part (ii) comes by applying our generalization of Jouanolou duality and from the fact that we prove $\Sym(I)$ to be $d$-weakly Gorenstein with respect to the ideal $\mm \, \Sym(I)$.
Part (iii) shows that the graded components of $\AAA$ with respect to the $\xx$-grading can be read from the natural filtration 
$$
\AAA = 0:_{\Sym(I)} \mm^{\delta+1} \;\; \supset \;\;  0:_{\Sym(I)} \mm^{\delta}  \;\; \supset \;\; \cdots  \;\; \supset \;\;  0:_{\Sym(I)} \mm^2  \;\; \supset \;\; 0:_{\Sym(I)} \mm = \AAA_{(\delta,*)}.
$$
Part (iv) goes even further, it gives an upper bound for the $\xx$-degree of the minimal generators of $\AAA$, and it can actually be sharp (see \autoref{rem_bound_gen_sharp}).
Part (v) implies that the abstract isomorphism 
$$
\AAA \,=\, \HL^0(\Sym(I)) \,\cong\, {}^*\Hom_T\left(\Sym(I), T\right)\left(-\delta, -\beta\right)
$$
naturally comes from a multiplication map, and this becomes a fundamental fact in our approach to study blowup algebras. 
Part (vi) gives a generalized Jouanolou duality statement with $R$ taking the role of $T$. 
To show part (vi), since $\Sym(I)$ is only $0$-weakly Gorenstein with  respect to the ideal $\nnn \Sym(I)$ (see \autoref{rem_no_weak_Gor_nnn}), we instead rely on the isomorphism of part (ii) and a duality result of Herzog and Rahimi (\cite{HeRa}).

In a similar vein, we relate the freeness of $\AAA$ to the depth of $\Rees(I)$.
Explicitly, we prove in \autoref{prop_depth_Rees} that $\AAA$ is a free $T$-module if and only if $\depth(\Rees(I)) \ge d$.

\smallskip

%(up to multiplication by a unit in $\kk$).

A remarkable feature of the classical Jouanolou duality is that it can be made completely explicit in terms of \emph{Morley forms}.
However, the usual notion of Morley forms is not enough in our setting.
For this reason, we devise a new reduction procedure that makes explicit the perfect pairing of part (v).
We now briefly describe our generalization of the theory of Morley forms. 
Let $\DD$ be the kernel of the natural multiplication map $\Sym(I) \otimes_T \Sym(I) \rightarrow \Sym(I)$.
We choose a suitable element $\Delta \in \Sym(I) \otimes_T \Sym(I)$  in the annihilator of the diagonal ideal $\DD$ that is homogeneous of degree $\delta$ in the $x$-grading.
Our definition of Morley forms $\morl_{(\delta-i,i)}$ comes by considering the graded components of this element $\Delta$.
By applying \autoref{thm_Morley_forms1}, we show that there is an explicit and computable homogeneous $T$-homomorphism
$$
\xi \,:\, \Hom_T\left(\Sym(I)_{(\delta-i,*)}, \AAA_{(\delta,*)}\right) \rightarrow \AAA_{(i,*)}, \qquad u \mapsto\frac{\Phi_{\Delta}(\omicron \circ u)}{\alpha}
$$
that gives the inverse map of the isomorphism $\nu : \AAA_{(i,*)} \rightarrow \Hom_T\left(\Sym(I)_{(\delta-i,*)}, \AAA_{(\delta,*)}\right)$. 
Here $\Phi_\Delta$ is a map with target $\AAA$ and $\alpha \in T$ is a non zero element, both determined by $\Delta$.
For more details on the notation, see \autoref{Morley},  \autoref{subsect_graded_T}, and \autoref{subsect_Morley}.

Since $\Sym(I)$ is assumed to be Cohen-Macaulay, it is of dimension $d+1$ and $\mm \, \Sym(I)$ is a minimal prime of $\Sym(I)$. When $\Sym(I)$ is a complete intersection at the minimal prime $\mm \, \Sym(I)$, we provide a simple and direct method to find the required element $\Delta$ (see \autoref{rem_classical_Morley}). %Here one chooses a certain regular sequence, and then one defines the Morley forms $\morl_{(i,\delta-i)}$, the Sylvester form $\syl \in \AAA_{(\delta,*)}$, the reduced Sylvester form $\rsyl \in \AAA_{(\delta,*)}$, and the element $\alpha$ in $T$ given as $\syl/\rsyl$. 
Several classes of ideals satisfy the condition that $\Sym(I)$ is a complete intersection at $\mm \, \Sym(I)$ (see \autoref{rem_zero_dim_plans} and \autoref{rem_general_eqs_Sym_Gor_3}). 
If $\Sym(I)$ is a complete intersection, then $\Delta$ coincides with the element used by Jouanolou in the construction of the classical Morley forms.

\smallskip

The last part of the paper is dedicated to applications.
We show that our results can deal with cases that  were unreachable with  previously existing methods.
Our presentation here is divided in terms of two families of ideals that we treat.

\emph{Zero dimensional ideals.}
If $I \subset R$ is additionally an $\mm$-primary ideal, then we prove that $\Rees(I)$ can even be approximated by two better understood algebras. 
One is the usual symmetric algebra $\Sym(I)$.
And the other, less standard choice is the symmetric algebra $\Sym(E)$, where $E$ is the module defined by the Koszul syzygies. 
In this case, the ideal $I$ satisfies all the conditions of \autoref{thm_Jou_dual_Sym}, and so we can use $\Sym(I)$ and all the above results to approximate and to study $\Rees(I)$.
On the other hand, in  \autoref{thm_zero_dim}, we also show that there is a short exact sequence $0 \rightarrow \mathcal{B} \rightarrow \Sym(E) \rightarrow \Rees(I) \rightarrow 0$ with $\mathcal{B} = \HL^0(\Sym(E))$, and that there is an isomorphism of bigraded $B$-modules 
$$
\HL^i(\Sym(E)) \, \cong \, {}^*\Ext_T^i\left(\Sym(E), T\right)\left(-dD-d, -1\right)
$$
for all $0 \le i \le d-1$.
This opens new possibilities to study $\Rees(I)$ when $I$ is assumed be a zero dimensional ideal.
We plan to pursue this approach further in a subsequent paper.

\emph{Gorenstein ideals of codimension three.}
In addition, suppose that $I \subset R$ is a Gorenstein ideal of codimension three.
This case is covered by \autoref{thm_Jou_dual_Sym} and so one can use the above results to study $\Rees(I)$ (see \autoref{thm_Gor_3}).
Let $h \ge 1$ be the degree of the homogeneous elements in a minimal alternating presentation matrix $\varphi \in R^{(d+1)\times(d+1)}$ of $I$.
For ease of exposition, we restrict to the case $h=2$ with the extra assumption that the monomial support of the entries of $\varphi$ generates an almost complete intersection, and we leave the other cases to be studied in a subsequent paper.
For this case, in \autoref{thm_sublime_h_2}, we give a complete picture of the problem. 
We show that the ideal $\JJ \subset B$ of defining equations of $\Rees(I)$ is minimally generated by the defining ideal of $\Sym(I)$ and three forms of bidegrees $(0, 2d-2)$, $(1, d-1)$ and $(1,d-1)$.
In fact,  we explicitly compute these minimal generators of $\JJ$ in terms of the Jacobian dual of $\varphi$ and Morley forms!
In \autoref{thm_sublime_h_2}, we also prove that $\deg(\GG) = 2^{d-2}$, $\deg(Y) = 2d-2$, $\depth(\Rees(I)) \ge d$, and $\depth(\gr_I(R)) \ge d-1$, where $\mathcal{G} : \PP_\kk^{d-1} \dashrightarrow \PP_\kk^d$ is the corresponding rational map and $Y \subset \PP_\kk^d$ is the closure of the image of $\mathcal{G}$.

\medskip

\noindent
\textbf{Outline.}
The structure of the paper is as follows.
In \autoref{sect_jou_dual}, we provide an extension of Jouanolou duality and Morley forms.
In \autoref{sect_det_rings}, we identify determinantal rings that satisfy the weakly Gorenstein condition.
The weakly Gorenstein condition for symmetric algebras is studied in \autoref{sect_weak_Gor_Sym}.
In \autoref{sect_dualities_Sym}, we provide a general framework of duality statements that are relevant in the study of blowup algebras.
Lastly, \autoref{sect_applications} is dedicated to the study of specific families of ideals where we apply the methods developed in this paper.

\section{An extension of Jouanolou duality and Morley Forms}
\label{sect_jou_dual}

In this section, we provide a generalization of Jouanolou duality that can be applied to a plethora of situations.
We assume the following setup throughout. 

\begin{setup}
	\label{setup_duality}
	Let $S$ be a positively graded Noetherian ring with $S_0$ a 
	factor ring of a local Gorenstein ring $T$. 
	%$T$ be a local Gorenstein ring and $S$ be a positively graded finitely generated $T$-algebra.
	Choose a positively graded polynomial ring $B = T[x_1,\ldots,x_d]$ such that we have a graded surjection $B \twoheadrightarrow S$.
	Let $\mm  = (x_1,\ldots,x_d)$ be the graded irrelevant ideal of $B$.
	Set $b := \deg(x_1) + \cdots + \deg(x_d)$.
\end{setup}

\begin{remark}
	In this paper, we freely use basic properties of canonical modules, and our standard reference is \cite{BH}*{Chapter 3}.
	%In particular, a Cohen-Macaulay ring has a canonical module if and only if it is a quotient of a Gorenstein ring. 
	Since $T$ is assumed to be a local Gorenstein ring, the graded canonical module of $B$ is given by $\omega_B = B(-b)$ (see \cite{BH}*{\S 3.6}).
	As a consequence the graded canonical module of $S$ can be computed as $\omega_S = \Ext_B^{c}\left(S, B(-b)\right)$ where $c=\dim(B)-\dim(S)$.
\end{remark}

For a graded $B$-module $M$, we denote the graded $T$-dual as 
$$
{}^*\Hom_T(M, T) \,:=\, \bigoplus_{j \in \ZZ} \Hom_T\left([M]_{-j}, T\right),
$$
and the corresponding right derived functor as ${}^*\Ext_T^i(M, T)$.
Note that ${}^*\Ext_T^i(M, T)$ is naturally a graded $B$-module for all $i \ge 0$.
For more details on the functors ${}^*\Ext_T^i$ and their properties, the reader is referred to \cite{BH}*{\S 1.5}.

The following result yields a version of Jouanolou duality in terms of the canonical module of $S$, and it is applicable in great generality.

\begin{theorem}
	\label{thm_duality_canonical_mod}
	Assume \autoref{setup_duality}.
	Suppose that $S$ is perfect over $B$ and of codimension $c$.
	Then we have a graded isomorphism of $\,S$-modules
	$$
	\HL^i(\omega_S) \, \cong \, {}^*\Ext_T^{i+c-d}(S, T)
	$$
	for all $i \in \ZZ$.
\end{theorem}

\begin{proof}[First proof]
	Let $F_\bullet : 0 \rightarrow F_c \rightarrow \cdots \rightarrow F_1 \rightarrow F_0$ be a minimal graded $B$-resolution of $S$.
	Since $S$ is perfect of codimension $c$, $F_\bullet$ has length equal to $c$.
	We then have the following isomorphisms of graded $B$-modules
	\begin{alignat*}{3}
		{}^*\Ext_T^{i+c-d}(S, T) & \;\cong\; \HH^{i+c-d}\left({}^*\Hom_T(F_\bullet, T)\right) \\
		& \;\cong\; \HH^{i+c-d}\left({}^*\Hom_T(F_\bullet \otimes_B B, T)\right) \\
		& \;\cong\; \HH^{i+c-d}\left( \Hom_B\left(F_\bullet, {}^*\Hom_T(B, T)\right) \right) & \text{by Hom-tensor adjointness} \\
		& \;\cong\; \HH^{i+c-d}\left( \Hom_B(F_\bullet, B) \otimes_B {}^*\Hom_T(B, T) \right) \\
		& \;\cong\; \HH^{i+c-d}\left( \Hom_B(F_\bullet, B) \otimes_B \HL^d(B)(-b) \right) & \text{by graded local duality}\\
		& \;\cong\; \HH_{c-(i+c-d)}\left( \Hom_B(F_\bullet, B)[-c] \otimes_B \HL^d(B)(-b) \right) & \text{by a homological shift}\\
		& \;\cong\; \Tor_{d-i}^B\left(\omega_S(b), \HL^d(B)(-b)\right) \, = \, \Tor_{d-i}^B\left(\omega_S, \HL^d(B)\right).
	\end{alignat*}
	The last step in the above sequence of isomorphisms follows from the fact that $\Hom_B(F_\bullet, B(-b))[-c]$ is a minimal homogeneous $B$-resolution of $\omega_S$, because by assumption $S$ is perfect over $B$.
	
	The \v{C}ech complex $C_\mm^\bullet : 0 \rightarrow B \rightarrow \bigoplus_{i=1}^d B_{x_i} \rightarrow \cdots \rightarrow B_{x_1\cdots x_d} \rightarrow 0$ with respect to the sequence $x_1,\ldots,x_d$ is a complex of flat $B$-modules, $\HH^i(C_\mm^\bullet) = 0$ for $i < d$, and  $\HH^d(C_\mm^\bullet) \cong \HL^d(B)$. 
	By computing $\Tor_{d-i}^B\left(\omega_S, \HL^d(B)\right)$ via this flat resolution of $\HL^d(B)$, we obtain that 
	$$
	\Tor_{d-i}^B\left(\omega_S, \HL^d(B)\right) \,\cong\, \HH^{d-(d-i)}\left(\omega_S \otimes_B C_\mm^\bullet \right) \,\cong\,\HL^i(\omega_S).
	$$
	Finally, by combining all these isomorphisms, we obtain 
	$$
	\HL^i(\omega_S) \cong \Tor_{d-i}^B\left(\omega_S, \HL^d(B)\right) \cong {}^*\Ext_T^{i+c-d}(S, T).\qedhere
	$$
\end{proof}

\begin{proof}[Second proof]\footnote{This proof is \emph{\`a la Jouanolou} in the sense that we utilize a spectral sequence argument that is present in several parts of Jouanolou's work.}
	As before, let $F_\bullet : 0 \rightarrow F_c \rightarrow \cdots \rightarrow F_1 \rightarrow F_0$ be a minimal graded $B$-resolution of $S$.
	Since $S$ is perfect by assumption, it follows that $G_\bullet:=\Hom_B(F_\bullet, B(-b))[-c]$ yields a minimal graded $B$-resolution of $\omega_S$.
	From the functorial isomorphism $\HL^d(B) \cong {}^*\Hom_T(\Hom_B(B, B(-b)), T)$, we obtain the following isomorphism of graded complexes
	\begin{align*}
		\HL^d(G_\bullet) 
		&\; \cong \; {}^*\Hom_T\left(\Hom_B\left(G_\bullet, B(-b)\right), T\right) \\
		&\; \cong \; {}^*\Hom_T\left(\Hom_B\left(\Hom_B\left(F_\bullet, B(-b)\right)[-c], B(-b)\right), T\right) \\
		&\; \cong \; {}^*\Hom_T\left(F_\bullet[c], T\right).
	\end{align*}
	The spectral sequences coming from the second quadrant double complex $G_{\bullet} \otimes_B C_\mm^\bullet$ converge in the second pages and yield the following isomorphism 
	$$
	\HL^i(\omega_S) \,\cong\, \HH_{d-i}\left(\HL^d(G_\bullet)\right)
	$$
	for all $i \ge 0$.
	After putting the above isomorphisms together, we obtain 
	\begin{align*}
		\HL^i(\omega_S) \,\cong\, \HH_{d-i}\left(\HL^d(G_\bullet)\right) & \,\cong\, \HH_{d-i}\left({}^*\Hom_T\left(F_\bullet[c], T\right)\right)  \\
		& \,\cong\, \HH^{i+c-d}\left({}^*\Hom_T\left(F_\bullet, T\right)\right)\\
		& \,\cong\, {}^*\Ext_T^{i+c-d}\left(S, T\right).
	\end{align*}
	This concludes the second proof of the theorem.
\end{proof}

To obtain a ``true'' generalization of Jouanolou's duality in terms of $S$, we need to relate the local cohomology modules of $S$ and $\omega_S$.
For this purpose, we introduce the following general definition. 

\begin{definition}
	\label{def_weak_Gor}
	In addition to \autoref{setup_duality} suppose that $S$ is Cohen-Macaulay.
	Let $\aaa \subset S$ be a homogeneous ideal and $i \ge 0$ be an integer.
	We say that $S$ is \emph{$i$-weakly Gorenstein with respect to $\aaa$} if there exists a homogeneous element $y \in \omega_S$ such that 
	$$
	\dim\big(\Supp\left(\omega_S/Sy\right) \cap V(\aaa)\big) \,<\, \dim(S) - i 
	$$	
	and $Sy \otimes_S S_\pp \cong \omega_{S_\pp}$ for all $\pp \in \Ass(S)$. We will refer to $y$ as a \emph{weak generator} of $\omega_S.$
\end{definition}

An equivalent condition is that there exists a homogeneous element $y \in \omega_S$ that generates $\omega_S$ at the associated primes of $S$ and at all the primes containing $\aaa$ with codimension $\le i$; in particular, $S$ becomes a Gorenstein ring after localizing at all these primes. However, the weakly Gorenstein property is considerably stronger due to the uniform choice of $y$ that works for every prime in question. In fact, 
if $S$ is $1$-weakly Gorenstein with respect to a nilpotent ideal then $S$ is already Gorenstein. In general, enlarging $\aaa$ weakens the $i$-weakly Gorenstein condition, whereas increasing $i$ strengthens it.  
A simple reinterpretation of the condition is given in the following remark.

\begin{remark}\label{other}
	In addition to \autoref{setup_duality} suppose that $S$ is Cohen-Macaulay. Then $S$ is $i$-weakly Gorenstein with respect to a homogeneous ideal $\aaa \subset S$ if and only if there is a homogeneous element $y \in \omega_S$ such that
	\begin{enumerate}[\rm (i)]
		\item $\HT\left(\Ann_S\left(\omega_S/Sy\right) + \aaa\right) \ge i+1$, and 
		\item $\HT\left(\Ann_S\left(\omega_S/Sy\right)\right) \ge 1$.
	\end{enumerate}
\end{remark}

The usefulness of this definition becomes apparent with the following lemma. 

\begin{lemma}
	\label{lem_weakly_Gor}
	Assume that $S$ is a positively graded Cohen-Macaulay ring.
	Suppose $S$ is $(i+1)$-weakly Gorenstein with respect to a homogeneous ideal $\aaa \subset S$. 
	Let $y \in \omega_S$ be a weak generator of the canonical module. Then we have a graded isomorphism of $\,S$-modules
	$$
	\HH_\aaa^i(\omega_S) \,\cong\, \HH_\aaa^i(S)(-\deg(y)).
	$$
\end{lemma}
\begin{proof}
	As the canonical module is faithful, for any $\pp \in \Ass(S)$ we get $\Ann_{S_\pp}(Sy\otimes_S S_\pp) = \Ann_{S_\pp}(\omega_{S}\otimes_S S_\pp) = 0$, and so it follows that $\Ann_S(y) = 0$.
	As a consequence, we have $Sy \cong S(-\deg(y))$.
	From the short exact sequence $0 \rightarrow Sy \rightarrow \omega_S \rightarrow \omega_S/Sy \rightarrow 0$, we obtain the exact sequence in cohomology 
	$$
	\HH_\aaa^{i-1}(\omega_S/Sy)  \rightarrow \HH_\aaa^i(Sy) \rightarrow \HH_\aaa^i(\omega_S) \rightarrow \HH_\aaa^i(\omega_S/Sy).
	$$ 
	Thus, to conclude the proof it suffices to show that $\grade(\aaa ,\omega_S/Sy) \ge i+1$.
	Equivalently, we need to prove that $\depth\left((\omega_S/Sy)_\pp\right) \ge i+1$ for all $\pp \in \Supp(\omega_S/Sy) \cap V(\aaa)$.
	From the definition of weakly Gorenstein we get that $\depth(S_\pp) = \dim(S_\pp) \ge i+2$ for all $\pp \in \Supp(\omega_S/Sy) \cap V(\aaa)$.
	It then follows from the short exact sequence above that $\depth\left((\omega_S/Sy)_\pp\right) \ge i+1$ for all $\pp \in \Supp(\omega_S/Sy) \cap V(\aaa)$, as required.
\end{proof}

\begin{remark}\label{rem_indep-deg}
	If in addition to  the assumption of \autoref{lem_weakly_Gor} one has $i\ge \grade(\aaa)$, then $\deg(y)$ is independent of the choice of $y$. 
\end{remark}
\begin{proof} Write $t=\grade(\aaa)$ and notice that $\HH_\aaa^t(S)\not=0$. Since  $\HH_\aaa^t(S)(-\deg(y))  \cong \, \HH_\aaa^t(\omega_S),$ by \autoref{lem_weakly_Gor}, it follows that $\deg(y)$ only depends on $S$. 
\end{proof}

Finally, we are ready for our promised generalization of Jouanolou duality.

\begin{theorem}[Generalized Jouanolou duality]
	\label{thm_gen_Jou}
	Assume \autoref{setup_duality}.
	Suppose that $S$ is perfect over $B$ of codimension $c$, and that $S$ is $(i+1)$-weakly Gorenstein with respect to $\mm S$.
	Let $y \in \omega_S$ be a weak generator of the canonical module.
	Then we have a graded isomorphism of $\,S$-modules
	$$
	\HL^i(S) \, \cong \, {}^*\Ext_T^{i+c-d}(S, T)\left(\deg(y)\right)
	$$
	for all $i \in \ZZ$.	\end{theorem}
\begin{proof}
	This follows by combining \autoref{thm_duality_canonical_mod} and \autoref{lem_weakly_Gor}.
\end{proof}

\subsection{Perfect pairing and Morley forms}\label{Morley}
In this subsection, our goal is to show that the isomorphism of \autoref{thm_gen_Jou} for $i=0$ and $c=d$ arises from  a perfect pairing given by multiplication. In addition, we want to make this isomorphism  explicit via Morley forms.

%\begin{setup}
%	\label{setup_duality_PP}
%	Let $T$ be a local Gorenstein ring and $S$ be a  standard graded finitely generated $T$-algebra.
%	Let $\mm$ be the graded irrelevant ideal of $S$. 

%\end{setup}

\begin{observation} \label{obs1} 	Assume \autoref{setup_duality}.	The following statements hold: 
	\begin{enumerate}[\rm (i)]
		\item  The module ${}^*\Hom_T(S, T)$ is concentrated in nonpositive degrees.
		\item $\HL^0({}^*\Hom_T(S, T))={}^*\Hom_T(S, T)$.
		\item  $\left[{}^*\Hom_T(S, T)\right]_0\cong T$.
	\end{enumerate}
\end{observation}

\begin{lemma}\label{LemmaB}
	If $S$ is standard graded, then for every $i\ge 0$, 
	$$ 0:_{{}^*\Hom_T(S, T)}\, \mm^i \;=\; \big[{}^*\Hom_T(S, T)\big]_{\ge -i+1}.
	$$
\end{lemma}
\begin{proof} Since the other inclusion is clear, we only need to show that 
	$$ 0:_{{}^*\Hom_T(S, T)}\mm^i \;\subset\; \big[{}^*\Hom_T(S, T)\big]_{\ge -i+1}.
	$$
	Let $f$ be a homogeneous element of  $0:_{{}^*\Hom_T(S, T)}\mm^i$. Suppose that  $f $ has degree $j$ with $j\le -i$. It follows that the map $f$ restricted to $S_{\le i-1}$ is zero. 
	On the other hand, $\mm^i f=0$ implies that $f$ restricted to $S_{\ge i}$ is also zero.  
	We conclude that $f=0$. 
\end{proof}

The following setup is used for the remainder of this subsection.

\begin{setup}
	\label{setup_duality_PP2}
	In addition to \autoref{setup_duality}, 
	assume that $S$ is standard graded with $S_0=T,$ perfect over $B$ with $\dim(S)=\dim(T)$,  and  $S$ is $1$-weakly Gorenstein with respect to $\mm S$.
	Write $\AAA$ for $\HL^0(S)$. 
	Let $y \in \omega_S$ be a weak generator of the canonical module
	of  degree $-\delta$.
\end{setup}

The next result describes the isomorphism of \autoref{thm_gen_Jou} (with $i=0$ and $c=d$) in terms of a perfect pairing induced by the natural multiplication map. It also shows, in particular, that $\AAA \not=0$ and $\delta \ge 0$.

\begin{theorem}\label{Gen_perfectpairing} 	Assume \autoref{setup_duality_PP2}.
	The following statements hold$\, : $
	\begin{enumerate}[\rm (i)]
		\item  The $S$-module $\AAA = \HL^0(S)$ is concentrated in degree at most $\delta$.
		\item The $T$-module $\AAA_{\delta}$ is free of rank one.
		\item 
		For $0 \le i \le \delta$, the natural multiplication map
		$$
		\mu\;:\; \AAA_{i} \; \otimes_T \; S_{\delta-i} \; \xrightarrow{\; {\rm mult} \;} \;  \AAA_{\delta}$$	is a perfect pairing that induces an isomorphism
		$$	\nu \,:\,  \AAA_{i} \;\xrightarrow{\;\cong \;}\; \Hom_T\left(S_{\delta-i}, \AAA_{\delta} \right)\, .
		$$
		In addition, if we write $\AAA_{\delta} = T \mathfrak{s}$ and fix an isomorphism  $\omicron: T \mathfrak{s}\xrightarrow{\;\cong \;} T$ with $\mathfrak s \mapsto 1$, then the composition of $\mu$ and $\omicron$, 
		$$
		\AAA_{i} \; \otimes_T \; S_{\delta-i} \; \xrightarrow{\; {\rm mult} \;} \;  \AAA_{\delta}=T \mathfrak{s} \xrightarrow{\; \omicron \;}  \; T$$
		is also a perfect pairing  that induces an isomorphism
		$$
		\nu' \,:\,  \AAA_{i} \;\xrightarrow{\;\cong \;}\; \Hom_T\left(S_{\delta-i}, T\right)\,
		$$
		as in \autoref{thm_gen_Jou}.
		
		\item $0:_S \m=\AAA_{\delta}=T \mathfrak{s}\cong T.$
		%\item 
		%Assume that $T$ is a positively graded Noetherian ring over a field $\, \kk$.
		%Then the generator $\mathfrak{s}$ is homogeneous of degree $\varepsilon \ge 0$, that is $[\HL^0(S)]_{\delta}=T \mathfrak{s} \cong T(-\varepsilon)$. For $0 \le i \le \delta$, the multiplication map 
		%	$$
		%	\mu : [\HL^0(S)]_{i} \; \otimes_T \; S_{\delta-i} \;\rightarrow\;  [\HL^0(S)]_{\delta}=T \mathfrak{s} \rightarrow \; T(-\varepsilon)$$
		%	is a perfect pairing that induces an isomorphism as in \autoref{thm_duality_Jou_bigrad}.
	\end{enumerate}
\end{theorem}
\begin{proof} According to  \autoref{thm_gen_Jou} we have a graded isomorphism of $\,S$-modules
	$$
	\AAA \,=\, \HL^0(S) \, \cong \, {}^*\Hom_T(S, T)\left(-\delta\right).
	$$
	Now part (i) and (ii) follow from \autoref{obs1}. 
	The same graded isomorphism of $S$-modules identifies the map $\mu$ with the  multiplication map 
	$$
	\Hom_T(S_{\delta-i}, T) \; \otimes_T \; S_{\delta-i} \; \xrightarrow{\; {\rm mult} \;} \;  \Hom_T(S_0, T)=T \;, 
	$$
	which is a perfect pairing since it induces the identity map
	$$
	\Hom_T\left(S_{\delta-i}, T\right)  \; \xrightarrow{\; {\rm id} \;} \Hom_T\left(S_{\delta-i}, T\right). \qedhere
	$$ 
	Finally, part (iv) follows from \autoref{thm_gen_Jou}, \autoref{LemmaB}, and parts (i) and (iii). \end{proof}

\smallskip

%From now on assume \autoref{setup_duality_PP2} and assume in addition that $T$ is a positively graded Noetherian ring over a field $\kk$. Let $\varepsilon$ be the degree of $\mathfrak s$ as in \autoref{Gen_perfectpairing} part (iv). 
We will now construct explicit inverses of the maps $\nu$ 
%render the isomorphism $\Hom_T\left(S_{\delta-i}, T\right) \to [\HL^0(S)]_{i}$ explicit 
using a generalization of Jouanolou's Morley forms. 
%Our Morley forms are induced by a suitable element of degree $\delta$ in the annihilator of the diagonal ideal $\mathbb D$  of the enveloping algebra $S^e:=S \otimes_TS$. 
To any homogeneous element of degree $\delta$ in the annihilator of the diagonal ideal $\mathbb D$ we associate forms that we call {\it Morley forms} in honor of Jouanolou. 
Recall that the {\it diagonal ideal} $\, \mathbb D$ of the enveloping algebra  $S^e := S\otimes_T S $ is the kernel of the natural multiplication map $S^e \twoheadrightarrow S$. 
This ideal is generated by the elements $\overline{x_i} \otimes 1 -1\otimes \overline{x_i}$, where $\overline{x_i}$ denotes the image of $x_i$ in $S$.

\smallskip

We think of $S^e$ and of ${}^*\Hom_T({}^*\Hom_T(S, T), S)$ as $S- S$-bimodules with $S$ acting on the left and on the right.  
The largest submodules on which  the left and  right $S$-module structures coincide are $0:_{S^e} \mathbb D$ and ${}^*\Hom_S({}^*\Hom_T(S, T), S)$, respectively. 
Consider the homogeneous homomorphism of $S-S$-bimodules
$$ 
S^e \to {}^*\Hom_T\big({}^*\Hom_T(S, T), S\big)
$$
given by $s_1\otimes s_2 \mapsto (f \mapsto s_2 f(s_1))$. Restricting this map we obtain a homogeneous  $S$-linear map
$$ 
0:_{S^e} \mathbb D \to {}^*\Hom_S\big({}^*\Hom_T(S, T), S\big)\, 
$$
(see also \cite{SS}*{proof of Theorem 3.1}, \cite{K}*{Proposition F.9}, \cite{EU}*{proof of Theorem A.1}). From \autoref{obs1}(ii) we see that the target of this map is ${}^*\Hom_S({}^*\Hom_T(S, T), \HL^0(S))\, .$ Hence we obtain a homogeneous $S$-linear map
$$ 
0:_{S^e} \mathbb D \to {}^*\Hom_S\big({}^*\Hom_T(S, T), \AAA\big)\, .
$$
\autoref{thm_gen_Jou} implies
$$  {}^*\Hom_S({}^*\Hom_T(S, T), \AAA)\, \cong {}^*\Hom_S({}^*\Hom_T(S, T), {}^*\Hom_T(S, T))(-\delta)\, ,
$$
and using Hom-Tensor adjointness we obtain
$$
{}^*\Hom_S({}^*\Hom_T(S, T), {}^*\Hom_T(S, T))\cong {}^*\Hom_T(S\otimes_S {}^*\Hom_T(S, T), T)
\cong {}^*\Hom_T({}^*\Hom_T(S, T), T)\, .
$$
The last  graded $S$-module is concentrated in degrees $\ge 0$ and its degree $0$ component is  $T$. It follows that any homogeneous $S$-automorphism of 
${}^*\Hom_T(S, T)$ is the identity map times a unit in $T$. Hence all homogeneous $S$-isomorphisms of degree $\delta$ from ${}^*\Hom_T(S, T)$ to $\AAA$ are equal up to multiplication by a unit in $T$. 

\smallskip

The above discussion yields the following remark.
\begin{remark}\label{remMF}
	Every homogeneous element $\Delta = \sum_{i} s_{i,1} \otimes s_{i,2}$ of degree $\delta$  in $0:_{S^e} \mathbb D$ yields a homogeneous  $S$-linear map  
	$$
	\Phi_{\Delta}: {}^*\Hom_T(S, T) \,\to \, \AAA, \qquad u \mapsto \sum_{i} s_{i,2} u(s_{i,1})
	$$
	of degree $\delta$. If $\Phi_{\Delta}$ is a bijection, then $\Phi_{\Delta}$ is independent of the choice of $\Delta$, up to multiplication by a unit in $T$.
\end{remark}

%The next lemmas guarantee the existence of a suitable element $\Delta$. 
Let $\pi: S \twoheadrightarrow T$ be the homomorphism of $T$-algebras with $\pi(x_i)=0$ for all $i$, and consider the map $\varepsilon =S \otimes_T \pi: S^e \twoheadrightarrow S.$ Write $\eta : S^e \twoheadrightarrow S$ for the natural multiplication map and recall that $\eta(0:_{S^e} \mathbb D) ={\mathfrak d}_N(S/T)$ is the {\it Noether different} of $S$ over $T$, which defines the ramification locus of $S$ over $T$. 
To identify elements $\Delta$ as in \autoref{remMF} that provide the desired isomorphisms, we need to understand the relationship between the three ideals ${\mathfrak d}_N(S/T)$, $ \varepsilon (0:_{S^e } \mathbb D)$, and $0:_S \mm = \AAA_{\delta}$  (the last equality follows from \autoref{Gen_perfectpairing}(iv)). 
Experimental evidence supports the following conjecture:

\begin{conjecture}\label{conj}
	\begin{enumerate}[\rm (i)]  \item $\varepsilon (0:_{S^e } \mathbb D)= 0:_S \mm$.
		\item If $\, T$ is a domain, then ${\mathfrak d}_N(S/T)=({\rm rank}_T S)\cdot  \varepsilon (0:_{S^e } \mathbb D)$.
	\end{enumerate}
\end{conjecture}

\noindent
The next two results provide further evidence for this conjecture and show, in particular, that
it holds after tensoring with the total ring of quotients of $\, T$.  

\begin{remark}\label{elem2} %One has ${\mathfrak d}_N(S/T) \subset 0:_S \mm$ and 
	One has  $\varepsilon (0:_{S^e } \mathbb D)\subset 0:_S \mm$. 
\end{remark} 
\begin{proof} This is clear since $\varepsilon(\mathbb D)=\mm$. 
\end{proof}

%The next proposition shows that \autoref{conj} holds after tensoring with the the total ring of quotients of $\, T$.

\begin{proposition}\label{elem} Assume \autoref{setup_duality_PP2}. For   $L$ the total ring of quotients of $\, T$, we let $S_L$ be the standard graded $L$-algebra $S \otimes_T L$, and $S^e_L$ and  $\mathbb D_L$ be the corresponding enveloping algebra and diagonal ideal. 	 The following statements hold$\, : $
	\begin{enumerate}[\rm (i)] 
		\item 
		$0$\, : $_{S^e_L } \mathbb D_L \cong S_L(-\delta) $. 
		
		\item 	For every associated prime $\pp$ of $\, \mm S$, we have $[0:_{S^e_L} \mathbb D_L]_\delta\not\subset (1 \otimes \pp) \, S_L^e\, .
		$	In particular, 
		$$[0:_{S^e} \mathbb D]_\delta\not\subset (1 \otimes \pp)\, S^e\, .
		$$
		\item For every associated prime $\mathfrak q$ of $\, T$, $\varepsilon([0:_{S^e} \mathbb D]_\delta)\not\subset {\mathfrak q} S \, .
		$ In fact, if $\Delta \in [0:_{S^e} \mathbb D]_\delta$ with $\Delta \not\in (1  \otimes \pp)\, S^e$ for every associated prime $\pp$ of $\mm S$, then $\varepsilon(\Delta)\not\in \mathfrak q S\, $ for every associated prime $\mathfrak q$ of $\, T$. 
		\item $\varepsilon(0:_{S^e_L} \mathbb D_L)=\varepsilon([0:_{S^e_L} \mathbb D_L]_\delta)=0:_{S_L} \mm_L\, ;$ in particular,  $\varepsilon(0:_{S^e} \mathbb D)_{\mathfrak q}=\varepsilon([0:_{S^e} \mathbb D]_\delta)_{\mathfrak q}=(0:_S \mm)_{\mathfrak q}\, $ for every associated prime $\mathfrak q$ of $\, S$.
		\item If $\, L$ is a field, then $\, {\mathfrak d}_N(S_L/L)=({\rm dim}_L (S_L))\cdot \varepsilon(0:_{S^e_L} \mathbb D_L)\, .$
	\end{enumerate}
\end{proposition}
\begin{proof} 
	The ring $L$ is a finite product of Artinian local Gorenstein rings.  
	Since $\dim(S)=\dim(T)$, the $B$-module $S$ is perfect of codimension $d$. Therefore $S_L$ is a perfect module of codimension $d$ over $B_L=B\otimes_T L=L[x_1,\ldots,x_d]$. 
	Thus $S_L$ is a finite module over $L$ of finite projective dimension. 
	Now the Auslander-Buchsbaum formula, applied to the factors of $L$, shows that $S_L$ is flat over $L$. 
	
	By \autoref{setup_duality_PP2}, 
	$\omega_{S_L}\cong S_L \cdot y\cong S_L (\delta).$ 
	Since $S_L$ is flat over $L$, it follows that $\omega_{S_L^e}\cong S_L^e (2 \delta),$  as can be seen from a homogeneous free resolution of $S_L$ over $B_L$. In particular $S_L$ is Gorenstein.

	%Since $S$ is generically Gorenstein, $S_L$ is Gorenstein. 
	%Thus the enveloping algebra $S_L\otimes_L S_L$ is Gorenstein, necessarily Artinian since it is a finite module over $L$.
	
	%By \autoref{setup_duality_PP2}, 
	%$\omega_{S_L}\cong S_L \cdot y\cong S_L (\delta),$
	%which implies  $ as $S_L$ is flat over $L$. 
	
	Now part (i) follows because
	$$ 0:_{S_L^e} \mathbb{D}_L \cong \Hom_{S_L^e} \big(S_L^e/ \mathbb D_L, S_L^e\big) \cong \Hom_{S_L^e} \big(S_L, \omega_{S_L^e}(-2\delta)\big) \cong \omega_{S_L}(-2\delta)\cong S_L(-\delta)\, .
	$$

	As to part (ii), since $[0:_{S^e_L} \mathbb D_L]_\delta$ generates the ideal  $0:_{S^e_L} \mathbb D_L$ by part (i), it suffices to prove   	$0:_{S^e_L} \mathbb D_L \not\subset (1\otimes \pp) \, S_L^e\, .
	$ Suppose the contrary. 
	Since $S^e_L$ is an Artinian Gorenstein ring, it follows that 
	$$ \mathbb D_L = 0 :_{S^e_L} (0 :_{S^e_L} \mathbb D_L) \supset 0 :_{S^e_L} \left((1\otimes \pp ) \, S^e_L\right) \supset \left(1 \otimes (0 :_{S_L} \pp) \right)\, S^e_L\, .
	$$
	Applying the multiplication map $S^e_L \to S_L$ that has $\mathbb D_L$ as its kernel, one sees that $0 :_{S_L} \pp=0$. This is impossible because $S_L$ is Artinian and $\pp S_L\not= S_L$ as $\pp =\qqq +\mm S$ for some associated prime $\qqq$ of $T$. 
	
	If $\qqq$ is an associated prime of $T,$  then $\pp=\qqq+\mm S$ is  an associated prime of $\mm S, $ and $\varepsilon^{-1}(\mathfrak q S)=(1 \otimes \pp)S^e $. 
	Hence we can use part (ii) to establish part (iii). 
	
	By \autoref{elem2}, part (iv) follows once we show that the inclusion  $\varepsilon([0:_{S^e_L} \mathbb D_L]_\delta)\subset 0:_{S_L} \mm_L$ is an equality. \autoref{Gen_perfectpairing}(iv) implies that  $0:_{S_L} \mm_L=L\mathfrak{s}$, hence $\varepsilon([0:_{S^e_L} \mathbb D_L]_\delta)=K\mathfrak{s}$ for some ideal $K$ of $L$. 
	If $K \not=L$, then $K$ is contained in a prime ideal of $L$, contradicting part (iii). 
	
	Let $\mathfrak q$ be an associated, hence minimal, prime of $\, S$. If $\mathfrak q$ contains $\mm S$, then $\mathfrak q$ contracts to a minimal prime of $T$ and the asserted equality locally at $\mathfrak q$ follows from the one just proved. If on the other hand $\mathfrak q$ does not contain $\mm S$, then $(0 :_S \mm)_{\mathfrak q}=0$ and we are done by \autoref{elem2}.

	Part (v) is a consequence of part (iv) and the equality  $ {\mathfrak d}_N(S_L/L)=({\rm dim}_L(S_L))\cdot (0:_{S_L} \mm_L) . $ 
	This equality follows from  \cite{EU}*{Theorems A.1 and A.5} and the fact that if $\dim_L(S_L)$ is a multiple of the characteristic then the trace map ${\rm Tr}:= {\rm Tr}_{S_L/L}$ is zero; indeed ${\rm Tr}(L)=0$ by the assumption on the characteristic and ${\rm Tr}(\mm S_L)=0$ because the elements of $\mm S_L$ are nilpotent.
\end{proof}

We choose an element $\Delta \in [0:_{S^e} \mathbb D]_\delta$ with  $\Delta \not\in (1 \otimes \pp)\, S^e$ for every associated prime $\pp$ of $\mm S$. If the residue field of $T$ is infinite, such an  element exists by \autoref{elem}(ii) and any general element in $[0:_{S^e} \mathbb D]_\delta$ will do. 
We consider $S^e$ as a standard bigraded $T$-algebra with 
$$
\text{bideg}(x_i \otimes 1) = (1, 0) \quad \text{ and } \quad \text{bideg}(1 \otimes x_i) = (0, 1).
$$ Thus we obtain the decomposition 
$$
\Delta \; = \; \sum_{i=0}^{\delta} \morl_{(\delta-i,i)} \quad\text{ where }\quad \morl_{(\delta-i,i)} \in \big[S^e\big]_{(\delta-i,i)}.
$$
We say that $\morl_{(\delta-i,i)}$ is the \emph{$i$-th Morley form associated  to $\Delta$}.
We now list some basic properties of these Morley forms. 
Recall that by \autoref{Gen_perfectpairing}(iv),
$$0:_S \mm \,=\, \AAA_{\delta} \,=\, T \mathfrak{s} \,\cong\, T\, .
$$

\begin{lemma}
	\label{lem_Morley_forms2} 
	The following statements hold$\, : $ 
	\begin{enumerate}[\rm (i)]
		%\item $\morl_{(\delta,0)} = \alpha_1 \cdot \mathfrak{s} \otimes 1 \in \Sym(I)_\delta \otimes_T \Sym(I)_0$ for some unit $\alpha_1 \in \kk$.
		\item $\morl_{(\delta,0)} =  \alpha \cdot \mathfrak{s}  \otimes 1\in S_\delta\otimes_T S_0\, $ for  some non zerodivisor $\alpha$ in  $T$.
		%\item For any $ \in B$, we have that $(L \otimes 1 - 1 \otimes L) \cdot \Delta = 0 \in \Sym(I) \otimes_T \Sym(I)$.
		\item For any $b \in S_l$ with $l\le i$, we have the equality 
		$$
		(b \otimes 1) \cdot \morl_{(\delta-i,i)} \; = \; (1 \otimes b) \cdot \morl_{(\delta-i+l,i-l)}   \; \in S_{\delta-i+l} \otimes_T S_{i}. 
		$$
	\end{enumerate}
\end{lemma}
\begin{proof} As for part (i), notice that $\morl_{(\delta,0)} =\varepsilon(\Delta) \otimes 1.$
	\autoref{elem2} and \autoref{Gen_perfectpairing}(iv) show that
	$\varepsilon(\Delta)=\alpha \cdot {\mathfrak s}$ for some $\alpha\in T$, and $\alpha$ is a non zerodivisor by 
	\autoref{elem}(iii). 
	Part (ii) is obvious since $\Delta \in 0:_{S^e}\mathbb D\, .$
\end{proof}

%We are now ready for the following theorem that makes explicit the perfect pairing induced by the natural multiplication map $\mu : \AAA_i \otimes_T S_{\delta-i} \rightarrow \AAA_\delta$. 
In the following theorem by utilizing our generalized Morley forms, we obtain an explicit inverse of the isomorphism $\nu : \AAA_i \rightarrow \Hom_T(S_{\delta-i}, \AAA_\delta)$ from \autoref{Gen_perfectpairing}(iii).  
For the statement of \autoref{thm_Morley_forms1}, we observe that $\AAA_i$ is $T$-torsionfree owing to the isomorphism $\AAA_i \cong \Hom_T(S_{\delta-i}, T)$ and that the element $\alpha$ of  \autoref{lem_Morley_forms2}(i) is a non zerodivisor in $T$.

\begin{theorem}\label{thm_Morley_forms1} Assume \autoref{setup_duality_PP2}, let $\mathfrak s$ and $\omicron$ be as in \autoref{Gen_perfectpairing}(iii), let $\alpha$ be as in \autoref{lem_Morley_forms2}(ii), and let $\Phi_{\Delta}$ be as in \autoref{remMF}.
	Let $0 \le i \le \delta$.
	The following statements hold:
	\begin{enumerate}[\rm (i)]
		\item For any $u \in \Hom_T(S_{\delta-i}, \AAA_\delta)$, we have that $\Phi_{\Delta}({\omicron}\circ u)\in \alpha \cdot \AAA_i$.
		%$\big(1 \otimes \mu_{\mathfrak s}^{-1} \circ u\big)(\morl_{(i,\delta-i)}) \in \alpha \cdot \AAA_i$.
		We have a $T$-homomorphism 
		$$
		\xi \,:\, \Hom_T(S_{\delta-i}, \AAA_\delta) \rightarrow \AAA_i, \qquad u \mapsto\frac{\Phi_{\Delta}({\omicron}\circ u)}{\alpha}.
		$$
		%$$
		%\xi \,:\, \Hom_T(S_{\delta-i}, \AAA_\delta) \rightarrow \AAA_i, \qquad u \mapsto\frac{\big(1 \otimes \mu_{\mathfrak s}^{-1} \circ %u\big)(\morl_{(i,\delta-i)})}{\alpha}.
		%$$
		\item The $T$-homomorphisms 
		$$
		\nu \,:\, \AAA_i \rightarrow \Hom_T(S_{\delta-i}, \AAA_\delta)  \quad \text{ and } \quad \xi \,:\, \Hom_T(S_{\delta-i}, \AAA_\delta) \rightarrow \AAA_i
		$$
		are inverse to each other.
	\end{enumerate}
\end{theorem}
\begin{proof}
	For any $a \in \AAA_i$ and the  corresponding multiplication map 
	$$
	\nu(a) \in \Hom_T(S_{\delta-i}, \AAA_\delta),
	$$ 
	we have the equalities 
	\begin{align*}
		\Phi_{\Delta}({\omicron}\circ \nu(a)) \;&=\; 
		\big(\omicron \otimes {\rm id}_S \big)\big((a \otimes 1) \cdot \morl_{(\delta-i,i)}\big) \\
		\;&=\; \big(\omicron \otimes {\rm id}_S \big)\big((1 \otimes a) \cdot \morl_{(\delta,0)}\big) &\quad \text{by \autoref{lem_Morley_forms2}(ii)}\\
		\;&=\; \big(\omicron \otimes {\rm id}_S \big)\big((1 \otimes a) \cdot ( \alpha \cdot \mathfrak s \otimes 1) \big) &\quad \text{by \autoref{lem_Morley_forms2}(i)}\\
		\;&=\; \alpha \cdot a \;\in \; \alpha \cdot \AAA_i.
	\end{align*}
	Since the map $\nu \,:\, \AAA_i \rightarrow \Hom_T(S_{\delta-i}, \AAA_\delta)$ was shown to be an isomorphism in \autoref{Gen_perfectpairing}(iii), we obtain that 
	$$\Phi_{\Delta}({\omicron}\circ u)\in \alpha \cdot \AAA_i \quad \text{ for all } \quad u  \in \Hom_T(S_{\delta-i}, \AAA_\delta)\, .
	$$
	As $\AAA_i$ is $T$-torsionfree and $\alpha$ is a non zerodivisor on $T$, division by $\alpha$ in $\alpha \cdot \AAA_i$ is well-defined. This completes the proof of part (i). 
	
	Notice that the above computation already shows that $\xi \circ \nu=\text
	{id}_{\AAA_i}$. Since $\nu$ is an isomorphism by \autoref{Gen_perfectpairing}(iii), $\xi$ is necessarily the inverse map of $\nu$. 
	%Let $u \in \Hom_T(S_{\delta-i}, \AAA_\delta)$, and by \autoref{Gen_perfectpairing} (iii) choose $a \in \AAA_i$ such that $u = \nu(a)$.
	%Then, up to multiplication by a unit in $\kk$, we obtain the equalities 
	%$$
	%(\nu \circ \xi)(u) \,=\, \nu\big((\xi \circ \nu)(a)\big) \,=\, \nu(a) \,=\, u.
	%$$
	%It follows that $\nu \circ \xi$ equals $\text{id}_{\Hom_T(S_{\delta-i}, \AAA_\delta)}$ up to multiplication by a unit in $T$. 
	%So, part (ii) of the theorem is also complete.
\end{proof}

The map $\xi$ in \autoref{thm_Morley_forms1}(i) can be made even more explicit by
using Morley forms and multiplication of polynomials and inverse polynomials:
There is a natural embedding followed by an isomorphism of graded $B$-modules
$$ \iota: {}^*\Hom_T(S, T) \hookrightarrow {}^*\Hom_T(B, T) \cong T[x_1^{-1}, \ldots, x_d^{-1}] .$$
Clearly, the $S$-module and the $B$-module structure of $\IM(\iota)$ coincide. So if 
$s\in S$ and $h \in \IM(\iota)$, then $s \cdot h = \widetilde{s} \, h,$ where $\widetilde{s} \in B$ is any preimage of $s$ and the multiplication on the right hand side is simply multiplication 
of a polynomial  and an inverse polynomial via the $B$-module structure of $T[x_1^{-1}, \ldots, x_d^{-1}] .$ Tensoring with
$S$ one obtains a homomorphism of graded $B\otimes_TS$-modules
$$\psi: {}^*\Hom_T(S, T) \otimes_TS \longrightarrow  T[x_1^{-1}, \ldots, x_d^{-1}]\otimes_TS,$$
and again the $S^e$-module and the $B\otimes_TS$-module structure of $\IM(\psi)$ coincide. In
concrete terms, if $\beta \in S^e$ and $H \in \IM(\psi),$ then
$$\beta \cdot H = \widetilde{\beta} \, H,$$
where $\widetilde{\beta} \in B\otimes_TS$ is any preimage of $\beta.$ 

With these identifications, the map $\Phi_{\Delta}$ of \autoref{remMF} becomes
$$\Phi_{\Delta}: \ {}^*\Hom_T(S, T) \, \longrightarrow  \, T\otimes_T \AAA \ , \
\ \ \ w \mapsto \Delta \cdot \psi(w \otimes 1)\, ,$$
and with notation as in \autoref{thm_Morley_forms1}(i) we obtain
$$\Phi_{\Delta}({\omicron}\circ u)=\Delta \cdot \psi(({\omicron}\circ u)\otimes 1 ).$$
In this equality, we can replace $\Delta$ by $\morl_{(\delta -i, i)}$
because ${\omicron}\circ u$ is homogeneous of degree $i-\delta.$ Thus
we have proved the following corollary, which is needed in \autoref{sec_last}:

\begin{corollary}\label{cor_Morley_forms} With the assumptions
	of \autoref{thm_Morley_forms1} and notation as in the discussion above, we have
	$$T\otimes_T \AAA_i= {\frac{1}{\alpha}} \cdot \morl_{(\delta -i, i)} \cdot \psi\left(\Hom_T(S_{\delta-i},T)\otimes 1 \right).$$
\end{corollary}

\vspace{.2cm}

\subsection{The case where the coefficient ring $T$ is graded}
\label{subsect_graded_T}

In this short subsection, we deal with the case where $T$ is a graded ring. 
This case is of particular importance due to its applicability in the study of blowup algebras 
that we are going to pursue later. 
The proofs are exactly the same and one only needs to indicate the necessary shifts in bidegree. 
\autoref{def_weak_Gor} of the weak Gorenstein property can be easily adapted to the bigraded setting. 
\begin{setup}
	\label{setup_duality_bigrad}
	Let $T$ be a positively graded Gorenstein ring with $T_0$ a local ring and let $\omega_T \cong T(a)$ with $a \in \ZZ$ be its canonical module.
	Let $B = T[x_1,\ldots,x_d]$ be a bigraded polynomial ring such that $\bideg(t) = (0, \deg(t))$ for all homogeneous $t \in T$ and $\bideg(x_i) = (b_i, 0)$ with $b_i > 0$ a positive integer for all $1 \le i \le d$.
	Let $S=B/\JJ,$ where $\JJ$ is a bihomogenous ideal.
	%be a bigraded algebra given as a quotient of $B$.
	Let $\mm \subset B$ be the  ideal $\mm = (x_1,\ldots,x_d)$.
	Set $b := b_1 + \cdots + b_d$.
\end{setup}

We now restate our generalization of Jouanolou duality in the current bigraded setting.

\begin{theorem}
	\label{thm_duality_can_bigrad}
	Assume \autoref{setup_duality_bigrad}.
	Suppose that $S$ is perfect over $B$ and of codimension $c$.
	Then we have a bigraded isomorphism of $\,S$-modules
	$$
	\HL^i(\omega_S) \, \cong \, {}^*\Ext_T^{i+c-d}\big(S, T\big)(0,a)
	$$
	for all $i\in \ZZ.$
\end{theorem}
\begin{proof}
	Either of the proofs of \autoref{thm_duality_canonical_mod} adapt to this case directly. 
	One only needs to notice that the bigraded canonical module of $B$ is $\omega_B = B(-b, a)$, and so  $\omega_S \cong \Ext_B^c(S, B(-b, a))$.
	%\cong \HH^c\left(\Hom_B\left(F_\bullet, B(-b,a)\right)\right)$ after choosing a minimal bigraded $B$-resolution $F_\bullet$ of $S$.
\end{proof}

\begin{theorem}[Generalized Jouanolou duality]
	\label{thm_duality_Jou_bigrad}
	Assume \autoref{setup_duality_bigrad}.
	Suppose that $S$ is perfect over $B$ of codimension $c$, and that $S$, as a bigraded ring, is $(i+1)$-weakly Gorenstein with respect to $\mm S$.
	If $y \in \omega_S$ is a bihomogeneous weak generator of $\omega_S,$ 
	%element as in \autoref{def_weak_Gor}.
	then there is a bigraded isomorphism of $\,S$-modules
	$$
	\HL^i(S) \, \cong \, {}^*\Ext_T^{i+c-d}\big(S, T\big)\big(\bideg(y) + (0,a)\big).
	$$
\end{theorem}
\begin{proof}
	Since \autoref{lem_weakly_Gor}  also holds in a bigraded setting, the result follows from \autoref{thm_duality_can_bigrad}.
\end{proof}

For the rest of this subsection we assume, in addition to \autoref{setup_duality_bigrad},  that $b_1=\cdots=b_d=1$, $S$ is perfect over $B$ with ${\rm dim}(S)={\rm dim}(T)$, and $S$ is $1$-weakly Gorenstein with respect to $\mm S$. Let $y \in \omega_S$ be a bihomogeneous 
%element as in \autoref{def_weak_Gor} 
weak generator of the canonical module of bidegree $(-\delta, -\gamma)$. 

In this case the element $\mathfrak s$ of \autoref{Gen_perfectpairing} can be chosen to be bihomogeneous of bidegree $(\delta, \gamma -a)$, the map $\nu$ of \autoref{Gen_perfectpairing}(iii) is homogeneous, and the map $\nu'$ of \autoref{Gen_perfectpairing}(iii) is homogeneous of degree $a-\gamma$. If the element $\Delta$ of \autoref{remMF} is bihomogeneous of bidegree $(\delta, \rho)$, then so is the induced map $\Phi_{\Delta}$. 

We consider the enveloping algebra $B\otimes_T B$ as a trigraded ring with 
$$
\text{trideg}(x_i \otimes 1) = (1, 0, 0),\; \text{trideg}(1 \otimes x_i) = (0, 1, 0),\; \text{trideg}(t) = (0,0, {\rm deg} (t)) \text{ for homogeneous } t\in T.
$$
This induces a triple grading on $S^e$. If the element $\Delta \in [0:_{S^e} \mathbb D]_\delta$ is bihomogeneous of bidegree $(\delta, \rho)$, we obtain the decomposition
$$
\Delta \; = \; \sum_{i=0}^{\delta} \morl_{(\delta-i,i)} \quad\text{ where }\quad \morl_{(\delta-i,i)} \in \big[S^e\big]_{(\delta-i, i,\rho)}.
$$
With this grading the element $\alpha$ of \autoref{lem_Morley_forms2} and the isomorphism $\xi$  of \autoref{thm_Morley_forms1} are homogeneous of degrees $a$ and $\gamma-a,$ respectively.

\vspace{.3cm}

\section{Weakly Gorenstein determinantal rings}
\label{sect_det_rings}

In this section, we study the weakly Gorenstein condition for determinantal rings.
Our findings show that under reasonable assumptions the weakly Gorenstein property holds.
These results combined with our generalization of Jouanolou duality (\autoref{thm_duality_Jou_bigrad}) yield new tools to study local cohomology modules of determinantal rings.

\begin{setup}
	\label{setup_det_sect}
	Let $B$ be a Noetherian ring and $\varphi \in B^{m \times n}$ be an $m \times n$ matrix with $m \le n$.
	Let $1 \le t \le m$ and $S$ be the determinantal ring $S := B/I_t(\varphi)$. 
	We consider the following submatrices of $\varphi$
	$$
	\left[
	\begin{array}{c|ccc}
		\delta & &*  & \\
		\hline
		& &  & \\
		*& & * & \\
		& &  & \\
	\end{array}		
	\right], 
	\qquad
	\left[
	\begin{array}{c|ccc}
		& & & \\
		& &  & \\
		\varphi'& & * & \\
		& &  & \\
	\end{array}		
	\right],
	\qquad	
	\left[
	\begin{array}{cccc}
		& & \varphi'' & \\
		\hline
		& &  & \\
		& & * & \\
		& &  & \\
	\end{array}		
	\right]
	$$
	where $\delta \in B^{(t-1) \times (t-1)}$, $\varphi' \in B^{m \times (t-1)}$, and $\varphi'' \in B^{(t-1) \times n}$.
	Let $\Delta := \det(\delta) \in B$.
\end{setup}

We first prove an elementary lemma.

\begin{lemma}
	\label{lem_inclusion_dets}
	We have the inclusion $I_{t-1}(\varphi')\cdot I_{t-1}(\varphi'') S \subset \Delta S$ in $S$.
\end{lemma}
\begin{proof}
	We proceed by induction on $t$.
	The case $t=1$ is vacuous.
	As we have $\varphi_{i,1}\varphi_{1,j} - \varphi_{1,1}\varphi_{i,j} \in I_2(\varphi)$ for all $i, j$, the claim follows for the case $t = 2$.
	
	Suppose that $t \ge 3$.
	Since containments are preserved under extensions, we may assume that $B = \ZZ[x_{i,j}]$ where $x_{i,j}$ are independent variables and $\varphi = (x_{i,j})$ is the generic $m\times n$ matrix.
	Since $S$ is a Cohen-Macaulay domain (see \cite{HOCHSTER_EAGON}*{Corollary 4}) and $\Delta S \neq 0,$ all the associated primes of $\Delta S$ have height one.
	
	Let $P \in \Spec(S)$ be a minimal prime of $(x_{1,1}, x_{1,2})S$. 
	Set $p = P \cap \ZZ$,  $\overline{B} = B \otimes_\ZZ \Quot(\ZZ/p)$, and $\overline{S} = S \otimes_\ZZ \Quot(\ZZ/p)$.
	Notice that $\HT(P) \ge \HT(P\overline{S})$.
	Since $t \ge 3$, the minimal monomial generators of the initial ideal of $I_t(\varphi)\overline{B}$ with any antidiagonal term order do not involve the variables $x_{1,1}$ and $x_{1,2}$ (see, e.g., \cite{MS}*{Theorem 16.28}), and this implies that $x_{1,1}, x_{1,2}$ form a regular sequence on $\overline{S}$ (see, e.g.,  \cite{E}*{Proposition 15.15}).
	Thus $\HT(P \overline{S}) \ge 2, $ and so $\HT(P) \geq 2.$ Since this holds for every $P,$ it follows that $\HT((x_{1,1}, x_{1,2})S) = 2.$
	Hence, there is an element $b \in B$ such that the image of $x_{1,1} + bx_{1,2}$ in $S$ is in no associated prime of $\Delta S$.
	After elementary column operations, we obtain that $\varphi_{1,1}$ is not contained in any of the associated primes of $\Delta S$, and so it suffices to show the containment after localizing at the element $\varphi_{1,1}$.
	Once $\varphi_{1,1}$ is invertible, we can apply standard arguments to reduce to the case $t-1$.
\end{proof}

The following result provides families of determinantal rings that have the weakly Gorenstein property. 
%We say that a matrix $\varphi$ as in \autoref{setup_det_sect} is {\it homogeneous} if there exist 
%integers $a_i$ and $b_j$ so that the nonzero entries $\varphi_{i,j}$ are homogeneous of degrees $b_j-a_i$ for all $i,j.$

\begin{theorem}
	\label{thm_det_weak_Gor}
	In addition to  \autoref{setup_det_sect}, assume that $B$  is  normal, Gorenstein, and local with infinite residue field.
	Suppose that 
	$$
	\HT\left(I_t(\varphi)\right) = (m-t+1)(n-t+1) \quad \text{ and }\quad  \HT\left(I_{t-1}(\varphi)\right) > (m-t+1)(n-t+1).
	$$
	%Set $S=B/I_t(\varphi)$ and 
	Let $\aaa \subset S$ be an ideal.
	
	After elementary row operations, we may assume that $\, \HT\left(I_{t-1}(\varphi'')S\right) > 0$. 
	Then $S$ is $\, i$-weakly Gorenstein with respect to $\aaa ,\,$ for $i = \HT\left(I_{t-1}(\varphi'')S + \aaa \right)-1$. 
	
	After elementary column operations, we may further assume that $\HT\left(I_{t-1}(\varphi')S\right) > 0.$ In this case  $I_{t-1}(\varphi')^{n-m}S$ is the canonical module of $\, S,$ and the image of $\, \Delta^{n-m}$ in this ideal is a weak generator
	of the canonical module.

\end{theorem}
\begin{proof}
	%	The condition $\HT\left(I_t(\varphi)\right) = (m-t+1)(n-t+1) =: c$ implies that $S$ is a Cohen-Macaulay ring.
	Since the row and column spaces of $\varphi \otimes_B S$ have rank $t-1$, we may assume that  $\HT\left(I_{t-1}(\varphi'')S\right) \ge  1$ after elementary row operations and $\HT\left(I_{t-1}(\varphi')S\right) \ge 1$ after elementary column operations.
	Let $B' = B[X]$ where $X = (x_{i,j})$ is a generic $m\times n$ matrix with $x_{i,j}$ independent variables, and set $S' = B'/I_t(X)$.  
	From \cite{Br}*{Theorem} we know the canonical module of $S'$  satisfies the isomorphism $\omega_{S'} \cong I_{t-1}(X')^{n-m}S'$, where $X'$ is the submatrix given by the first $t-1$ columns of $X$ (see also \cite{BV}*{Chapter 8}). The ring $S'$ is Cohen-Macaulay,
	$S \cong S'/I_1(X-\varphi)$, and
	the entries of the matrix $X - \varphi$ form a regular sequence on $S'$ (notice that $\HT(I_t(X))=\HT(I_t(\varphi))$ by assumption, and thus $\dim(S) = \dim(S')-mn$). It follows that
	$$
	\omega_S \, \cong \, \omega_{S'} \otimes_{S'} S \,\cong\,  I_{t-1}(X')^{n-m}S' \otimes_{S'} S \, \surjects\, I_{t-1}(\varphi')^{n-m}S.
	$$
	Thus we  obtain a surjection  
	$$
	\omega_S \surjects I_{t-1}(\varphi')^{n-m}S.
	$$ 
	As $\HT\left(I_{t-1}(\varphi)\right) > (m-t+1)(n-t+1)$, $S$ is generically a complete intersection on the Gorenstein ring $B$, hence generically Gorenstein.
	It follows that $\omega_S$ has rank $1$.
	On the other hand, $I_{t-1}(\varphi')^{n-m}S$ is an ideal of positive grade and hence again a module of rank $1$.
	It follows that the kernel of the above surjection is a torsion submodule of $\omega_S$, and hence zero because $\omega_S$ is a maximal Cohen-Macaulay module.

	Let $y$ be the image of $\Delta^{n-m}$ in $S$.
	Notice that $y \in \omega_S$ under the isomorphism $\omega_S \cong I_{t-1}(\varphi')^{n-m}S$.
	From \autoref{lem_inclusion_dets} it follows that $\Ann_S(\omega_S/Sy) \supset I_{t-1}(\varphi'')^{n-m}S$, and the latter ideal has positive height.
	Thus, by \autoref{other}, $S$ is $i$-weakly Gorenstein with respect to $\aaa$ when $i = \HT\left(I_{t-1}(\varphi'')S + \aaa \right)-1$.
\end{proof}

We believe that simple combinations of \autoref{thm_duality_Jou_bigrad} and \autoref{thm_det_weak_Gor} can lead to a better understanding of local cohomology modules of determinantal rings in several cases of interest. 
To be more precise, we provide a sample that will be useful later (see \autoref{thm_zero_dim}). 

\begin{proposition}
	\label{prop_det_2_by_2}
	Let $R = \kk[x_1,\ldots,x_d]$ and $T = \kk[y_1,\ldots,y_n]$ be standard graded polynomial rings over a field $\kk$.
	Let $\mm = (x_1,\ldots,x_d) \subset R$ be the irrelevant ideal.
	Let $B$ be the standard bigraded polynomial ring $R \otimes_\kk T$.
	Let $f_1,\ldots,f_n$ be a sequence of forms in $R$ of the same degree $D \ge 1$.
	Let $\varphi$ be the $2 \times n$ matrix given by 
	$$
	\varphi \; = \; \left(
	\begin{array}{ccc}
		y_1 & \cdots & y_n \\
		f_1 & \cdots & f_n
	\end{array}
	\right),
	$$
	and set $S = B/I_2(\varphi)$.
	If $\,\HT(I_2(\varphi)) = n-1$, then there is a bigraded isomorphism of $\,S$-modules
	$$
	\HL^i(S) \, \cong \, {}^*\Ext_T^{i+n-1-d}\left(S, T\right)\left(d-(n-1)D, -1\right)
	$$
	for all $0 \le i \le d-1$.
\end{proposition}
\begin{proof}	
	We first show the result when $\kk$ is infinite.
	
	\begin{claim}\label{claim_dual_det_inf_field}
		The claimed isomorphisms exist when $\kk$ is infinite. 			
	\end{claim}
	To prove the claim 
	notice that $\HT\left(I_1(\varphi'')\right)\ge n$ and $\HT\left(I_1(\varphi'')S + \mm S\right) = \HT\left((y_1,\ldots,y_n)S  + \mm S\right)= \dim(S) = d+1.$ Thus \autoref{thm_det_weak_Gor} implies that $S$ is $d$-weakly Gorenstein with respect to $\mm S$.
	
	After elementary column operations over the infinite field $\kk$, we may assume that $y_1$ is a non zerodivisor on $S$.
	Notice that $I_2(\varphi)$ is a geometric $(n-1)$-residual intersection of the ideal $(y_1,f_1)$ because 
	$$
	I_2(\varphi) \, = \, \left(\Delta_{1,2},\ldots, \Delta_{1,n}\right) : (f_1,y_1),
	$$
	$\HT\left(I_2(\varphi)\right) = n-1$ and $\HT\left(I_2(\varphi) + (f_1,y_1)\right) \ge n$, where $\Delta_{1,i}$ denotes the $2\times 2$ minor corresponding to columns $1$ and $i$ (see \cite{HSCM}).
	From the formulas for the canonical module of residual intersections (\cite{HU2}*{Proposition 2.3}, \cite{KU}), we obtain that $\omega_S \cong (f_1,y_1)^{n-2}S\left(-d + (n-1)D, -1\right)$.
	Let $y \in \omega_S$ be the class of $y_1^{n-2}$, which by \autoref{thm_det_weak_Gor} can be chosen as the element that makes $S$ $d$-weakly Gorenstein with respect to $\mm S$.
	Since $\bideg(y) = (0, n-2) - (-d+(n-1)D,-1) = (d-(n-1)D, n-1)$, \autoref{thm_duality_Jou_bigrad} yields
	$$
	\HL^i(S) \, \cong \, {}^*\Ext_T^{i+n-1-d}\left(S, T\right)\left(d-(n-1)D, -1\right)
	$$
	for all $0 \le i \le d-1$, as required.

	\vspace{.2cm}
	We now settle the proposition for an arbitrary field $\kk$.
	Let $\kk'$ be an infinite field containing $\kk$, and set $S' := S \otimes_\kk \kk'$.
	The computation above of the canonical module shows that $\dim_{\kk'}\big(\left[\omega_{S'}\right]_{(d-(n-1)D, n-1)}\big) = 1$.
	This implies that $\dim_{\kk}\big(\left[\omega_{S}\right]_{(d-(n-1)D, n-1)}\big) = 1$.
	Therefore, take an non-zero element $0 \neq y \in \left[\omega_{S}\right]_{(d-(n-1)D, n-1)}$.
	It necessarily follows that, up to multiplication by a unit in $\kk'$, the extension of $y$ into $\omega_{S'} \cong \omega_S \otimes_\kk \kk'$ equals the element chosen in \autoref{claim_dual_det_inf_field} to make $S'$ weakly Gorenstein over $\mm S'$.
	This completes the proof for the general case.
\end{proof}

\section{Weakly Gorenstein symmetric algebras}
\label{sect_weak_Gor_Sym}

The principal motivation of this section is to find a comprehensive family of ideals whose symmetric algebra has the weakly Gorenstein property. 
This first goal will allow us to apply our generalized Jouanolou duality (\autoref{thm_duality_Jou_bigrad}) in several new cases. 
However, along the way, we achieve  more: 
\begin{enumerate}[(i)]
	\item we compute explicitly the canonical module of a large family of symmetric algebras,
	\item when the approximation complex  $\mathcal{Z}_\bullet$ is acyclic, we define a new complex that also resolves the symmetric algebra but has the advantage of having free modules in the last $g-1$ positions, where $g$ is the grade of the ideal, 
	\item when the ideal is an almost complete intersection or  perfect of deviation two, we give an explicit free resolution for the symmetric algebra, provided $\mathcal{Z}_\bullet$ is acyclic. 
\end{enumerate}

Modules of Koszul cycles were used before in \cite{HSV3} to obtain information
about the free resolution of the symmetric algebra of an ideal $I \subset R$ and
about $\, \omega_{\Sym(I)} \otimes_{\Sym(I)} R$.

\smallskip

Throughout this section we assume the following setup.

\begin{setup}
	\label{setup_Sym_weak_Gor}
	Let $(R, \mm)$ be a $d$-dimensional Cohen-Macaulay local ring.
	Let $I = (f_1,\ldots,f_n) \subsetneq R$ be an ideal minimally generated by $n$ elements, and set $g = \HT(I) \ge 1$ to be the height of $I$.
	Let $B = R[y_1,\ldots,y_n]$ be a standard graded polynomial ring.
	There is a surjection of standard graded $R$-algebras $B \twoheadrightarrow \Sym(I)$ mapping $y_i$ to $f_i \in [\Sym(I)]_1.$  Let  $K_\bullet$
	%$$
	%K_\bullet := K(f_1,\ldots,f_n; R) : 0 \rightarrow \wedge^n R^n \rightarrow \cdots \rightarrow \wedge^1R^n \rightarrow \wedge^0R^n \rightarrow 0
	%$$ 
	be the Koszul complex associated to the sequence $f_1,\ldots, f_n$.
	Let $Z_\bullet$ and $H_\bullet$ be the cycles and homologies of $K_\bullet$, respectively.
\end{setup}

%Our work in this part was inspired by \cite{HSV3}, where %$Z_{\bullet}$ was used to provide information about the minimal %free resolution of the symmetric algebra $\Sym (I)$ and about 
%$\omega_{\Sym(I)} \otimes R$.

\begin{definition}
	Let $k \geq 0$ be an integer.
	\begin{enumerate}
		\item [(G)] $I$ satisfies the \emph{condition $G_k$} if 
		$\mu(I_\mathfrak{p}) \le \HT(\pp)$ for all $\pp \in V(I) \subset \Spec(R)$ with $\HT(\pp) \le k-1$.
		\item [(F)] 
		$I$ satisfies the \emph{condition $F_0$} if 
		$\mu(I_\mathfrak{p}) \le \HT(\pp) +1$ for all $\pp\in \Spec(R)$.
		\item [(SD)] $I$ has the \emph{sliding depth condition $\SD_k$} if $\depth(H_i) \ge 
		\min\{d-g, d-n+i+k\}$ for all $i$.
		\item [(SCM)] $I$ is \emph{strongly Cohen-Macaulay} if $H_i$ is a Cohen-Macaulay module for all $i$.		
	\end{enumerate}
\end{definition}

One can check that the conditions strongly Cohen-Macaulay and  $\SD_k$ yield the following inequalities
$$
\depth(Z_i) \, \ge \,  
\begin{cases}
	\min\{d, d-g+2\} & \text{ if $I$ is strongly Cohen-Macaulay}\\
	\min\{d, d-n+i+k+1\} & \text{ if $I$ satisfies $\SD_k$}
\end{cases}
$$
for all $i \ge 0$ (see \cite{HVV}*{p.161}).
Recall that height $2$ perfect ideals and height $3$ Gorenstein ideals (perfect ideals with last Betti number $1$) are strongly Cohen-Macaulay (see \cite{HSCM}*{Proposition 0.3}, \cites{APERY,Gaeta,PS_LINKAGE}, \cite{W_GOR3}).

Let $\mathcal{Z}_\bullet$ be the approximation complex
$$
\mathcal{Z}_\bullet: \quad 0 \rightarrow Z_{n-1} \otimes_R B(-n+1) \rightarrow \cdots \rightarrow Z_i \otimes_R B(-i) \rightarrow \cdots \rightarrow Z_1 \otimes_R B(-1) \rightarrow Z_0 \otimes_R B
$$
corresponding to the sequence $f_1,\ldots,f_n,$ which is a complex of graded $B$-modules.
For more details regarding approximation complexes the reader is referred to \cite{HSV}.
For the sake of completeness, we include a couple of well-known results regarding symmetric algebras. 

\begin{lemma}
	\label{lem_ht_LL_with_F0}
	If $\,I \subset R$ has $F_0$ on the punctured spectrum of $R$, then $\dim(\Sym(I))  = \max\left\{n, d+1\right\}$.
\end{lemma}
\begin{proof}
	By the Huneke-Rossi formula (see \cite[Theorem 2.6]{HR}, \cite[Theorem 1.2.1]{V}), we have 
	$$
	\dim(\Sym(I)) = \sup\left\{\mu(I_\pp) + \dim(R/\pp) \mid \pp \in \Spec(R)\right\}.
	$$
	Since $I$ has $F_0$ on the punctured spectrum of $R$, we deduce  $\dim(\Sym(I)) = \max\left\{\mu(I), \dim(R)+1\right\} = \max\left\{n, d+1\right\}$.
\end{proof}

\begin{lemma}
	\label{lem_Sym_CM}
	If $\,I \subset R$ satisfies the conditions $\SD_0$ and  $F_0$, then the following statements hold$\, : $ 
	\begin{enumerate}[\rm (i)]
		\item $\Sym(I)$ is Cohen-Macaulay. 
		\item The approximation complex $\mathcal{Z}_\bullet$ is acyclic.
	\end{enumerate}
\end{lemma}
\begin{proof}
	See for instance \cite[Theorem 5.4, Theorem 10.1]{HSV}.
\end{proof}

We now construct a new complex that will allow us to compute the canonical module of $\Sym(I)$ in many cases. 
%It seems that 
Apparently, the existence of this quite useful complex has been unnoticed. Its construction was inspired to us by \cite[Theorem 5.8]{HSV3}.

\begin{proposition}
	\label{prop_modified_Z_complex}
	Let $\mathfrak{D}_{p,q}$ be the double complex 
	\begin{equation*}		
		\begin{tikzpicture}[baseline=(current  bounding  box.center)]
			\matrix (m) [matrix of math nodes,row sep=2.5em,column sep=2.3em,minimum width=2em, text height=1.8ex, text depth=0.25ex]
			{
				0 & 0 & 0 &  \wedge^nR^n \otimes_R B(-n+g-1)   \\
				\vdots & \vdots & \vdots & \vdots \\
				0 & \wedge^nR^n \otimes_RB(-n+2) & \cdots & \wedge^{n-g+3}R^n \otimes_RB(-n+g-1) \\
				\wedge^nR^n \otimes_RB(-n+1) & \wedge^{n-1}R^n \otimes_RB(-n+2) & \cdots & \wedge^{n-g+2}R^n \otimes_RB(-n+g-1) \\
			};						
			\path[-stealth]
			(m-4-1) edge node [above] {$\partial_\mathbf{y}$} (m-4-2)
			(m-4-2) edge node [above] {$\partial_\mathbf{y}$} (m-4-3)
			(m-4-3) edge node [above] {$\partial_\mathbf{y}$} (m-4-4)
			(m-3-3) edge node [above] {$\partial_\mathbf{y}$} (m-3-4)
			(m-3-2) edge node [above] {$\partial_\mathbf{y}$} (m-3-3)
			(m-1-4) edge node [right] {$\partial_\mathbf{f}$} (m-2-4)
			(m-2-4) edge node [right] {$\partial_\mathbf{f}$} (m-3-4)
			(m-3-4) edge node [right] {$\partial_\mathbf{f}$} (m-4-4)
			(m-3-2) edge node [right] {$\partial_\mathbf{f}$} (m-4-2)
			;
		\end{tikzpicture}	
	\end{equation*}
	where $\mathfrak{D}_{p,q} = \wedge^{n-g+2+p+q}R^n\otimes_R B(-n+g-1-p)$, and $\partial_\mathbf{f}$ and $\partial_\mathbf{y}$ are the differentials of the Koszul complexes $K(f_1,\ldots,f_n; B)$ and $K(y_1,\ldots,y_n; B)$, respectively.
	Let $\mathfrak{T}_\bullet := {\rm Tot}(\mathfrak{D}_{\bullet, \bullet})$ be the total complex 
	$$
	\mathfrak{T}_\bullet:  \quad 0 \;\rightarrow\; \bigoplus_{j=1}^{g-1} \left(\wedge^nR^n \otimes_R B(-n+g-j)\right) \;\rightarrow\; \cdots \;\rightarrow\; \wedge^{n-g+2}R^n \otimes_RB(-n+g-1) \;\rightarrow\; 0
	$$
	of $\mathfrak{D}_{\bullet, \bullet}$.
	We can splice together $\mathfrak{T}_\bullet$ and the truncated complex $0 \rightarrow \mathcal{Z}_{n-g} \rightarrow \cdots \rightarrow \mathcal{Z}_0 \rightarrow 0$ to obtain a complex of graded $B$-modules
	$$
	\mathfrak{L}_\bullet: \quad 0 \rightarrow \mathfrak{T}_{g-2} \rightarrow \cdots \rightarrow \mathfrak{T}_{1} \rightarrow \mathfrak{T}_{0} \rightarrow \mathcal{Z}_{n-g} \rightarrow \cdots \rightarrow \mathcal{Z}_0 \rightarrow 0.
	$$
	The complex $\mathfrak{L}_\bullet$ is acyclic if and only if $\, \mathcal{Z}_\bullet$ is.
\end{proposition}
\begin{proof}
	Since the Koszul complex is depth sensitive and $g = \grade(I)$, it follows that $H_i = 0$ for all $i \ge n-g+1$.
	Hence, after computing homology by columns in the double complex $\mathfrak{D}_{\bullet,\bullet}$ we  get that the only non-zero row is the bottom one, which is then given by  
	$$
	0 \rightarrow \mathcal{Z}_{n-1} \rightarrow \cdots \rightarrow \mathcal{Z}_{n-g+2} \xrightarrow{\Psi} \mathcal{Z}_{n-g+1} \rightarrow 0.
	$$
	This implies that $\HH_i(\mathfrak{T}_\bullet) = \HH_{n-g+1+i}(\mathcal{Z}_\bullet)$ for all $i \ge 1$ and that $\HH_0(\mathfrak{T}_\bullet) = \Coker(\Psi)$. 
	Since $\mathcal{Z}_\bullet$ is a complex, there is a natural map from $\Coker(\Psi)$ to $\Ker(\Phi)$, where $\mathcal{Z}_{n-g} \xrightarrow{\Phi} \mathcal{Z}_{n-g-1} \rightarrow \cdots \rightarrow \mathcal{Z}_0$.
	Therefore, we can always make the claimed splicing, and $\mathfrak{L}_\bullet$ is acyclic 
	if and only if $\mathcal{Z}_\bullet$ is.
\end{proof}

The following theorem contains the main results of this section.

\begin{theorem}
	\label{thm_Sym_weak_Gor}
	Assume \autoref{setup_Sym_weak_Gor} where $R$  is Gorenstein with infinite residue field. 
	Suppose that the ideal $I \subset R$ has height $g = \HT(I) \ge 2$ and satisfies the conditions $F_0$ and $\SD_1$. After changing the generators $f_1, \ldots, f_n$ of $\, I$ the element $y_1$ is a non zerodivisor on $\Sym(I),$ and
	the following statements hold$\,:$ 
	\begin{enumerate}[\rm (i)]
		\item $\omega_{\Sym(I)} \cong \left(f_1,y_1\right)^{g-2}\Sym(I)(-1).$
		\item $\Sym(I)$ is $d$-weakly Gorenstein with respect to the ideal $\mm \, \Sym(I),$ and the image of $y_1^{g-2}$ in $\omega_{\Sym(I)}$ is a weak
		generator of the canonical module of degree $g-1.$
	\end{enumerate} 
\end{theorem}
\begin{proof}
	The condition $F_0$ implies, in particular, that $n \le d+1$.
	From \autoref{lem_ht_LL_with_F0} and \autoref{lem_Sym_CM} we have that $\Sym(I)$ is a Cohen-Macaulay ring of dimension $d+1$ and that the  approximation complex $\mathcal{Z}_\bullet$ is acyclic. 
	As $\dim(S/(y_1,\ldots,y_n)S) = \dim(R) = d,$ it then follows that $\grade\left((y_1,\ldots,y_n)S\right) \ge 1$.
	So, after possibly changing the generators of $I$, we can assume that $y_1$ is a non zerodivisor on $S,$ as asserted.
	
	The condition $\SD_1$ yields the following lower bounds for the depth of the Koszul cycles 
	$$
	\depth(Z_i) \, \ge \, \min\left\{d, d-n+i+2\right\}  \quad \text{for all $i$}.
	$$
	This implies that $\depth(\mathcal{Z}_i) \ge  \min\left\{d+n, d+i+2\right\} $.
	We consider the complex of graded $B$-modules
	$$
	\mathfrak{L}_\bullet: \quad 0 \rightarrow \mathfrak{T}_{g-2} \rightarrow \cdots \rightarrow \mathfrak{T}_{1} \rightarrow \mathfrak{T}_{0} \rightarrow \mathcal{Z}_{n-g} \rightarrow \cdots \rightarrow \mathcal{Z}_0 \rightarrow 0
	$$
	constructed in \autoref{prop_modified_Z_complex}, which is acyclic by the current assumptions.
	To simplify notation, set $S:=\Sym(I)$.
	
	Denote by $(\mathfrak{L}_\bullet, \partial_\bullet)$ the complex $\mathfrak{L}_\bullet$ with its differential. 
	Let $U_i = \Coker(\partial_{i+1})$.
	For all $0 \le i \le n-g$, we have the short exact sequence 
	$$
	0 \rightarrow U_{i+1} \rightarrow \mathcal{Z}_i \rightarrow U_i \rightarrow 0
	$$
	that gives the induced exact sequence 
	$$
	\Ext_B^{n-1-i-1}(U_{i+1}, \omega_B) \rightarrow \Ext_B^{n-1-i}(U_i, \omega_B) \rightarrow \Ext_B^{n-1-i}(\mathcal{Z}_i, \omega_B).
	$$
	For $0 \le i \le n-g$, we have $\Ext_B^{n-1-i}(\mathcal{Z}_i, \omega_B) = 0$ since $n-i-1 > \max\{0,n-i-2\} \ge d+n - \depth(\mathcal{Z}_i)$.
	Thus we obtain the surjection 
	$$
	\Ext_B^{n-1-i-1}(U_{i+1}, \omega_B) \surjects \Ext_B^{n-1-i}(U_i, \omega_B)
	$$
	for all $0 \le i \le n-g$.
	By composing these surjections, we obtain a surjective homomorphism of graded $B$-modules
	$$
	\Ext_B^{g-2}(U_{n-g+1}, \omega_B) \surjects \Ext_B^{n-1}(U_0, \omega_B) = \Ext_B^{n-1}(S, \omega_B) \cong \omega_{S}.
	$$
	The complex $0 \rightarrow \mathfrak{T}_{g-2} \rightarrow \cdots \rightarrow \mathfrak{T}_{1} \rightarrow \mathfrak{T}_{0} \rightarrow 0$ gives a homogeneous free $B$-resolution of $U_{n-g+1}$.

	If $g = 2$, then $U_{n-g+1}$ is equal to the free $B$-module $B(-n+1)$.
	In this case we have a surjection $\Hom_B(U_{n-g+1}, \omega_B) \surjects  \omega_S$, hence $\omega_S$ is generated by one homogeneous element of degree $1;$ this means that $\omega_S \cong S(-1)$ and, in particular, $S$ is Gorenstein.
	Hence the proof is complete for the case $g = 2$.
	
	Assume that $g \ge 3$. We are going to show that 
	$\Ext_B^{g-2}(U_{n-g+1}, \omega_B),$ after tensoring with $S$ and 
	factoring out $S$-torsion, maps isomorphically to both $\omega_S$ and 
	$(f_1,y_1)^{g-2}S(-1),$ which will finish the proof of (i).

	The module $\Ext_B^{g-2}(U_{n-g+1}, \omega_B)$ is the cokernel of the map 
	$$
	\Hom_B(\mathfrak{T}_{g-3}, \omega_B) \rightarrow \Hom_B(\mathfrak{T}_{g-2}, \omega_B).
	$$
	The map $\mathfrak{T}_{g-2} \rightarrow \mathfrak{T}_{g-3}$ is given explicitly as
	$$
	\bigoplus_{j=1}^{g-1}B(-n+g-j) \xrightarrow{\;\;  \;\;} \bigoplus_{j=1}^{g-2}B^n(-n+g-j),
	$$
	and its matrix representation with respect to the standard bases is 
	$$
	\mathfrak{A} \;=\;	\left[
	\begin{array}{cccccc}
		f_1 & y_1 & 0 & \cdots & 0 & 0\\
		\vdots & \vdots & \vdots & \vdots & \vdots & \vdots\\
		(-1)^{n-1}f_{n} & (-1)^{n-1}y_{n} & 0 & \cdots & 0 & 0\\
		\hline
		0 & -f_1 & y_1 & \cdots & 0 & 0\\
		\vdots & \vdots & \vdots & \vdots & \vdots & \vdots\\
		0 & (-1)^{n}f_{n} & (-1)^{n-1}y_{n} & \cdots & 0 & 0\\			
		\hline
		\vdots & \vdots & \vdots & \vdots & \vdots & \vdots\\
		\hline
		0 & 0 & 0 & \cdots & (-1)^{g-3}f_1 & y_1\\
		\vdots & \vdots & \vdots & \vdots & \vdots & \vdots\\
		0 & 0 & 0 & \cdots & (-1)^{n+g-4}f_{n} & (-1)^{n-1}y_{n}\\	 		
	\end{array}		
	\right] \ .
	$$
	Let $\mathfrak{C} = \mathfrak{A}^T$ and 
	$$
	\Omega \; := \; \Coker\left(\bigoplus_{j=1}^{g-2} \omega_B^n(n-g+j)  \xrightarrow{\;\; \mathfrak{C} \;\;} \bigoplus_{j=1}^{g-1}\omega_B(n-g+j)\right). 
	$$
	Denote by $w_j$ the generator of $\Omega$ corresponding to the $j$-th column of $\mathfrak{C},$ and write $^{-}$ for images in $S$. Since in $S$ we have the relations
	$\overline{f_iy_j} = \overline{f_jy_i},$  we obtain a surjective homomorphism
	of graded $B$-modules
	$$
	\Omega \twoheadrightarrow (f_1,y_1)^{g-2}S(-1), \;\;\text{ where } \;\; w_j \mapsto (-1)^{j \choose 2}\, \overline{f_1^{j-1}y_1^{g-j-1}}.
	$$
	Thus we have surjective homomorphisms
	of graded $S$-modules
	$$
	\Omega \otimes_B S \twoheadrightarrow (f_1,y_1)^{g-2}S(-1)
	\qquad \text{ and } \qquad
	\Omega \otimes_B S \, \twoheadrightarrow \, \omega_S.
	$$
	
	%Since $\dim(S/(y_1,\ldots,y_n)S) = \dim(R) = d$ and $S$ is a Cohen-Macaulay ring of dimension $d+1$, it follows that $\grade\left((y_1,\ldots,y_n)S\right) \ge 1$.
	%So, after possibly changing the generators of $I$, we can assume that $\overline{y_1}$ is a non zerodivisor on $S$.
	The matrix $\mathfrak{A}$ has size $(g-2)n \times (g-1)$, and $g \geq 3.$ So by the pigeonhole principle, every $(g-1)\times (g-1)$ submatrix of $\mathfrak{A}$ contains, after possibly multiplying one column by $-1,$ a $2\times (g-1)$ submatrix of the form 
	$$
	\left(
	\begin{array}{cccccccc}
		0 & \cdots & 0 & (-1)^{i_1-1}f_{i_1} & (-1)^{i_1-1}y_{i_1} & 0 & \cdots & 0\\
		0 & \cdots & 0 & (-1)^{i_2-1}f_{i_2} & (-1)^{i_2-1}y_{i_2} & 0 & \cdots & 0
	\end{array}
	\right).
	$$
	Accordingly, $I_{g-1}(\mathfrak{C}) \subset \left(\left\{f_iy_j - f_jy_i\right\}_{i,j}\right) \subset B$, hence $I_{g-1}(\mathfrak{C} \otimes_B S)=0$.
	Notice that 
	$$\overline{y_1^{g-2}} \in I_{g-2}(\mathfrak{C} \otimes_B S)
	$$ 
	and $\overline{y_1^{g-2}}$ is a non zerodivisor on $S$, therefore $\rank_S(\mathfrak{C} \otimes_B S) = g-2$ and $\rank_S(\Omega \otimes_B S) = 1$.
	The canonical module $\omega_S$ is faithful and the epimorphic image of a module of rank $1$, hence $\omega_S$ is an $S$-module of rank $1$.
	The ideal $(f_1,y_1)^{g-2}S$ contains the non zerodivisor $\overline{y_1^{g-2}}$, and so it is a torsion-free $S$-module of rank $1$.
	Finally, since $\omega_S$ is always torsion-free, we obtain the isomorphisms
	$$
	\omega_S \; \cong \; \frac{\Omega \otimes_B S}{\text{tor}\left(\Omega \otimes_B S\right)} \; \cong \; (f_1,y_1)^{g-2}S(-1).
	$$
	This concludes the proof of part (i) of the theorem.
	
	We now concentrate on part (ii) of the theorem, which becomes straightforward after having computed the canonical module $\omega_S$.
	
	Since $(\overline{y_1},\ldots,\overline{y_n}) \cdot \overline{f_1}  \subset (\overline{y_1})$ in $S$, it follows that 
	$$
	(\overline{y_1},\ldots,\overline{y_n})^{g-2} \,\subset \, \Ann_S\left((\overline{f_1},\overline{y_1})^{g-2}/(\overline{y_1})^{g-2}\right) \,=\,  \Ann_S\left(\omega_S/Su_1\right),
	$$
	where $u_1$ denotes the image in $\omega_S$ of the element $w_1 \in \Omega.$
	As a consequence, we get 
	$$
	\HT\left(\Ann_S\left(\omega_S/Su_1 \right)\right) \ge 1 \qquad \text{ and } \qquad \HT\left(\Ann_S\left(\omega_S/Su_1 \right) + \mm S\right) \ge d+1.
	$$ 
	Therefore, $S$ is $d$-weakly Gorenstein with respect to $\mm S$, and this establishes the remaining part (ii).
\end{proof}

\begin{remark}
	It should be mentioned that the proof of \autoref{thm_Sym_weak_Gor} works under the weaker assumption that $R$ is a Cohen-Macaulay ring with canonical module $\omega_R$.
	In that case, we have the formula 
	$$
	\omega_{\Sym(I)} \,\cong\, \omega_R\left(f_1,y_1\right)^{g-2}\Sym(I)(-1)
	$$
	when $g = \HT(I) \ge 2$ and $I \subset R$ satisfies $F_0$ and $\SD_1$.
\end{remark}

\begin{remark}
	\label{rem_no_weak_Gor_nnn}
	In addition to the hypotheses of \autoref{thm_Sym_weak_Gor}
	assume that $\mu(I)=d+1.$ An interesting question is whether $\Sym(I)$ is weakly Gorenstein with respect to $\nnn$, where $\nnn = (y_1,\ldots,y_n) \subset B$.
	The natural choice for making $S:=\Sym(I)$ weakly Gorenstein with
	respect to $\nnn$ is to choose the element $f_1^{g-2}$ instead of $y_1^{g-2}$ in the proof of \autoref{thm_Sym_weak_Gor}.
	However, this choice does not work. 
	Indeed, notice that $\HT(\mm S) = 0$ and $y_1^{g-2}$ is a non zerodivisor on $S$, so the image of $f_1^{g-2}$ cannot generate $(f_1,y_1)^{g-2}S_\pp = S_\pp$ for all $\pp \in \Ass(S)$.
	The same argument shows that no element in $\mm S$ would work.
	On the other hand, in \autoref{thm_Jou_dual_Sym} we do consider the local cohomology modules $\HH_\nnn^i(S)$ for $i \ge 2$ by using further duality results.
\end{remark}

\begin{remark}
	Our computation of the canonical module of $\Sym(I)$ extends or complements known results in the literature.
	In \cite[Theorem 5.7.8]{V}, by utilizing computations with divisor class groups, the formula of \autoref{thm_Sym_weak_Gor}(i) was obtained. 
	However, the result of \cite[Theorem 5.7.8]{V} further requires $\Sym(I)$ to be a normal domain (which is a strong condition; in fact, in our applications $\Sym(I)$ will typically not be a domain).
	In \cite[Corollaries 2.5 and 2.9]{HSV2}, a formula for the canonical module $\omega_{\Rees(I)}$ of $\Rees(I)$ is obtained 
	assuming that $I$ is generically a complete intersection, $\Rees(I)$ is Cohen-Macaulay, and $\gr_I(R)$ is Gorenstein; that formula has the same form as our formula for $\omega_{\Sym(I)}$.
	We exploit this connection between $\omega_{\Sym(I)}$ and $\omega_{\Rees(I)}$ in the next corollary to deduce that $I$ is of linear type. Recall that an ideal
	$I$ is said to be {\it of linear type} if the natural map $\Sym(I) \twoheadrightarrow \Rees(I)$ is an isomorphism. A version of \autoref{cor_beautiful} is proved in \cite[Theorem 3.8]{SUV}, with different 
	methods and the hypothesis that $I$ be strongly Cohen-Macaulay as opposed to $F_0$ and $SD_1.$

\end{remark}

\begin{corollary}
	\label{cor_beautiful}
	In addition to \autoref{setup_Sym_weak_Gor} 
	assume that$\, :$
	\begin{enumerate}[\rm (a)]
		\item $I \subset R$ satisfies  the conditions $F_0$ and $\SD_1$, and  $I$ is generically a complete intersection $($a complete intersection locally at each of its associated primes$)$ with $g = \HT(I) \ge 2.$
		\item $\gr_I(R)$ is a Gorenstein ring.
	\end{enumerate} 
	Then $I$ is of linear type and strongly Cohen-Macaulay.
\end{corollary}
\begin{proof} Notice that $R$ is Gorenstein because $\gr_I(R)$ is (see, e.g., \cite[proof of Proposition 11.16]{HIO}).
	%To prove that $I$ is of linear type, we may assume by induction on $d$ that $I$ is of linear type locally on the punctured spectrum (the condition $\SD_1$ localizes, see \cite[p.676]{HSV3}). 
	We first prove that $I$ is of linear type.
	By  \autoref{lem_ht_LL_with_F0},
	we have $\dim\left(\Sym(I)\right) = d+1 = \dim\left(\Rees(I)\right)$.
	After dualizing, the surjection $\Sym(I) \twoheadrightarrow \Rees(I)$ yields the natural inclusion 
	$$
	\varphi:	\omega_{\Rees(I)} \,\hookrightarrow \, \omega_{\Sym(I)}.
	$$   	
	Since $I$ is generically a complete intersection with $g \geq 1$ and $\gr_{I}(R)$ is Gorenstein, \cite[Theorem 2.1]{SUV} shows that $\Rees(I)$ is Cohen-Macaulay. 
	With these hypotheses it follows from
	\cite[Corollaries 2.5 and 2.9]{HSV2} that $\omega_{\Rees(I)} \cong (f_1,y_1)^{g-2}\Rees(I)(-1).$ Combining this with \autoref{thm_Sym_weak_Gor}(i) we obtain a natural homomorphism of graded $\Sym(I)$-modules
	$$\psi:		\omega_{\Sym(I)} \cong (f_1,y_1)^{g-2}\Sym(I)(-1) \;\twoheadrightarrow\; (f_1,y_1)^{g-2}\Rees(I)(-1) \cong \omega_{\Rees(I)}.
	$$
	Since $g \geq 2,$ the ideal $I$ is of linear type locally in codimension $1,$ and so
	both $\varphi$ and $\psi$ are isomorphisms locally at every prime ideal of $R$ of height $1.$
	%on the punctured spectrum of $R.$
	%Since $I$ is of linear type locally on the punctured spectrum of $R$, both $\varphi$ and $\psi$ are isomorphisms locally on the punctured spectrum of $R.$
	
	We claim that the composition $\psi \circ \varphi$ is an isomorphism. Since $\Rees(I)$ is Cohen-Macaulay, $\End(\omega_{\Rees(I)})$ is naturally isomorphic to 
	$\Rees(I). $ As $\psi \circ \varphi \in
	\left[\End(\omega_{\Rees(I)})\right]_0,$ it follows that $\psi \circ \varphi$ is multiplication by an element $a \in \left[\Rees(I)\right]_0=R$. 
	But $\psi \circ \varphi$ is an isomorphism locally at every prime ideal of $R$ of 
	height $1,$ so the element $a$ is a unit in $R$ locally at every such prime. 
	%punctured spectrum of $R$. Since $d\geq 2,$ 
	%By our inductive assumption, $a$ is a unit in $R$ locally on the punctured spectrum of 
	%$R$. Since $d\geq 2,$ 
	This can only happen if $a$ is a unit, which shows that $\psi \circ \varphi$ is an isomorphism. 
	
	Thus $\omega_{\Rees(I)}$ is isomorphic to a direct summand of $\omega_{\Sym(I)}$. Recall that $\Sym(I)$ is Cohen-Macaulay by  \autoref{lem_Sym_CM}(i). Therefore $\Spec(\Sym(I))$ is connected in codimension one by Hartshorne's connectedness theorem (see, e.g., \cite[Theorem 18.12]{E}), and so $\omega_{\Sym(I)}$
	is indecomposable (see, e.g., \cite[Theorem 3.6]{HH}). Thus the $\Sym(I)$-modules $\omega_{\Sym(I)}$ and $\omega_{\Rees(I)}$ are isomorphic and hence have the same annihilators. As canonical modules of unmixed rings are faithful, we conclude that the natural
	surjection $\Sym(I) \twoheadrightarrow \Rees (I)$ is an isomorphism, as asserted.
	
	Now $I$ satisfies $SD_0$, $I$ is of linear type and generically a complete intersection, and $\gr_I(R)$ is Gorenstein. With
	these hypotheses \cite[Corollary 3.11]{SUV} implies that $I$ is also strongly Cohen-Macaulay.
\end{proof}

Finally, we give an explicit free resolution for the symmetric algebra when the ideal is an almost complete intersection or a perfect ideal of deviation two, both under the assumption that   $\mathcal{Z}_\bullet$ is acyclic.
The \emph{deviation} of $I \subset R$ is defined as $d(I) := \mu(I) - \HT(I) = n-g$; one says that $I$ is an almost complete intersection when $d(I) \leq 1$.

\begin{theorem}
	\label{thm_res_Sym}
	Assume \autoref{setup_Sym_weak_Gor}.
	Suppose that the approximation complex $\mathcal{Z}_\bullet$ is acyclic and let 
	$$
	F_\bullet : \;\; \cdots \rightarrow F_2 \rightarrow R^{n} \xrightarrow{(f_1,\ldots,f_n)} R \rightarrow 0
	$$ 
	be a free $R$-resolution of $\,R/I$.
	Let $\mathfrak{T}_\bullet$ be the acyclic complex of free $B$-module defined in \autoref{prop_modified_Z_complex}.
	The following statements hold$\,:$
	\begin{enumerate}[\rm (i)]
		\item If $I \subset R$ is an almost complete intersection, then a homogeneous free $B$-resolution of $\, \Sym(I)$ is given by 
		$$
		\mathfrak{F}_\bullet: \;\; \cdots \rightarrow \mathfrak{F}_i \rightarrow  \cdots \rightarrow   \mathfrak{F}_1 \rightarrow \mathfrak{F}_0 \rightarrow 0
		$$
		where $\mathfrak{F}_0 = B$, $\mathfrak{F}_1 = F_2 \otimes_{R} B(-1)$, and $\mathfrak{F}_i = \mathfrak{T}_{i-2} \,\oplus\, \left(F_{i+1}\otimes_{R} B(-1)\right)$  \,for $i \ge 2.$
		\smallskip
		\item If $I \subset R$ is a perfect ideal of deviation two and $F_\bullet : 0 \rightarrow F_g \rightarrow \cdots \rightarrow F_2 \rightarrow R^{n} \xrightarrow{(f_1,\ldots,f_n)} R \rightarrow 0$ is a free $R$-resolution of $R/I$, then a homogeneous free $B$-resolution of $\, \Sym(I)$ is given by 
		$$
		\mathfrak{F}_\bullet: \;\; \cdots \rightarrow \mathfrak{F}_i \rightarrow  \cdots \rightarrow   \mathfrak{F}_1 \rightarrow \mathfrak{F}_0 \rightarrow 0
		$$
		where $\mathfrak{F}_0 = B$, $\mathfrak{F}_1 = F_2 \otimes_{R} B(-1)$, and $\mathfrak{F}_i = \mathfrak{T}_{i-3} \oplus \big(\big(F_{i+1} \oplus K_{i+1} \oplus F_{g-i+2}^*\big) \otimes_R B(-1)\big)$  \,for $i \ge 2$.
	\end{enumerate}
\end{theorem}
\begin{proof}
	Let $\LL$ be the kernel of the natural surjection $B \surjects \Sym(I)$. 
	
	(i)
	From \autoref{prop_modified_Z_complex} we have the acyclic complex of graded
	$B$-modules
	$$
	\mathfrak{L}_\bullet: \quad 0 \rightarrow \mathfrak{T}_{g-2} \rightarrow \cdots \rightarrow \mathfrak{T}_{1} \xrightarrow{\Psi} \mathfrak{T}_{0} \rightarrow \mathcal{Z}_{1}  \rightarrow B \rightarrow 0
	$$
	that resolves $\Sym(I)$.
	The map $\Coker(\Psi) \rightarrow \mathcal{Z}_1$ induces a morphism of complexes 
	of graded $B$-modules $u_{\bullet} : \mathfrak{T}_\bullet \rightarrow G_\bullet$, where
	$$
	G_\bullet : \quad \cdots \rightarrow F_3 \otimes_{R} B(-1) \rightarrow F_2 \otimes_{R} B(-1) \rightarrow 0
	$$ 
	is the homogeneous resolution of $\mathcal{Z}_1 = Z_1 \otimes_R B(-1)$ obtained by truncating $F_\bullet \otimes_{R} B(-1)$.
	Since we have a short exact sequence $0 \rightarrow \Coker(\Psi) \rightarrow \mathcal{Z}_1 \rightarrow \LL \rightarrow 0$, the mapping cone $C(u_{\bullet})$ yields a homogeneous free $B$-resolution of $\LL$.
	So, the proof of this part is complete.
	
	(ii) The complex of \autoref{prop_modified_Z_complex} is now given by
	$$
	\mathfrak{L}_\bullet: \quad 0 \rightarrow \mathfrak{T}_{g-2} \rightarrow \cdots \rightarrow \mathfrak{T}_{1} \xrightarrow{\Psi} \mathfrak{T}_{0} \rightarrow \mathcal{Z}_{2} \xrightarrow{\Phi} \mathcal{Z}_{1} \rightarrow B \rightarrow 0.
	$$
	Notice that there is a short exact sequence $0 \rightarrow \text{B}_2(K_\bullet) \rightarrow \text{Z}_2(K_\bullet) \rightarrow \text{H}_2(K_\bullet) \rightarrow 0$, the truncated Koszul complex $0 \rightarrow K_n \rightarrow \cdots \rightarrow K_3 \rightarrow 0$ is a free $R$-resolution of $\text{B}_2(K_\bullet)$, and $\Hom_R(F_\bullet,R)[-g]$ gives a free $R$-resolution of $\text{H}_2(K_\bullet) \cong \omega_{R/I}$ as $R/I$ is a perfect $R$-module.
	By the Horseshoe lemma, a homogeneous free $B$-resolution of $\mathcal{Z}_2 = Z_2 \otimes_{R} B(-1)$ is given by a complex $P_\bullet$ with $P_i = \big(K_{i+3} \oplus F_{g-i}^*\big) \otimes_{R} B(-1)$.
	As in part (i), a mapping cone construction along the short exact sequence $0 \rightarrow \Coker(\Psi) \rightarrow \mathcal{Z}_2 \rightarrow \IM(\Phi) \rightarrow 0$ yields a homogeneous free $B$-resolution $Q_\bullet$ of $\IM(\Phi)$ with $Q_i = P_i \oplus \mathfrak{T}_{i-1}$.
	Recall that the complex $G_\bullet$ of part (i) is a homogeneous free 
	B-resolution of $\mathcal{Z}_1 = Z_1 \otimes_R B(-1)$. 
	Therefore, another mapping cone construction along the short exact sequence $0 \rightarrow \IM(\Phi) \rightarrow \mathcal{Z}_1 \rightarrow \LL \rightarrow 0$ gives a homogeneous free $B$-resolution $L_\bullet$ of $\LL$ with $L_0 = F_2 \otimes_{R} B(-1)$ and
	$$
	L_i \,=\, Q_{i-1} \oplus \big(F_{i+2} \otimes_{R}  B(-1)\big)  \,=\, \mathfrak{T}_{i-2} \,\oplus\, \big(\big(F_{i+2} \oplus K_{i+2} \oplus F_{g-i+1}^*\big) \otimes_R B(-1)\big)
	$$
	for all $i \ge 1$.
	The resolution $\mathfrak{F}_\bullet$ is now given by setting $\mathfrak{F}_0 = B$ and $\mathfrak{F}_{i} = L_{i-1}$ for all $i \ge 1$.
\end{proof}

\begin{remark} The proof above also gives a description of the differentials in the
	resolutions $\mathfrak{F}_\bullet$ of \autoref{thm_res_Sym}.
\end{remark}

\section{A general framework of dualities to study blowup algebras}
\label{sect_dualities_Sym}

In this section, we apply our generalization of Jouanolou duality to study the defining equations of several interesting classes of Rees algebras. 
Determining the defining equations of Rees algebras is a problem of utmost importance with applications in Algebraic Geometry, Commutative Algebra and applied areas like Geometric Modeling (see \cites{KPU3,KPU4,KPU5,Buse,KM,CBD14,V_REES_EQ,COX_MOV_CURV,COX_HOFFMAN_WANG,HONG_SV,KPU_RAT_SC,CBD13,Morey, MU}).

\begin{setup}
	\label{setup_def_eq_Rees}
	Let $\kk$ be a field, $R = \kk[x_1,\ldots,x_d]$ be a standard graded polynomial ring and $\mm = (x_1,\ldots,x_d) \subset R$ be the graded irrelevant ideal. 
	Let $I\subset R$ be an ideal minimally generated by $n$ forms 
	$f_1,\ldots,f_n$  of the same degree $D \ge 1$.
	%, and set $I = (f_1,\ldots,f_n) \subset R$ to be the ideal generated by these forms.  
	Let $T = \kk[y_1,\ldots,y_n]$ be a standard graded polynomial ring and $\nnn = (y_1,\ldots,y_n) \subset T$ be the graded irrelevant ideal.
	Let $B$ be the standard bigraded polynomial ring $B = R \otimes_\kk T$ (i.e., $\bideg(x_i) = (1, 0)$ and $\bideg(y_i) = (0, 1)$). 
	Since we are primarily interested in the $\xx$-grading, for any bigraded $B$-module $M$, we denote by $M_i$ the graded $T$-module $M_i := \bigoplus_{j \in \ZZ} [M]_{(i,j)}$.
\end{setup}

As customary, we consider the $\kk$-algebra homomorphism 
$$
\Phi : B \twoheadrightarrow \Rees(I) = R[It] = \bigoplus_{j = 0}^\infty I^jt^j \subset R[t], \quad x_i \mapsto x_i \text{ and } y_i \mapsto f_it.
$$
Our goal is to determine the defining ideal $\JJ := \Ker(\Phi) \subset B$ of the Rees algebra $\Rees(I)$.
Let $\GG : \PP_\kk^{d-1} \dashrightarrow \PP_\kk^{n-1}$ be the rational map 
$$
(x_1:\cdots:x_d) \;\mapsto\; (f_1(x_1,\ldots,x_d) : \cdots : f_n(x_1,\ldots,x_d))
$$
determined by the forms $f_1,\ldots,f_n$ generating $I$.
The Rees algebra $\Rees(I)$ provides the bihomogeneous coordinate ring of the closure of the graph of $\GG$, and this reinforces the interest in finding the defining equations of $\Rees(I)$.

Typically, a good way to study the Rees algebra is to approximate it by the symmetric algebra, which is much better understood, at least as far as 
defining equations are concerned. If
$$F_1 \xrightarrow{\varphi} F_0 \rightarrow I \rightarrow 0$$ is a homogeneous minimal free presentation of $I,$ where the $i$-th basis element of $F_0$ maps to $f_i,$ then the defining ideal of $\Sym(I)$ 
is the kernel $\LL$ of the induced map $B=\Sym(F_0) \twoheadrightarrow \Sym(I)$. This
ideal can be described explicitly as
$$
\LL=\, (g_1,\ldots,g_u) =  I_1\left(\left[y_1,\ldots,y_n\right]  \cdot \varphi\right)
$$
%that is, $\Sym(F_0)=B,$ $\Sym(I) \cong B/\LL,$ and 
with $u = \rank(F_1)=\mu(\syz(I)).$ 
%is the minimal number of generators of the syzygy module of $I$.
There is a natural exact sequence of bigraded $B$-modules 
$$
0 \rightarrow \AAA \rightarrow \Sym(I) \rightarrow \Rees(I) \rightarrow 0,
$$
where $\AAA = \JJ/\LL$ coincides with the $R$-torsion of the symmetric algebra.
Therefore, one can study $\AAA$ to determine (or to obtain information about) the defining ideal $\JJ$ of the Rees algebra $\Rees(I)$.
Our main contribution in this direction is \autoref{thm_Jou_dual_Sym} below. To prove it,
we need the following adaptation of \autoref{thm_Sym_weak_Gor}(ii) to our bigraded setting:

\begin{proposition}
	\label{thm_Sym_weak_Gor_bigraded}
	Assume \autoref{setup_def_eq_Rees}. 
	Suppose that the ideal $I \subset R$ has height $g = \HT(I) \ge 2$ and satisfies the conditions $F_0$ and $\SD_1$. Then $\Sym(I)$ is $d$-weakly Gorenstein with respect to 
	$\mm\Sym(I),$ and the canonical 
	module has a weak generator that is bihomogeneous of bidegree $\left(d-(g-1)D,g-1\right).$
\end{proposition}

\begin{proof}
	We first prove the claim for the case when the field $\kk$ is infinite. 
	\autoref{thm_Sym_weak_Gor} implies that $S$ is $d$-weakly Gorenstein 
	with respect to $\m S,$ and a weak generator $u_1$ of the canonical module
	is identified in the proof of the same theorem. 
	To show that $u_1$ is bihomogeneous with 
	$$\bideg(u_1) = \left(d-(g-1)D,g-1\right),$$
	we note that the complex
	$$\mathfrak{L}_\bullet: \, 0 \rightarrow \mathfrak{T}_{g-2} \rightarrow \cdots \rightarrow \mathfrak{T}_{1} \rightarrow \mathfrak{T}_{0} \rightarrow \mathcal{Z}_{n-g} \rightarrow \cdots \rightarrow \mathcal{Z}_0 \rightarrow 0$$ of \autoref{prop_modified_Z_complex} can be made bihomogeneous. Thus we have
	$$
	\mathfrak{T}_{g-2} = \bigoplus_{j=1}^{g-1}B\left(-(g-j)D,-n+g-j\right).
	$$
	On the other hand, the surjection
	$$
	\Hom_B\left(\mathfrak{T}_{g-2},\omega_B\right) \otimes_B S \; = \; \bigoplus_{j=1}^{g-1}S\left((g-j)D-d,-g+j\right) \; \twoheadrightarrow \; \omega_S
	$$
	introduced in the proof of \autoref{thm_Sym_weak_Gor} is already bihomogeneous.
	As $u_1$ is the image of the first standard basis element of the module on the left, 
	it follows that $u_1$ is indeed bihomogeneous with $\bideg(u_1) = \left(d-(g-1)D,g-1\right).$ This completes the proof of the claim when $\kk$ is infinite.
	
	Next we treat the case of an arbitrary ground field $\kk.$ We proceed as in the 
	proof of \autoref{prop_det_2_by_2}. Let $\kk'$ be an infinite field  containing $\kk.$
	Write $R' := R \otimes_\kk \kk'$ and $S' := S \otimes_\kk \kk',$ and let 
	$u_1'$ be the weak generator of $\omega_{S'} \cong \omega_S \otimes_{\kk}\kk'$ considered 
	in the previous paragraph. As in the proof of \autoref{prop_det_2_by_2}, it suffices to
	show that, up to multiplication by a unit in $\kk',$ the element $u_1'$ is
	extended from a bihomogeneous element of $\omega _S,$ and this in turn follows once
	we have proved that the bigraded Hilbert function of $\omega_{S'}$ has value $1$
	when evaluated at $\bideg(u_1')$. \autoref{thm_Sym_weak_Gor} gives an
	isomorphism $\omega_{S'} \cong (f_1,y_1)^{g-2}S'.$ The proof of the same theorem
	shows that, up to sign, this isomorphism maps $u_1$ to the image 
	$\overline{y_1^{g-2}}$ of $y_1^{g-2}$ in $S'.$ That proof and the argument in
	the previous paragraph also show that the isomorphism $\omega_{S'} \cong (f_1,y_1)^{g-2}S'$
	is bihomogeneous, though not necessarily of bidegree $(0,0).$ So it 
	suffices to prove that the bigraded Hilbert function of $(f_1,y_1)^{g-2}S'$ has value $1$
	when evaluated at $\bideg(\overline{y_1^{g-2}})=(0,g-2),$ which is obvious.
\end{proof}

\vspace{.001cm}

\begin{theorem}
	\label{thm_Jou_dual_Sym}
	Assume \autoref{setup_def_eq_Rees}.
	Suppose that $I = (f_1,\ldots,f_n) \subset R$ has height $g = \HT(I) \ge 2$, that $I$ satisfies the conditions $G_d$ and $\SD_1$, and that $n = d+1$.
	Let $\delta := (g-1)D-d$ and $\beta := d-g+2$.
	Then the following statements hold$\, :$ 
	\begin{enumerate}[\rm (i)]
		\item $\mathcal{A} = \HL^0(\Sym(I))$.\smallskip
		
		\item For all $0 \le i \le d-1$, there is an isomorphism of bigraded $B$-modules 
		$$
		\HL^i(\Sym(I)) \, \cong \, {}^*\Ext_T^i\left(\Sym(I), T\right)\left(-\delta, -\beta\right).
		$$
		In particular, $\AAA = \HL^0(\Sym(I)) \cong {}^*\Hom_T\left(\Sym(I), T\right)\left(-\delta, -\beta\right) .$
		\smallskip
		
		\item 
		For all $i < 0$ and $i > \delta$, we have $\AAA_i = 0$. 
		There is an isomorphism $\AAA_\delta \cong T(-\beta)$ of graded $T$-modules.
		For all $0 \le i \le \delta$, we have the equality of $\Sym(I)$-ideals
		$$
		\AAA_{\ge i} \; = \; 0:_{\Sym(I)} \mm^{\delta+1-i}.
		$$
		
		\item If $g\ge 3$ and $\, I$ satisfies $\SD_2,$ then $\AAA$ is minimally generated in $ \xx$-degrees at most  $ (g-2)D-d+1\, .$	
		
		\vspace{.25cm}	
		\item 
		Let $0 \le i \le \delta$. 
		The natural multiplication map 
		$
		\mu : \AAA_i \; \otimes_T \; \Sym(I)_{\delta-i} \;\rightarrow\;  \AAA_\delta, \; a \otimes b \mapsto a\cdot b
		$
		is a perfect pairing that induces the abstract isomorphism of graded $T$-modules
		$$
		\nu : \AAA_i \;\xrightarrow{\;\cong\;}\; \Hom_T\left(\Sym(I)_{\delta-i}, \AAA_{\delta}\right)
		$$
		seen in part {\rm(ii)}.
		\smallskip
		
		\item For all $2 \le i \le d+1$, there is an isomorphism of bigraded $B$-modules 
		$$
		\HH_\nnn^i(\Sym(I)) \, \cong \, {}^*\Ext_R^{i-1}(\Sym(I), R)\left(-(g-1)D, g-1\right).
		$$
	\end{enumerate}
\end{theorem}
\begin{proof}
	To simplify notation, set $S := \Sym(I)$.
	
	\vspace{.2cm}	
	
	(i) From the assumed conditions it follows that $I$ is of linear type on the punctured spectrum of $R$ (see \cite[Theorem 5.1, Corollary 4.8]{HSV}). This implies that $\AAA = \HL^0(S)$.
	%	Alternatively, see \cite[\S 3.7]{KPU4}.
	
	\vspace{.2cm}			
	
	(ii) One uses \autoref{thm_Sym_weak_Gor_bigraded} and applies \autoref{thm_duality_Jou_bigrad} with $c=n-1=d$ and $a=-n=-d-1.$

	\vspace{.2cm}			
	
	(iii)
	The isomorphism $\AAA \cong {}^*\Hom_T\left(S, T\right)(-\delta,-\beta)$ implies that $\AAA_i=0$ when $i < 0$ and $i > \delta$, and it gives the isomorphism $\AAA_\delta \cong T(-\beta)$ of graded $T$-modules.
	The same isomorphism for $\AAA$
	%\cong {}^*\Hom_T\left(S, T\right)(-\delta,-\beta)$ 
	together with \autoref{LemmaB} shows the equality
	%We readily get the inclusion 
	$
	\AAA_{\ge i} \,=\, 0:_S\mm^{\delta+1-i}.
	$
	%Denote by $\varpi : \AAA \xrightarrow{\cong}  {}^*\Hom_T(S, T)(-\delta,-\beta)$ the isomorphism that we got from part (ii).
	%Let $0 \neq a \in \AAA_i$ be a non-zero element, and $\eta = \varpi(a) \in \Hom_T(S_{\delta-i}, \AAA_\delta)$  be the corresponding non-zero $T$-linear map. For any $b \in \mm^{\delta-i}$, we have that $b \cdot \eta \in \Hom_T(S_0, \AAA_\delta)$ is the map given by $(b \cdot \eta)(s) := \eta(b \cdot s)$ for all $s \in S_0$. Since $S_{\delta-i} = \mm^{\delta-i} \cdot S_0$, there exists  $b \in \mm^{\delta-i}$  such that $b \cdot \eta \neq 0$. Taking that same $b \in \mm^{\delta-i}$ and going backwards via $\varpi$ give us that $b \cdot a \neq 0 \in \AAA_\delta$. This implies that any $a \in 0:_S \mm^{\delta+1-i}$ has $\xx$-degree at least $i$. Therefore, we also have the other containment $\AAA_{\ge i} \supseteq 0:_S\mm^{\delta+1-i}.$
	
	\vspace{.2cm}			
	
	(iv) Since the ideal $I \subset R$ satisfies $\SD_2$, the graded strands  of bidegree $(*,k)$ of the complex $\mathfrak{L}_{\bullet}$ of \autoref{prop_modified_Z_complex} are acyclic complexes of $R$-modules that satisfy the assumptions of \cite[Theorem 4.3]{KPU5} with $i:=d$ and $t:=1$.
	As $g\geq 3,$ the $(d-1)$-st module in each of these complexes is 
	$$
	\left[\mathfrak{T}_{g-3}\right]_{(*,k)} = \left[\bigoplus_{j=1}^{g-2}B^n\left(-(g-1-j)D,-n+g-j\right)\right]_{(*,k)},
	$$
	so it follows that these $R$-modules are generated in degrees at most $(g-2)D$.
	Thus according to \cite[Theorem 4.3]{KPU5}, the $\xx$-degrees of the minimal generators of $\AAA$ are at most $(g-2)D-d+1$.
	
	\vspace{.2cm}			
	
	(v) This follows from the analogue of \autoref{Gen_perfectpairing}(iii) in the bigraded setting, which can be proved using part (ii).   
	%It is clear that the natural multiplication map $\AAA_i \otimes_T S_{\delta-i} \rightarrow \AAA_\delta$ induces a homomorphism 
	%$$
	%\nu : \AAA_i \rightarrow \Hom_T(S_{\delta-i}, \AAA_\delta). 
	%	$$
	%	From the bigraded $B$-isomorphism $\AAA \cong {}^*\Hom_T\left(S, T\right)(-\delta,-\beta)$ we already know that 
	%	$$
	%	\dim_\kk\left(\left[\AAA_i\right]_j\right) \, = \, \dim_\kk\left(\left[\Hom_T(S_{\delta-i}, \AAA_\delta)\right]_j\right)
	%	$$
	%	for all $j \in \ZZ$.
	%	Hence, to show that $\nu : \AAA_i \rightarrow \Hom_T(S_{\delta-i}, \AAA_\delta)$ is an isomorphism, it suffices to prove that $\nu$ is injective.
	%	However, part (iii) implies that for any $0 \neq a \in \AAA_i$, there is some $b \in \mm^{\delta-i} \subset R_{\delta-i} \subset S_{\delta-i}$ such that $b \cdot a \neq 0 \in \AAA_\delta$, and so it follows that $\nu(a) \in \Hom_T(S_{\delta-i}, \AAA_\delta)$ is not the zero map.
	%	So, $\nu$ is an isomorphism and the natural multiplication map $\mu : \AAA_i \otimes_T S_{\delta-i} \rightarrow \AAA_\delta$ is a perfect pairing.
	
	\vspace{.2cm}			
	
	(vi) Recall that $S$ is a Cohen-Macaulay ring of dimension $d+1$ (see \autoref{lem_ht_LL_with_F0} and \autoref{lem_Sym_CM}).
	We apply a duality result of Herzog and Rahimi \cite[Theorem, Corollary 1.6]{HeRa} that is being reproduced in \autoref{thm_Her_Rah_duality} below, where we also provide a short direct proof for the reader's convenience. For $2 \leq i \leq d+1$, we have the following isomorphisms of bigraded $B$-modules, where $-^{\vee_B}$ denotes bigraded $\kk$-duals,
	\begin{align*}
		\HH_\nnn^{i}(S) &\;\cong \; \big(\HH_\mm^{d+1-i}(\omega_S)\big)^{\vee_B} &\text{by \autoref{thm_Her_Rah_duality}} \\
		&\;\cong \; \big(\HH_\mm^{d+1-i}(S)\big)^{\vee_B}\left(d-(g-1)D,g-1\right) & \text{by 
			%\autoref{thm_duality_can_bigrad} and part (ii)
			\autoref{lem_weakly_Gor} and \autoref{thm_Sym_weak_Gor_bigraded}} \\
		&\;\cong \; \HH_\nnn^{i}(\omega_S)\left(d-(g-1)D,g-1\right) & \text{by \autoref{thm_Her_Rah_duality}}.
	\end{align*}
	By exchanging the roles of $R$ and $T$ in \autoref{thm_duality_can_bigrad}, we obtain the isomorphism 
	$$
	\HH_\nnn^i(\omega_S) \; \cong \; {}^*\Ext_R^{i-1}(S, R)(-d,0)
	$$
	for every $i \ge 0$.
	Combining the above isomorphisms we conclude that 
	$$
	\HH_\nnn^i(S) \, \cong \, {}^*\Ext_R^{i-1}(S, R)\left(-(g-1)D, g-1\right),
	$$  
	as required.
\end{proof}

If $B$ is a standard bigraded polynomial ring over a field $\kk$ as in \autoref{setup_def_eq_Rees} and $M$ is a bigraded $B$-module, then   
by

$$
(M)^{\vee_B} := \bigoplus_{i,j \in \ZZ} \Hom_\kk\big(\left[M\right]_{(-i,-j)}, \kk\big)
$$ 
we denote bigraded $\kk$-dual of $M$.
The next result provides a short proof of Herzog-Rahimi duality.
This duality was used already in the proof of the theorem above.

\begin{theorem}[{Herzog-Rahimi duality \cite[Theorem, Corollary 1.6]{HeRa}}]
	\label{thm_Her_Rah_duality}
	Let $B,$ $\mm,$ $\nnn$ be as in \autoref{setup_def_eq_Rees} and let $M$ be a finitely generated bigraded $B$-module.
	Then there are two converging spectral sequences of bigraded $\,B$-modules
	$$
	\HH_\nnn^p\left(\Ext_B^q(M, \omega_B)\right) \;\, \Longrightarrow\;\, \left(\HL^{d+n-p-q}(M)\right)^{\vee_B}
	$$
	and 
	$$
	\HH_\mm^p\left(\Ext_B^q(M, \omega_B)\right) \;\, \Longrightarrow\;\, \left(\HH_\nnn^{d+n-p-q}(M)\right)^{\vee_B}.
	$$
	In particular, if $M$ is Cohen-Macaulay, then there are two isomorphisms of bigraded $B$-modules 
	$$
	\HH_\nnn^i(\omega_M) \; \cong \; \big(\HH_\mm^{\dim(M)-i}(M)\big)^{\vee_B} \quad \text{ and } \quad \HH_\mm^i(\omega_M) \; \cong \; \big(\HH_\nnn^{\dim(M)-i}(M)\big)^{\vee_B}
	$$
	for every $i \in \ZZ$.
\end{theorem}
\begin{proof}
	Let $F_\bullet : \cdots \rightarrow F_i \rightarrow \cdots \rightarrow F_1 \rightarrow F_0 \rightarrow 0$ be a minimal bigraded free $B$-resolution of $M$.
	As in the second proof of \autoref{thm_duality_canonical_mod}, via the spectral sequences coming from the second quadrant double complex $F_\bullet \otimes_R C_\mm^\bullet$ we obtain the isomorphisms 
	$
	\HL^i(M) \, \cong \, \HH_{d-i}\left(\HL^d(F_\bullet)\right).
	$
	Dualizing with the exact functor $(\bullet)^{\vee_B}$, we now get
	$$
	\left(\HL^i(M)\right)^{\vee_B} \; \cong \; \HH^{d-i}\left(\left(\HL^d(F_\bullet)\right)^{\vee_B}\right) \; \cong \; \HH^{d-i}\Big(\HH_\nnn^n\big(\Hom_B(F_\bullet, \omega_B)\big)\Big)
	$$
	where the last isomorphism follows from the functorial isomorphism $\left(\HL^d(B)\right)^{\vee_B} \cong \HH_\nnn^n\left(\Hom_B(B, \omega_B)\right)$.
	Let $G^\bullet := \Hom_B\left(F_\bullet, \omega_B\right)$ and consider the first quadrant double complex $G^\bullet \otimes_T C_\nnn^\bullet$; the corresponding spectral sequences are given by 
	$$
	{}^{\text{I}}E_2^{p,q}\; = \; \begin{cases}
		\HH^p\left(\HH_\nnn^n(G^\bullet)\right) \,  \cong\, \left(\HL^{d-p}(M)\right)^{\vee_B} & \text{ if } q = n \\
		0 & \text{ otherwise}
	\end{cases}
	$$
	and 
	$$
	{}^{\text{II}}E_2^{p,q}\; = \; \HH_\nnn^p\left(\HH^q(G^\bullet)\right) \; \cong \; \HH_\nnn^p\left(\Ext_B^q(M, \omega_B)\right).
	$$
	Therefore, we obtain the convergent spectral sequence 
	$$
	\HH_\nnn^p\left(\Ext_B^q(M, \omega_B)\right) \;\, \Longrightarrow\;\, \left(\HL^{d+n-p-q}(M)\right)^{\vee_B}.
	$$
	The other spectral sequence is obtained by a completely symmetric argument.
	The two claimed isomorphisms follow immediately from the spectral sequences.
\end{proof}

\begin{remark}
	The isomorphisms of \autoref{thm_Jou_dual_Sym} are probably best seen in the form of a diagram. 
	Assume \autoref{setup_def_eq_Rees} with all the conditions and notations of \autoref{thm_Jou_dual_Sym}.
	We have the following diagram of bihomogeneous $B$-isomorphisms, where a label on an arrow specifies the range of $i$ where the isomorphism is valid:
	\begin{equation*}		
		\begin{tikzpicture}[baseline=(current  bounding  box.center)]
			\matrix (m) [matrix of math nodes,row sep=4.5em,column sep=8em,minimum width=2em, text height=1.8ex, text depth=0.25ex]
			{
				\left({}^*\Ext_R^{d-i}\left(\Sym(I), R\right)\right)^{\vee_B} (d,0)	& \\
				\left(\HH_\nnn^{d+1-i}\left(\omega_{\Sym(I)}\right)\right)^{\vee_B}	& \HL^i\left(\Sym(I)\right) \\
				\left(\HH_\nnn^{d+1-i}\left({\Sym(I)}\right)\right)^{\vee_B} (-\delta,-\beta+n)	& \HL^i\left(\omega_{\Sym(I)}\right)(-\delta,-\beta+n)\\
				& {}^*\Ext_T^i\left(\Sym(I), T\right)(-\delta,-\beta)\\
			};						
			\path[-stealth]
			(m-2-1) edge node [above] {\rm\tiny by \autoref{thm_Her_Rah_duality}} node [below] {} (m-2-2)
			(m-2-2) edge (m-2-1)
			(m-3-1) edge node [above] {\rm\tiny by \autoref{thm_Her_Rah_duality}} node [below] {} (m-3-2)
			(m-3-2) edge (m-3-1)
			(m-3-2) edge node [left] {\rm\tiny by \autoref{thm_duality_can_bigrad}} node [right] {} (m-4-2)
			(m-4-2) edge (m-3-2)
			(m-2-2) edge node [left] {\rm\tiny by \autoref{lem_weakly_Gor} and \autoref{thm_Sym_weak_Gor_bigraded}} node [right] {$0 \le i \le d-1$} (m-3-2)
			(m-3-2) edge (m-2-2)
			(m-1-1) edge node [left] {\rm\tiny by \autoref{thm_duality_can_bigrad}} node [right] {} (m-2-1)
			(m-2-1) edge (m-1-1)
			;
		\end{tikzpicture}	
	\end{equation*}
	In particular, when $0 \le i \le d-1$, we obtain the surprising fact that the above six $B$-modules are related by bigraded $B$-isomorphisms.
\end{remark}

We end this subsection with a simple condition for $\AAA$ to be free as a $T$-module. 

\begin{proposition}
	\label{prop_depth_Rees}
	Assume \autoref{setup_def_eq_Rees}.
	Suppose that $\Sym(I)$ is Cohen-Macaulay  and $\, I$ is of linear type on the punctured spectrum of $\, R$.
	Then $\depth(\Rees(I)) \ge d$ if and only if $\AAA$ is a free $T$-module. 
\end{proposition}
\begin{proof} Again, we write $S:=\Sym(I).$ Since $S$ is Cohen-Macaulay and $\dim(S) \geq \dim(\Rees(I)),$ the assertion is obvious if $\AAA=0.$ So we may assume that $\AAA \neq 0.$ In particular,
	$I \neq 0$ and so $\dim(S) \geq d+1.$ As $S$ is Cohen-Macaulay and $I$
	is not of linear type, but is of linear type on the punctured spectrum, 
	it follows that $n=d+1$ (see \cite[Theorem 2.6]{HR} and \cite[Proposition 2.4]{HSV}). 
	Therefore $\dim(T)=d+1.$
	
	We have the equality $\AAA = \HL^0(\Sym(I))$ because $I$ is of linear type on the punctured spectrum. Thus
	there is a positive integer $k > 0$ such that $\AAA$ is a module over $B/\mm^kB$. From the finite homomorphism $T \rightarrow B/\mm^k B$, we see that $\AAA$ is a finitely generated $T$-module and that 
	$$
	\depth_T(\AAA) \,= \, \depth_{B/\mm^kB}(\AAA) \,=\, \depth_{B}(\AAA). 
	$$
	
	Since $\depth(S) \geq d+1,$ the short exact sequence of $B$-modules $0 \rightarrow \AAA \rightarrow S \rightarrow \Rees(I) \rightarrow 0$ shows that
	$\depth(\Rees(I)) \geq d$ if and only if  $\text{depth}_{B}(\AAA) \ge d+1$ or, equivalently, $\text{depth}_{T}(\AAA) \ge d+1.$ The last inequality holds if and only if $\AAA$ is a free $T$-module because
	$d+1=\dim(T).$
\end{proof}

\subsection{Explicit equations via Morley forms}
\label{subsect_Morley}

We are now going to apply the theory of Morley forms developed in \autoref{Morley} to symmetric algebras of ideals.
%, and this allows us to make explicit the isomorphism 
%$
%\AAA = \HL^0(\Sym(I)) \cong {}^*\Hom_T\left(\Sym(I), T\right)\left(d-(g-1)D, -n+g-1\right)
%$ 
%from
%\autoref{thm_Jou_dual_Sym}. 
Unlike in the classical case (where $\Sym(I)$ is a complete intersection),  here the Morley forms \emph{do not} give a perfect pairing.
We instead need to introduce a division/reduction to make explicit the perfect pairing seen in \autoref{thm_Jou_dual_Sym}(v).
Throughout this subsection, we assume the hypotheses of \autoref{thm_Jou_dual_Sym}.

\begin{setup}
	\label{setup_Morley_forms}
	Assume \autoref{setup_def_eq_Rees}.
	Suppose that $I = (f_1,\ldots,f_n) \subset R$ has height $g = \HT(I) \ge 2$, that $I$ satisfies the conditions $G_d$ and $\SD_1$, and that $n = d+1$.
	In particular, \autoref{thm_Jou_dual_Sym} yields the isomorphism
	$$
	\AAA = \HL^0(\Sym(I)) \cong {}^*\Hom_T\left(\Sym(I), T\right)\left(-\delta, -\beta\right)
	$$
	where $\delta := (g-1)D - d$ and $\beta := d-g+2$.
	Assume that the defining ideal $\LL$ of $\Sym(I)$ contains a bihomogeneous regular sequence $\ell_1,\ldots,\ell_d$ that generates $\LL$ at $\mm B \in \text{Min}(\LL)$ and satisfies the
	conditions $\sum_{i=1}^d \deg_{\mathbf{x}}(\ell_i) = \delta + d$ and $\deg_{\mathbf{y}}(\ell_i) = 1$
	(recall that $\HT(\LL)=d=\HT(\mm B)).$
\end{setup}
\vspace{-.3cm}

Owing to the above bigraded isomorphism for $\AAA,$ we have 
$\AAA_{(\delta,\star)} \cong T(-\delta,-\beta),$ and the current definition of $\delta$ is consistent with the one in \autoref{remMF} and the
discussion following \autoref{thm_duality_Jou_bigrad} (see also \autoref{thm_Sym_weak_Gor_bigraded}).

%Owing to \autoref{thm_Sym_weak_Gor_bigraded}, or the above bigraded isomorphism for $\AAA,$
%the current definition of $\delta$ is consistent with the one in \autoref{remMF} and in the
%discussion following \autoref{thm_duality_Jou_bigrad}, and $\AAA_{(\delta},\star)} \cong T(-\delta,-%\beta).$ 

We write 
%We can write 
$$
\left[\ell_1,\ldots,\ell_d\right] \; = \; \left[x_1,\ldots,x_d\right] \cdot G\, 
$$
where $G \in B^{d \times d}$ is a $d \times d$ matrix whose entries are bihomogeneous with constant bidegrees along the columns. 
Multiplying with the adjoint of $G$ yields  
$$
\left(x_1,\ldots,x_d\right) \cdot \det(G) \,  \subset \, \left(\ell_1,\ldots,\ell_d\right) \,  \subset \, \LL.
$$
Let $\syl \in \Sym(I)$ be the image of $\det(G)$ in $\Sym(I)$.
Notice that $\syl \in 0:_{\Sym(I)} \mm \subset \AAA$. Moreover, $\syl$ is bihomogeneous
with $\bideg(\syl) = (\sum_{i=1}^d \deg_{\mathbf{x}}(\ell_i)-d, d) = (\delta,d),$ so
$$\syl \in \AAA_{(\delta,d)}.$$

As $(\ell_1, \ldots, \ell_d) \subset \mm B$ is an inclusion of complete intersection ideals of
height $d,$ we have that $$(\ell_1,\ldots,\ell_d) :_B \mm B = (\ell_1,\ldots,\ell_d, \det(G))$$
(see \cite{Wiebe}, \cite[Proposition 3.8.1.6]{Jo2}). Thus, up to multiplication by a unit in $\kk,$ the 
element $\det(G)$ modulo $(\ell_1, \ldots, \ell_d)$ only depends on the ideal 
$(\ell_1, \ldots, \ell_d).$ Hence, up to multiplication by a unit in $\kk,$ the element
$\syl$ is uniquely determined by $(\ell_1, \ldots, \ell_d).$ 
Following standard notation in the literature, we call $\syl$ the \emph{Sylvester form of $\ell_1,\ldots,\ell_d$ with respect to $x_1,\ldots,x_d$ } (or the \emph{determinant of the Jacobian dual}).

The element $\syl \in \AAA_{(\delta,d)}$ is non-zero. Indeed, since $\ell_1,\ldots,\ell_d$ generate $\LL$ locally at the minimal prime $\mm B \in \text{Min}(\LL)$, it follows that 
$$
(\ell_1,\ldots,\ell_d) :_B  \LL \,\not\subset\, \mm B.
$$
Since $\mm B$ is a prime ideal, we obtain $\LL \not\supset (\ell_1,\ldots,\ell_d) :_B \mm B = (\ell_1,\ldots,\ell_d, \det(G))$. Thus $\det (G) \notin \LL,$ as asserted.

We choose a generator $\rsyl \in \AAA_{(\delta,\beta)} \cong \kk(- \delta, - \beta)$ and 
call it a \emph{reduced Sylvester form}. It coincides with the element $\mathfrak s$ in \autoref{Gen_perfectpairing} and the discussion following \autoref{thm_duality_Jou_bigrad}.
As $\syl \in \AAA_{(\delta, d)}$ and $\AAA_{(\delta,\star)} \cong T(-\delta,-\beta),$ there
exists a unique element $\alpha \in T_{d- \beta}$ so that
$$
\syl = \alpha \cdot \rsyl.
$$
In other words, $\rsyl = \frac{\syl}{\alpha}.$

We now construct an explicit element $\Delta\in \left[0:_{\Sym(I)^e} \mathbb D\right]_{(\delta, \star)}$ that can be used in the definition of Morley forms, as described in \autoref{Morley} and the
discussion following \autoref{thm_duality_Jou_bigrad}.
%Consider the enveloping algebra $B^e=B \otimes_T B$. 
%and notice that the diagonal ideal $\left(x_1\otimes 1 - 1 \otimes x_1, \;\ldots,\; x_d \otimes 1 - 1 \otimes x_d \right) \subset B \otimes_T B$ is the kernel of the natural multiplication map $B \otimes_T B \rightarrow B$.
As each $\ell_i\otimes 1 - 1 \otimes \ell_i$ is in the kernel of the multiplication map $B^e \rightarrow B$, we can write 
$$
\left[\ell_1\otimes 1 - 1 \otimes \ell_1, \;\ldots,\; \ell_d \otimes 1 - 1 \otimes \ell_d \right] \; = \; \left[x_1\otimes 1 - 1 \otimes x_1, \;\ldots,\; x_d \otimes 1 - 1 \otimes x_d \right] \cdot H\, ,
$$
where $H \in (B^e)^{d \times d}$ is a $d \times d$ matrix whose entries are bihomogeneous with constant bidegrees along the columns.
%whose entries  are bihomogeneous elements of $B^e$ with constant bidegrees along the columns.
Let 
$
\Delta
$
be the image of $\det(H)$ in $\Sym(I)^e$. Notice that  $\Delta\in 0:_{\Sym(I)^e} \mathbb D$  and  that $\Delta$ is bihomogeneous of bidegree $(\delta, d)$, in other words
$$\Delta\in [0:_{\Sym(I)^e} \mathbb D]_{(\delta, d)}.$$

Let $\Pi: B \twoheadrightarrow T$ be the homomorphism of $T$-algebras with $\Pi(x_i)=0$ for all $i$, and consider the map $\epsilon = B \otimes_T \Pi: B^e \twoheadrightarrow B$. Applying  $\epsilon$
to the entries of $H$ we obtain a matrix $G$ as above. Since $\syl$ is well defined up to
multiplication by a unit in $\kk,$ we may assume that $$\varepsilon(\Delta)=\syl,$$  where 
$\varepsilon: S^e \twoheadrightarrow S$ is defined as in \autoref{Morley}. Recall that $\syl=\alpha \cdot \rsyl$ and $\rsyl=\mathfrak{s}.$

On the other hand, $\varepsilon(\Delta)= \varepsilon(\morl_{(\delta,0)})$ and $\morl_{(\delta,0)}=\alpha \cdot \mathfrak{s} \otimes 1$ with $\alpha$ defined as in \autoref{lem_Morley_forms2}(i). Comparing
the two expressions for $\varepsilon(\Delta)$ we see that the current definition of $\alpha$ coincides with the one in \autoref{Morley}.

Since $\varepsilon(\Delta)=\syl$ and $\syl \neq 0$ by the above, we
also see that $\Delta \notin \ker{\varepsilon}=(1 \otimes \mm)\, \Sym(I)^e,$ where
$\mm$ generates the unique associated prime of the ideal $\mm \, \Sym(I).$ Thus the
hypotheses of \autoref{thm_Morley_forms1} are satisfied and we obtain:

\begin{remark}
\label{rem_classical_Morley}
Assume \autoref{setup_Morley_forms}.
The element $\Delta = \overline{\det(H)} \in \Sym(I)^e$ can be used to define Morley forms as in \autoref{Morley}.
Thus we obtain the decomposition
$$
\Delta \;=\; \overline{\det(H)} \; = \; \sum_{i=0}^{\delta} \morl_{(\delta - i,i)} \quad\text{ where }\quad \morl_{(\delta - i,i)} \in \big[\Sym(I)^e\big]_{(\delta - i,i,d)}.
$$
Applying \autoref{thm_Morley_forms1} with $\mathfrak s:=\rsyl$ and $\alpha$ defined
by the identity $\syl=\alpha \cdot \rsyl,$ we obtain explicit isomorphisms between the $T$-modules
$\Hom_T(\Sym(I)_{\delta-i}, \AAA_\delta)$ and  $\AAA_i$. 
\end{remark}

The above requirement that $\Sym(I)$ is a complete intersection at the minimal prime $\mm \, \Sym(I)$ is satisfied in several situations of interest (see \autoref{rem_zero_dim_plans} and \autoref{rem_general_eqs_Sym_Gor_3}).

\vspace{.1cm}

\section{Applications to certain families of ideals}
\label{sect_applications}

In this section, we apply the results and techniques developed in previous sections. 
For organizational purposes, we divide the section into two subsections. 
Each subsection covers a family of ideals where our results are applied.

\subsection{Zero dimensional ideals}

In this subsection, we study Rees algebras of zero dimensional ideals. 
In this case, we are able to approximate the Rees algebra with two algebras: one is the symmetric algebra of the ideal, and the other is the symmetric algebra of the module whose syzyies are the Koszul syzygies of the ideal.
This provides two methods for studying Rees algebras. 
We will use the following setup. 

\begin{setup}
\label{setup_zero_dim}
Assume \autoref{setup_def_eq_Rees} with $I \subset R$ an $\mm$-primary ideal, and suppose that $n = d +1$.
Let $(K_\bullet, \partial_\bullet)$ be the Koszul complex of the sequence $f_1,\ldots,f_n$, and $H_\bullet$ be its homology.
Let $E := \Coker(\partial_2)$ be the module defined by the Koszul syzygies.
\end{setup} 

The next lemma collects some basic facts about the relation between the module $E$ and the ideal $I$.

\begin{lemma}
\label{lem_general_properties_Sym_E}
The following statements hold$\, :$
\begin{enumerate}[\rm (i)]
	%\item There is a short exact sequence 
	%$
	%0 \rightarrow H_1 \rightarrow E \rightarrow I \rightarrow 0.
	%$
	%\item $\rank(E)=1$ and \,$\Supp(H_i) \subset \{\mm\}$ \,for all $i \ge 0$.
	\item $\Sym(E) \cong B/\mathcal{K}$ where $\mathcal{K}$ is the determinantal ideal 
	$$
	\mathcal{K} \,=\, I_2\left(\begin{array}{cccc}
		y_1 & y_2 & \cdots &y_{d+1}\\
		f_1 & f_2 & \cdots & f_{d+1}
	\end{array}\right).
	$$
	\item $\Rees(E) \cong \Rees(I)$, and there is a natural exact sequence 
	$
	0 \rightarrow \HL^0(\Sym(E)) \rightarrow \Sym(E) \rightarrow  \Rees(I) \rightarrow 0. 		
	$
	\item $\dim\left(\Sym(E)\right)=\dim\left(\Sym(I)\right)=d+1$.
\end{enumerate}
\end{lemma}
\begin{proof}
%	htarrow H_1 \rightarrow E \rightarrow I$.
%	Since $I\cdot H_i=0$ and $I$ is $\mm$-primary, then we have $\mm^k\cdot H_i=0$ for some $k > 0$.
%	Therefore,  it follows that $\Supp(H_i) \subset \{\mm\}$ and $\rank(H_i)=0$.
%	Part $(i)$ and the additivity of rank, yield that $\rank(E)=\rank(I)=1$.	

(i) By construction, a presentation of $E$ is given by 
$
K_2 \xrightarrow{\partial_2} K_1 \rightarrow E \rightarrow 0.
$
Let $\{e_1,\ldots,e_{d+1}\}$ be a basis of $K_1$
with $\partial_1(e_i)=f_i$ for all $i.$ Then 
$\{e_i \wedge e_j\}_{1 \le i < j \le d+1}$ is a basis of $K_2,$
and		the map $\partial_2$ is defined by 
$
\partial_2(e_i \wedge e_j) = f_ie_j - f_je_i.
$ 
Let $\left[\partial_2\right]$ be the matrix representation of $\partial_2$ with respect to the chosen bases of $K_1$ and $K_2$. 
Then the defining ideal of $\Sym(E)$ is
$$
\mathcal{K} = I_1\left( \left[y_1, y_2, \ldots, y_{d+1}\right]  \cdot \left[\partial_2\right] \right)
= \left( {\{f_iy_j - f_jy_i\}}_{1 \le i < j \le d+1} \right)
=  I_2\left(\begin{array}{cccc}
	y_1 & y_2 & \cdots &y_{d+1}\\
	f_1 & f_2 & \cdots & f_{d+1}
\end{array}\right).
$$

(ii) The existence of the exact sequence follows from the fact that   $E_\pp \cong I_\pp = R_\pp$ for every $\pp \in \Spec(R) \setminus \{\mm\}$. The exact sequence in turn shows that the kernel of the natural surjection $\Sym(E) \twoheadrightarrow  \Rees(I)$ is the $R$-torsion of $\Sym(E).$

(iii) This part follows from the Huneke-Rossi formula (see \cite{HR}, \cite[Theorem 1.2.1]{V}).
\end{proof}

As a consequence of the lemma above, we have the following two short exact sequences
$$
0 \rightarrow \AAA \rightarrow \Sym(I) \rightarrow \Rees(I) \rightarrow 0 \quad \text{ and } \quad 0 \rightarrow \mathcal{B} \rightarrow \Sym(E) \rightarrow \Rees(I) \rightarrow 0
$$
where $\AAA = \HL^0(\Sym(I))$ and $\mathcal{B} = \HL^0(\Sym(E))$.
The next theorem makes $\Sym(I)$ and $\Sym(E)$ good candidates to approximate $\Rees(I)$ when $I \subset R$ is $\mm$-primary.

\begin{theorem}
\label{thm_zero_dim}
Assume \autoref{setup_zero_dim}.
Then the following statements hold$\,:$
\begin{enumerate}[\rm (i)]
	\item For all $0 \le i \le d-1$, there are isomorphisms of bigraded $B$-modules 
	$$
	\HL^i(\Sym(I)) \, \cong \, {}^*\Ext_T^i\left(\Sym(I), T\right)\left(-dD-d-D, -2\right).
	$$
	In particular, $\AAA = \HL^0(\Sym(I)) \cong {}^*\Hom_T\left(\Sym(I), T\right)\left(-dD-d-D, -2\right)$.
	\item For all $0 \le i \le d-1$, there are isomorphisms of bigraded $B$-modules 
	$$
	\HL^i(\Sym(E)) \, \cong \, {}^*\Ext_T^i\left(\Sym(E), T\right)\left(-dD+d, -1\right).
	$$
	In particular, $\mathcal{B} = \HL^0(\Sym(E)) \cong {}^*\Hom_T\left(\Sym(E), T\right)\left(-dD+d, -1\right)$.
\end{enumerate}
\end{theorem}
\begin{proof}
Part (i) follows from \autoref{thm_Jou_dual_Sym}(ii) (notice that $d\geq 2$), and part (ii) is a consequence of \autoref{prop_det_2_by_2}.
\end{proof}

\begin{remark}
\label{rem_zero_dim_plans}
After having the abstract duality statements of part (ii) in the above theorem, all the statements of \autoref{thm_Jou_dual_Sym} hold for $\Sym(E)$ with bigraded shifts given by  $\delta := dD-d$ and $\beta := 1$. Also, at least when $\kk$ is infinite,
%Also, since 
the defining ideal of $\Sym(E)$ contains a regular sequence as in \autoref{setup_Morley_forms}, and so the technique of generalized Morley forms 
%is a complete intersection at the minimal prime $\mm \, \Sym(E)$, the general Morley form techniques 
developed in \autoref{subsect_Morley}, and in particular \autoref{rem_classical_Morley}, apply to $\Sym(E)$.
We are not pursuing this approach in the present paper due to length constraints.
In a subsequent paper, we plan to study Rees algebras of $\mm$-primary ideal; 
from a geometric point of view this is relevant, as it entails studying the graph of morphisms $\PP_\kk^{d-1} \rightarrow \PP_\kk^{d}$ parametrizing a hypersurface.
\end{remark}

\subsection{Gorenstein ideals of height three}\label{sec_last}

In this subsection, we concentrate on Gorenstein ideals of height $3$ and we determine the defining equations of the Rees algebra for a particular family.

\begin{setup}
\label{setup_gorenstein}
In addition to \autoref{setup_def_eq_Rees},
assume that $I = (f_1,\ldots,f_n) \subset R$ is a Gorenstein ideal of height $3$ with $\mu(I) = n = d+1$ and that $I$ satisfies $G_d$.
Let $\varphi \in R^{(d+1) \times (d+1)}$ be an alternating presentation matrix of $I$ whose non-zero entries are homogeneous of degree $h \ge 1$ {\rm(}consequently, $\deg(f_i) = D = \frac{n-1}{2} h=\frac{d}{2}h ${\rm)}.  \end{setup}

\begin{theorem}
\label{thm_Gor_3}
Assume \autoref{setup_gorenstein}. 
Then we have bigraded isomorphisms of $\, B$-modules 
$$
\HL^i(\Sym(I)) \, \cong \, {}^*\Ext_T^i\left(\Sym(I), T\right)\left(-d(h-1), -d+1\right)
$$
for all $0 \le i \le d-1$.
\end{theorem}
\begin{proof}
%[First proof]
The statement is a particular case of \autoref{thm_Jou_dual_Sym}(ii).
\end{proof}

We notice that with the assumptions above the ideal $I$ is of linear type locally on the
punctured spectrum (by \autoref{thm_Jou_dual_Sym}(i), for instance) and that $I_1(\varphi)$
is an $\mm$-primary ideal (because $I$ satisfies $G_d$ and $\mu(I)\ge d$). If $I_1(\varphi)$ is a complete 
intersection, the defining ideal of $\Rees(I)$ has been determined explicitly in \cite{KPU4}*{Theorem 
9.1 and Remark 9.2}, without any restriction on the number of generators of $I$ (see also \cite{Johnson}*{2.10} for the case case $n=d+1$). Thus we may assume that $I_1(\varphi)$ is 
not a complete intersection. In this subsection we are going to treat this case under the
following hypotheses:

%	For the rest of this subsection, we consider a particular case that was not covered before in the literature. 
%	This shows the applicability of the methods developed in this paper. 

\begin{setup}
\label{setup_pfaffians}
In addition to \autoref{setup_gorenstein},
assume that $I_1(\varphi)$ is an almost complete intersection, but not a complete intersection. Let $c_1, \ldots, c_{d+1}$ be homogeneous generators of $I_1(\varphi)$, necessarily of degree $h$. 
%Hence we may suppose that $I_1(\varphi) = \left(m, x_1^h, \ldots, x_d^h\right)$ where $m$ is a monomial of degree $h$.
The defining ideal of $\Sym(I)$ is generated by the entries of the row vector $\left[y_1,\ldots,y_{d+1}\right] \cdot \varphi $. This vector can be rewritten as
$$
\left[y_1,\ldots,y_{d+1}\right] \cdot \varphi   \;= \; \left[c_1, ,\ldots, c_{d+1}\right] \cdot A,
$$
where $A \in T^{(d+1)\times (d+1)}$ is a $(d+1)\times (d+1)$ matrix 
with linear entries in $T$.
The matrix $A$ is referred to as the \emph{Jacobian dual} of $\varphi$ in the literature.
Let $\Delta_{i,j}$ be the minor of $A$ obtained by deleting the $i$-th row and the $j$-th column multiplied by $(-1)^{i+j}$. 
\end{setup}

%\begin{example}
%	The  matrix
%\begin{equation*}
%	\varphi = \left[\!\begin{array}{ccccc}
%		0&x_{1}^{2}&x_{2}^{2}&x_{3}^{2}&x_{4}^{2}\\
%		-x_{1}^{2}&0&x_{1}x_{2}&x_{2}^{2}&x_{3}^{2}\\
%		-x_{2}^{2}&-x_{1}x_{2}&0&x_{1}^{2}&x_{2}^{2}\\
%		-x_{3}^{2}&-x_{2}^{2}&-x_{1}^{2}&0&x_{1}^{2}\\
%		-x_{4}^{2}&-x_{3}^{2}&-x_{2}^{2}&-x_{1}^{2}&0
%	\end{array}\!\right] 
%\end{equation*}
%satisfies the conditions of \autoref{setup_pfaffians} with $d = 4$, $n = 5$ and $h = 2$.
%Its Jacobian dual is
%\begin{equation*}
%	A = \left[\!\begin{array}{ccccc}
%		0&-y_{3}&y_{2}&0&0\\
%		-y_{2}&y_{1}&-y_{4}&y_{3}-y_{5}&y_{4}\\
%		-y_{3}&-y_{4}&y_{1}-y_{5}&y_{2}&y_{3}\\
%		-y_{4}&-y_{5}&0&y_{1}&y_{2}\\
%		-y_{5}&0&0&0&y_{1}
%	\end{array}\!\right].
%\end{equation*}
%One can check that $[y_1,\ldots,y_5] \cdot \varphi = [x_1x_2, x_1^2, \ldots, %x_4^2] \cdot A$.
%\end{example}

\begin{lemma}
\label{lem_Gor_3_prop_A}
Assume \autoref{setup_pfaffians}. The matrix $A$ has the following properties$\, : $
\begin{enumerate}[\rm (i)]
	\item $\rank(A) = d$.
	\item $\Ker(A)$ is a free $T$-module generated by the vector $[y_1,\ldots,y_{d+1}]^t$.
	\item 
	The adjoint of $A^t$ is of the form 
	$$
	{\rm adj}(A^t) \; = \; 
	\left[\begin{array}{cccc}
		y_1\delta_1 & y_2\delta_1 & \cdots & y_{d+1}\delta_1 \\
		y_1\delta_2 & y_2\delta_2 & \cdots & y_{d+1}\delta_2\\
		\vdots & \vdots & \cdots & \vdots\\
		y_1\delta_{d+1} & y_2\delta_{d+1} & \cdots & y_{d+1}\delta_{d+1}
	\end{array}\right]
	$$
	where $\delta_i = \Delta_{i,j}/y_j \in T$.
	\item The ideal $I_d(A)\subset T$ satisfies $\HT(I_d(A)) \geq 2.$
	\item 
	$\Ker(A^t)$ is a free $T$-module generated by the vector $\left[\delta_1\ldots,\delta_{d+1}\right]^t$.
\end{enumerate}
\end{lemma}
\begin{proof}
(i) We first show that $\rank(A) \geq d.$ Let $Q$ be the quotient field of $T.$ Notice that $\LL':=\LL(R \otimes_{\kk}Q)$ is a
proper ideal of the polynomial ring $R \otimes_{\kk}Q$ and that $\HT(\LL') \geq \HT(\LL) =n-1=d.$ Thus $\LL'$ requires at least $d$ generators, showing that $\rank(A) \geq d.$

\begin{comment}
	has height at least $\HT(I_d(A)) \ge 2$.
	By contradiction, suppose not. 
	Then there exists $\pp \in \Spec(T)$ of height one containing $I_d(A)$.
	Since the only minimal primes of $\Sym(I)$ are $\mm \, \Sym(I)$ and $\AAA$, and $\pp$ is in none of them, it follows that 
	$
	\HT(\LL, \pp) \ge d+1.
	$
	Hence $\HT\left((\LL, \pp)/\pp \right)\ge d$, and so we have that $\HT\left(\LL (R \otimes_\kk Q)\right) \ge d$ where $Q = \Quot(T/\pp)$.
	But the image of $A$ in $Q$ has rank $\le d-1$. 
	Hence the proper ideal $\LL (R \otimes_\kk Q)$ can be generated by $d-1$ elements, a contradiction.
\end{comment}

Since $\varphi$ is alternating, it follows that 
$$
[m,x_1^h,\ldots,x_d^h] \cdot A \cdot [y_1,\ldots,y_{d+1}]^t = [y_1,\ldots,y_{d+1}] \cdot \varphi \cdot [y_1,\ldots,y_{d+1}]^t = 0,
$$
and so $A \cdot [y_1,\ldots,y_{d+1}]^t = 0$.
This shows that $\rank(A) \le d,$ which  
%Since $I_d(A) \neq 0$, we have the desired equality.
settles part (i).

\vspace{.15cm}

(ii)	By the above paragraph, we have a complex $0 \rightarrow T \xrightarrow{\left[y_1,\ldots,y_{d+1}\right]^t} T^{d+1} \xrightarrow{\; A \; } T^{d+1}$. 
This complex is exact by the Buchsbaum-Eisenbud acyclicity criterion since $\HT(y_1,\ldots, y_{d+1}) \ge 2$ and $\rank(A) = d$.
This shows part (ii).

\vspace{.2cm}

(iii)	As $A \cdot \left[y_1,\ldots,y_{d+1}\right]^t = 0$ and $A \cdot \text{adj}(A)=0$, part (ii) implies that each column of $\text{adj}(A)$ is a multiple of $[y_1,\ldots,y_{d+1}]^t$.
This establishes part (iii).

\vspace{.2cm}

(iv) Suppose that $\HT(I_d(A)) \leq  1.$ In this case, $I_d(A) \subset \pp$ for some homogeneous $\pp \in \Spec(T)$ of height one. 

The only minimal primes of the defining ideal $\LL$ of $\Sym(I)$ are 
$\mm B$ and $\JJ$, the defining ideal of $\Rees(I)$. Indeed, let 
$\mathfrak{P}$ be a minimal prime of $\LL;$ if $\mathfrak{P} \cap R =
\mm$ then $\mathfrak{P} \supset \mm B$ and so $\mathfrak{P} =\mm B$; if $\mathfrak{P} \cap R = \qqq \subsetneq \mm$ then $\LL_{\mathfrak{P}} = \JJ_{\mathfrak{P}}$, because $I$ is of linear type on the punctured spectrum of $R,$ and so $\mathfrak{P}=\JJ.$

We now show that $\pp \not\subset \JJ$. To see this we determine the analytic spread $\ell:=\ell(I)$ of $I,$ which 
is the dimension of the special fiber ring $\mathcal{F}:=\Rees(I)\otimes_R \kk = \Rees(I)_{(0,\star)}.$ For this we may assume that $\kk$ is infinite, in which case $I$ is integral 
over a homogeneous ideal $J$ generated by $\ell$ elements. 
Since $I$ is of linear type on the punctured spectrum, it follows that $\HT(J:I) \geq d.$ 
Thus $ \mu(J) \geq d$ by \cite{HSCM}*{Theorem 3.1(i)} because $I$ is strongly Cohen-Macaulay and satisfies $G_d$. 
Therefore $\ell \ge d.$ On the other hand, one always has $\ell \leq d,$ hence $\ell =d$. This shows that $\mathcal{F}$ is a domain of embedding codimension 1, and so 
$\mathcal {F} \cong T/fT$, where $f$ is homogeneous with $\deg(f)-1=r,$ the reduction number of $I.$

By part (iii), $\delta_i y_j \in \pp$ for all $i,j$ and so 
$\delta_i \in \pp$ for all $i.$ Since at least one $\delta_i \neq 0$ by part (i), it follows that
$\pp$ contains a non-zero homogeneous polynomial in $T$ of degree $d-1$. If $\pp \subset \JJ,$ then
$\JJ \cap T=fT$ contains a non-zero homogeneous element of degree $d-1,$ which shows that $r=\deg(f) -1 \leq d-2.$ By \cite{PU99}*{Theorem 3.1} (see also \cite{UBal}*{Corollary 5.6}) the inequality $r \leq d-2=\ell-2$ implies that $\mu(I_1(\varphi)) \leq d.$  So $I_1(\varphi)$ is a complete intersection, which is ruled out by \autoref{setup_pfaffians}. This proves that $\pp \not\subset \JJ.$

As $\pp \subset T$ we trivially have that $\pp \not\subset \mm B,$ hence 
$\pp$ is not contained in any minimal prime of $\LL.$ Thus
$\HT(\LL + \pp B) \geq \HT(\LL) +1 =d+1,$ and so 
$\HT\left((\LL + \pp B)/\pp B \right) \geq d.$ The rest of the argument is similar to the proof
of part (i), but now we take $Q$ to be the quotient field of $T/\pp$ and
$\LL':=\LL(R \otimes_{\kk}Q).$ As the ideal $\LL'$ is proper and has height 
at least $d,$ it follows that $\LL'$ requires at least $d$ generators. Therefore 
the image of $A$ in $Q$ has rank at least $d,$ which shows that
$I_{d}(A) \not\subset \pp$, contrary to our assumption. This proves
that indeed $\HT(I_{d}(A)) \geq 2.$

\begin{comment}

	We first show that $\rank(A) \geq d.$ Let $Q$ be the quotient field of $T.$ Notice that $\LL':=\LL(R \otimes_{\kk}Q)$ is a
	proper ideal of the polynomial ring $R \otimes_{\kk}Q$ and that $\HT(\LL') \geq \HT(\LL) =d.$ Thus $\LL'$ requires at least $d$ generators, showing that $\rank(A) \geq d.$

	the reduction number is $d-2=\ell(I)-2$ which is the expected one. Hence, according to \cite{UBal}*{}, $I_1(\varphi)$ satisfies the row condition. This implies that $I_1(\varphi)$ is a complete intersection, a contradiction.

	As $\pp B$ is not in the minimal primes of $\Sym(I)$, it follows that 
	$$
	\HT(\LL, \pp) \,\ge\, d+1.
	$$
	Hence $\HT\left((\LL, \pp)/\pp B \right)\ge d$, and so we have that $\HT\left(\LL (R \otimes_\kk Q)\right) \ge d$ where $Q = \Quot(T/\pp)$.
	But the image of $A$ in $Q$ has rank $\le d-1$. 
	Hence the proper ideal $\LL (R \otimes_\kk Q)$ can be generated by $d-1$ elements, a contradiction.	
\end{comment}

(v) Part (iii) 
gives the complex	
$$0 \rightarrow T \xrightarrow{\left[\delta_1,\ldots,\delta_{d+1}\right]^t} T^{d+1} \xrightarrow{\; A^t \; } T^{d+1}$$ and the inclusion $(\delta_1,\ldots,\delta_{d+1}) \supset I_d(A).$
Since $\HT(I_d(A)) \ge 2$ by (iv), the Buchsbaum-Eisenbud acyclicity criterion shows that this complex is exact.
\end{proof}

\begin{remark}
\label{rem_general_eqs_Sym_Gor_3}
Assume that $\kk$ is an infinite field. 
Then there exists a bihomogeneous regular sequence $\ell_1,\ldots,\ell_d$ belonging to $\LL$ that generates $\LL$ at $\mm B \in \text{Min}(\LL)$,  and satisfies $\bideg(\ell_i) = (h, 1)$ (as required in \autoref{setup_Morley_forms}).
\end{remark}
\begin{proof}
Since $\rank(A) = d$ (see \autoref{lem_Gor_3_prop_A}(i)), $\mu(\LL \otimes_B B_{\mm B}) \le d$.
Now the claim follows because $\LL$ is generated in bidegree $(h,1)$.
\end{proof}

%For brevity of exposition, we specialize to the case of $h = 2$, and we utilize the following setup for the remaining of this subsection.
In the next theorem, we determine explicitly the defining equations of the Rees algebra of a Gorenstein ideal as in \autoref{setup_pfaffians} with $h = 2$ and $I_1(\varphi)$ a monomial ideal.

\begin{theorem}
\label{thm_sublime_h_2}
In addition to  \autoref{setup_pfaffians} assume that $h = 2$ and $I_1(\varphi)$ is a monomial ideal.
Without loss of generality we may assume that $I_1(\varphi)=(x_1x_2, x_1^2, \ldots, x_d^2)\, .$
Then the following statements hold$\, : $
\begin{enumerate}[\rm (i)]
	\item The natural isomorphism of bigraded $B$-modules ${}^*\Hom_T(B, T) \cong T[x_1^{-1}, \ldots, x_d^{-1}]$ restricts to an isomorphism 
	\begin{align*}
		{}^*\Hom_T(\Sym(I), T) &\;\cong\; B \cdot \left(\delta_1 x_1^{-1}x_2^{-1} + \delta_2x_1^{-2} + \cdots + \delta_{d+1} x_d^{-2}\right) x_3^{-1}\cdots x_d^{-1}\\
		&\quad + B \cdot x_1^{-1}x_3^{-1}\cdots x_d^{-1}  \\
		&\quad + B \cdot x_2^{-1}x_3^{-1}\cdots x_d^{-1}  \;\subset\; T[x_1^{-1},\ldots,x_d^{-1}].
	\end{align*}	
	In addition, $\AAA = \HL^0(\Sym(I)) \cong {}^*\Hom_T\left(\Sym(I), T\right)(-d,-d+1)$ is a free $T$-module that is minimally generated as a $B$-module by three elements of bidegrees $(0, 2d-2),$ $(1, d-1),$ and $(1,d-1)$.

	\item We use the notation of $\, \autoref{subsect_graded_T}$ and in particular \autoref{cor_Morley_forms}. 
	Define elements $U_i \in \Sym(I)$ by the equations
	\begin{align*}
		&1\otimes_T U_0 \;=\; \frac{1}{\alpha} \cdot \morl_{(d,0)} \cdot \left( \left(\delta_1 x_1^{-1}x_2^{-1} + \delta_2x_1^{-2} + \cdots + \delta_{d+1} x_d^{-2}\right) x_3^{-1}\cdots x_d^{-1} \otimes_T 1\right)\\ 
		&1\otimes_T U_1 \;=\; \frac{1}{\alpha} \cdot \morl_{(d-1,1)} \cdot \left( x_1^{-1} x_3^{-1}\cdots x_d^{-1}\otimes_T 1\right) \\
		%\text{ and } \\ 
		&1\otimes_T U_2 \;=\; \frac{1}{\alpha} \cdot \morl_{(d-1,1)} \cdot \left( x_2^{-1} x_3^{-1}\cdots x_d^{-1}\otimes_T 1\right) .
	\end{align*}
	These elements have bidegees $(0, 2d-2),$ $(1, d-1),$  $(1,d-1),$ respectively, and they generate the ideal $\AAA \subset \Sym(I)$ minimally.
	In particular, the defining ideal $\JJ \subset B$ of $\, \Rees(I)$ is minimally generated by $\LL$ and lifts to $B$ of the elements $U_0,U_1,U_2$.\\
	Moreover by \autoref{rem_general_eqs_Sym_Gor_3}, if $\, \kk$ is an infinite field, then 
	there exists a $B$-regular sequence $\ell_1,\ldots,\ell_d$ in $\LL$ that satisfies the conditions of \autoref{setup_Morley_forms}; therefore the discussion following \autoref{setup_Morley_forms} and in particular \autoref{rem_classical_Morley} provide explicit Morley forms and the linear polynomial $\alpha\in T$.
	
	\item Let $\, \GG : \PP_\kk^{d-1} \dashrightarrow \PP_\kk^d$ be the rational map determined by the forms $f_1,\ldots,f_n$ generating $I.$
	Let $Y \subset \PP_\kk^{d}$ be the closure of the image of $\, \GG.$
	The degree of $\, \GG$ is equal to $\deg(\GG) = 2^{d-2}$ and $\deg(Y) = 2d-2.$
	
	\item $\depth\left(\Rees(I)\right) = d$ and $\depth\left(\gr_I(R)\right) = d-1.$
\end{enumerate}
\end{theorem}
\begin{proof}		
We write  $S := \Sym(I)$.
From \autoref{thm_Gor_3}, we have an isomorphism of bigraded $S$-modules $\AAA \cong {}^*\Hom_T(S, T)(-d,-d+1)$. 
In order to determine each of the $T$-module $\Hom_T(S_{i}, T)$ we are going to construct a (not necessarily minimal) homogeneous free presentation $$P_i
\xrightarrow{\;\psi_i\;}
Q_i
\rightarrow
S_i
\rightarrow
0$$
of $S_i$.

%Let $\aaa \subset R$ be the ideal generated by the monomial support of $I_1(\varphi)$.
Denote by $M_i$ the set of monomials in $R_i$ that are not in 
the ideal $I_1(\varphi)$.  Let $C = \{x_1x_2, x_1^2,\ldots,x_d^2\}$ be the minimal monomial generating set of $I_1(\varphi)$ and let $C^j$ be the minimal monomial generating set of $(I_1(\varphi))^j$.
For $i \ge 2$ and $\alpha =(\alpha_3,\ldots,\alpha_{d}) \in \{0,1\}^{d-2}$, we define the set 
%$\Dij$  consisting of monomials of degree $i-2,$
$$
\Dij \,:= \, \big\{ x_1^{b_1}x_2^{b_2}x_3^{\alpha_3} \cdots  x_d^{\alpha_{d}} c  \in R_{i-2} \mid c \in C^j \text{ and }  b_1+b_2 \le 1\big\}. 
$$ 
Write $\Dij \cdot C$ for the set of products of elements in $\Dij$ and in $C.$ The particular shape of $C$ easily shows that 
the sets $\Dij$ give a partition of the minimal monomial generating set of $\mm^{i-2} \subset R$, and the sets $\Dij \cdot C$ give a partition 
of the minimal monomial generating set of $\mm^{i-2}I_1(\varphi) \subset R$.

We consider the (not necessarily minimal) free presentation of the graded $T$-module $S_i$  
$$
P_i
\xrightarrow{\;\psi_i\;}
Q_i
\rightarrow
S_i
\rightarrow
0 \, ,
$$
where 
$$
Q_i \;:=\; \bigoplus_{\alpha} \overbrace{\bigoplus_{u \in \Dij } \bigoplus_{c \in C } T \cdot q_{u,c}}^{Q_{i,\alpha}} \;\;\; \oplus \;\;  \overbrace{\bigoplus_{v \in M_i} T \cdot q_v}^{N_i}
$$
and 
$$
P_i \; := \; 
\bigoplus_{\alpha}\underbrace{\left(\bigoplus_{u \in \Dij} \bigoplus_{l=1}^{d+1} T \cdot p_{u,l}  \;\;\; \oplus \bigoplus_{r \in F_{i,\alpha}} T
	\cdot p_r\right)}_{P_{i,\alpha}} ;
$$
here $q_{u,c}$, $q_v$, $p_{u,l}$  and $p_r$ are the chosen basis elements for the free $T$-modules $Q_i$ and $P_i$.
The free $T$-module $B_i$ has the natural description $B_i = \bigoplus_{|\gamma| = i} T \cdot \xx^\gamma \cong T^{\binom{i+d-1}{d-1}}$.
In the free $T$-module $Q_i$, $q_{u,c}$ corresponds to the pair $(u \in \Dij, c \in C)$ and $q_v$ corresponds to the monomial $v \in M_i$.
The map $Q_i \rightarrow S_i$ is defined by sending $q_{u,c} \mapsto \overline{uc} \in S_i$ and $q_v \mapsto \overline{v} \in S_i$.
The sets $\Dij \cdot C$ together with $M_i$ give a partition of the monomials of degree $i$ in $R$.
Hence we have a natural surjection  $Q_i \surjects S_i$.

Notice that there can exist $(u \in \Dij, c \in C)$ and $(u' \in \Dij, c' \in C)$ such that $uc = u'c'$, and as a consequence we need a relation to identify $q_{u,c}$ and $q_{u',c'}$.
The fact that the sets $\Dij \cdot C$ are disjoint implies that there are no further relations among the chosen basis elements of $Q_i$.
Denote by $\Fij$ a set indexing all the needed homogeneous relations (columns in $\psi_i$) among the $q_{u,c}$'s with $u \in \Dij$ and $c \in C$.
Let $p_r \in P_i$ be the basis element that gives the relation $r \in \Fij$.
For any $r \in \Fij$, we have that 
$$
\psi_i(p_r) \, \in \, \bigoplus_{u \in \Dij} \bigoplus_{c \in C} T \cdot q_{u,c} \;\subset \, Q_{i, \alpha}. 
$$
It now follows that $B_i \cong Q_i \big/ \big(\sum_{r \in \Fij} T \cdot \psi_i(p_r)\big)$.
The relations of $S_i$ come from a bihomogeneous generating set of the ideal $\mm^{i-2} \cdot \LL$, which in turn can be chosen to be the union of bihomogeneous generating sets of $\Dij \cdot \LL$.
In the free $T$-module $P_i$, $p_{u,l}$ corresponds to the pair $(u \in \Dij, g_l)$.
Since $g_l = [x_1x_2, x_1^2, \ldots, x_d^2] \cdot [a_{1,l},a_{2,l},\ldots,a_{d,l}]^t$, for any $u \in \Dij$, we have that
$$
\big[\psi_i(p_{u,1}), \psi_i(p_{u,2}),\ldots,\psi_i(p_{u,d+1})\big] \, = \, \big[q_{u,x_1x_2}, q_{u,x_1^2}, \ldots, q_{u,x_d^2} \big] \cdot A.
$$

By combining these facts, it follows that $\psi_i$ is a presentation matrix of $S_i$. 
See \autoref{examp_illus} for an illustration of this presentation.  
%	We chose to work with the non-minimal presentation $\psi_i$ because, as described above, it has a nice decomposition into blocks with the matrix $A$ diagonally plus the trivial relations indexed by the $\Fij$'s.

We have $\Hom_T(S_i, T) = \Ker(\psi_i^t) \subset \Hom_T(Q_i, T),$ and 
%by abuse of notation, 
we denote the dual basis elements of $ \Hom_T(Q_i, T)$ by $q_{u,c}^*$ and $q_v^*.$ The isomorphism $\AAA \cong {}^*\Hom_T(S, T)(-d,-d+1)$ shows that $\Ker(\psi_i^t)=\Hom_T(S_i, T) = 0$ whenever $i> d.$  

We notice that $\psi_i(P_{i,\alpha}) \subset Q_{i,\alpha},$ and we write $\psi_{i,\alpha}: P_{i,\alpha} \to Q_{i,\alpha}$ for the restriction map. Thus we obtain a direct sum decomposition of maps
$$
\psi_i \;=\; \bigoplus_{\alpha} \, \Pij .
$$
%The direct sum decomposition of $\psi_i$ gives
As a consequence,
$$\Hom_T(S_i, T) \,\cong\, \Ker(\psi_i^t) = \bigoplus_{ \alpha }  \Ker(\Pij^t) \; \oplus\;  \bigoplus_{v \in M_i} T \cdot q_v^*\, .
$$

%We divide the matrix $\psi_i$ into blocks of rows:  let $\Pij$ be the submatrix given by the rows corresponding to the basis vectors $\{q_{u,c}\}_{u \in \Dij, c \in C}$ and $\eta_v$ be the submatrix given by the row corresponding with the basis vector $q_v$.
%We have $\Ker(\eta_v^t) = T \cdot q_v$ because $\eta_v$ is the zero matrix.

If $|\alpha|:=\sum  \alpha_i = i-2$, then $j=0$ and $D_{i,\alpha} = \{z_\alpha\}$ where $z_\alpha := x_{3}^{\alpha_3}\cdots x_{d}^{\alpha_d}$, and so there are no relations among the $q_{u,c}$'s indexed by the pairs $(u \in D_{i,\alpha}, c \in C)$.
Hence, when $|\alpha| = i-2$,  \autoref{lem_Gor_3_prop_A}(v) shows that $\Ker(\psi_{i,\alpha}^t)$ is the free $T$-module given by 
$$
\Ker(\psi_{i,\alpha}^t) \;= \; T \cdot \underbrace{\left(\delta_1 q_{z_\alpha,x_1x_2}^* + \delta_2 q_{z_\alpha,x_1^2}^* + \cdots + \delta_{d+1} q_{z_\alpha,x_d^2}^*\right)}_{Z_{i,\alpha}}.
$$	
%Denote by $Z_{i,\alpha} := \delta_1 q_{z_\alpha,x_1x_2} + \delta_2 q_{z_\alpha,x_1^2} + \cdots + \delta_{d+1} q_{z_\alpha,x_d^2}$ the generator of $\Ker(\psi_{i,\alpha}^t)$.
%Since there is a direct sum decomposition of maps
%$$
%\psi_i \;=\; \bigoplus_{\alpha} \Pij \;\oplus\; \bigoplus_{v \in M_i} \eta_v,
%$$
We obtain that
\begin{align*}
	\Ker(\psi_i^t) &\,=\,  \bigoplus_{|\alpha| < i-2 }  \Ker(\Pij^t)  \;\;\; \oplus\; \bigoplus_{ |\alpha| = i-2}  \Ker(\psi_{i,\alpha}^t)  \;\;\; \oplus\; \bigoplus_{v \in M_i} T \cdot q_v^*\\
	&\,=\, \bigoplus_{|\alpha| < i-2 }  \Ker(\Pij^t)  \;\;\; \oplus\; \bigoplus_{ |\alpha| = i-2}  T \cdot Z_{i,\alpha} \;\;\; \oplus\; \bigoplus_{v \in M_i} T \cdot q_v^*.
\end{align*}

We now consider the case $i = d$.
Notice that $M_d = \emptyset$.
Let $\sigma := (1,\ldots,1) \in \NN^{d-2}$.
The $T$-module $\AAA_0$ is free of rank at most $1;$ indeed,
$\AAA_0 \subset T$ is the defining ideal of the special fiber ring of 
$I,$ which is a principal ideal (see the proof of \autoref{lem_Gor_3_prop_A}).
%and it is generated by the equation of the hypersurface given as the %image of the rational map $\GG : \PP_\kk^{d-1} \dashrightarrow \PP_\kk^d$ determined by the generators of $I$.
The isomorphism $\AAA_0 \cong \Hom_T(S_d, T)(-d+1)$ then shows that $\Hom_T(S_d, T) = \Ker(\psi_d^t)$ is also free of rank at most $1,$ hence indecomposable. As moreover $Z_{d,\sigma} \neq 0,$
we conclude that 
$$
\Ker(\psi_d^t) \,=\, \Ker(\psi_{d,\sigma}^t) \,=\, T \cdot Z_{d,\sigma}
$$ 
and  $\Ker(\psi_{d,\alpha}^t) = 0$ for all $\alpha \neq  \sigma$.

For a given $\alpha$ set $\eta:=x_{3}^{1-\alpha_3}\cdots x_{d}^{1-\alpha_{d}}.$ There is a commutative diagram of $T$-linear maps
$$
\begin{tikzpicture}
	\matrix (m) [matrix of math nodes,row sep=4em,column sep=6em,minimum width=2em]
	{
		P_{i,\alpha} & Q_{i, \alpha}\\
		P_{i+d-2-|\alpha|,  \sigma}& Q_{i+d-2-|\alpha|,  \sigma} \\};
	\path[-stealth]
	(m-1-1) edge node [left] {$\cdot \eta$} node [right]{$\nvisom $}  (m-2-1)
	edge  node [above] {$\psi_{i,\alpha}$} (m-1-2)
	(m-2-1.east|-m-2-2) edge 
	node [above] {$\psi_{i+d-2-|\alpha|,  \sigma}$} (m-2-2)
	(m-1-2) edge node [right] {$\cdot \eta$} node [left] {$\visom$}(m-2-2);
\end{tikzpicture}
$$
where the vertical maps are isomorphisms. Thus $\Ker(\Pij^t)\cong\Ker(\psi_{i+d-2-|\alpha|,  \sigma}^t)$.
If $|\alpha| < i-2$, then $i+d-2-|\alpha| > d$, and therefore $\Ker(\psi_{i+d-2-|\alpha|,  \sigma}^t) = 0.$ So it follows that $\Ker(\Pij^t)=0.$

In summary, we have now proved that for all $0 \le i \le d,$
%we conclude that $\Hom_T(S_i, T)$ is the free $T$-module given by 
$$
\Hom_T(S_i, T) = \Ker(\psi_i^t) \,=\,   \bigoplus_{ |\alpha| = i-2}  T \cdot Z_{i,\alpha} \;\;\; \oplus\; \bigoplus_{v \in M_i} T \cdot q_v^*
$$
and that this is a free $T$-module.
%for all $0 \le i \le d$.
Using the natural embedding $\Hom_T(S_i, T) \subset \Hom_T(B_i, T), $ we can rewrite $\Hom_T(S_i, T)$ as 
$$
\Hom_T(S_i, T) \; = \;  \bigoplus_{|\alpha| = i-2}  T \cdot \left(\delta_1 (x_1x_2z_\alpha)^*  + \cdots + \delta_{d+1} (x_d^2z_\alpha)^*\right) \;\;\; \oplus\; \bigoplus_{v \in M_i} T \cdot v^* \;\;\subset \;\; \Hom_T(B_i, T) \, ,
$$
%by considering the reduction $Q_i \surjects B_i \cong Q_i \big/ \big(\sum_{r \in \Fij} T \cdot \psi_i(p_r)\big)$.
where $(x_1x_2z_\alpha)^*, \ldots(x_dz_\alpha)^*, v^*$ denote dual basis elements of $\Hom_T(B_i, T)$. 

%We have the natural isomorphism  of bigraded $B$-modules.
Moreover, using the identification ${}^*\Hom_T(B, T) \cong T[x_1^{-1},\ldots,x_d^{-1}]$, we obtain an isomorphism of bigraded $B$-modules
$$
\Hom_T(S_i, T) \;\cong\;  \bigoplus_{ |\alpha| = i-2}  T \cdot \left(\delta_1 (x_1x_2z_\alpha)^{-1}  + \cdots + \delta_{d+1} (x_d^2z_\alpha)^{-1}\right) \;\;\; \oplus\; \bigoplus_{v \in M_i} T \cdot v^{-1} \;\;\subset \;\; T[x_1^{-1},\ldots,x_d^{-1}].
$$
Since $M_{d-1} = \{ x_1x_3\cdots x_d,\, x_2x_3\cdots x_d \}$, we conclude that 
%there is an isomorphism of bigraded $B$-modules
%${}^*\Hom_T(S, T)$ is the free $T$-module given by
\begin{align*}
	{}^*\Hom_T(S, T) &\;\cong\; B \cdot \left(\delta_1 x_1^{-1}x_2^{-1} + \delta_2x_1^{-2}+ \cdots + \delta_{d+1} x_d^{-2}\right) x_3^{-1}\cdots x_d^{-1}\\
	&\quad + B \cdot x_1^{-1}x_3^{-1}\cdots x_d^{-1}  \\
	&\quad + B \cdot x_2^{-1}x_3^{-1}\cdots x_d^{-1}  \;\subset\; T[x_1^{-1},\ldots,x_d^{-1}].
\end{align*}	
The three generators above are minimal for bidegree reasons.
The bihomogeneous isomorphism 
$$
\AAA = \HL^0(S) \cong {}^*\Hom_T(S, T)(-d,-d+1)
$$ 
shows that $\AAA$ is minimally generated as an $S$-module by three bihomogeneous elements of bidegrees $(0,2d-2),$ $(1,d-1),$ $(1,d-1).$
This completes the proof of part (i).

\vspace{.2cm}

Part (ii) is a direct consequence of part (i) and \autoref{cor_Morley_forms}.
\begin{comment}
	Indeed, under the natural isomorphism ${}^*\Hom_T(B, T) \cong T[x_1^{-1},\ldots,x_d^{-1}]$, any element $\chi \in \left[T[x_1^{-1},\ldots,x_d^{-1}]\right]_{d-i}$ such that $\LL \cdot \chi = 0$ gives a $T$-homomorphism $u_\chi \in \Hom_T(S_{d-i}, \AAA_d) \cong \Hom_T(S_{d-i}, T)(-d+1)$, and we have the equality 
	$$
	\frac{1}{\alpha}\cdot \big(\text{id}_S \otimes (\omicron \circ u_\chi)\big)(\morl_{(i,\delta-i)}) \; = \; \frac{1}{\alpha} \cdot \morl_{(i,\delta-i)} \cdot (1 \otimes_T \chi)
	$$
	in $\AAA_i \subset S_i$.
\end{comment}

\cite[Theorems 5.4 and 5.8(iii)]{CR} and \cite{BCD2020}*{Theorem 2.4(iii)} show that $\deg(\GG) \cdot \deg(Y) = (d-1)2^{d-1}$. 
Since $\deg(Y) = 2(d-1)$ by part (i), it follows that $\deg(\GG) = 2^{d-2}$. 
Thus part (iii) is proven.

Part (iv) is a consequence of part (i), \autoref{prop_depth_Rees}, \cite{H82}*{proof of Proposition 1.1}, and \cite{PU99}*{Theorem 3.1}.
\end{proof}

%	The parts in blue color are not part of the matrix $M_2$, and they are used display the ordering that we chose for the natural bases of the source and target of $M_2$.
%
%The relations of $S_3$ come from the generators of the ideal $\mm \cdot \LL$, and this gives a total of $20$ relations.
%The monomial ideal $\mm \cdot (x_1^2, x_2^2, x_3^2, x_4^2, x_1x_2)$ only covers $18$ of the monomials of degree $2$ in $R$, and it misses the two monomials $x_1x_3x_4$ and $x_2x_3x_4$.
%Since we are repeating twice the basis elements in $B_2$ that correspond with the monomials $x_1x_2^2$ and $x_1^2x_2$, we need to add the last two columns (relations) that identify the second row with the tenth and the fifth row with the sixth.
%By combining these basic facts, we obtain that $M_3$ is a presentation matrix of $S_3$.

\begin{example}
\label{examp_illus}
Let $d = 4$, $R = \kk[x_1,\ldots,x_4]$, $T = \kk[y_1,\ldots,y_5]$ and $h = 2$.
In this case, the $T$-module $S_3$ has (a non-minimal) presentation given by $T^{22} \xrightarrow{\;\;\psi_3\;\;} T^{22} \rightarrow S_3 \rightarrow 0$, where 

$$
\scriptsize
\psi_3 = 
\left[
\begin{smallmatrix}%{l|cccccccccccccccccccccc}
	& \color{blue}x_1g_1& \color{blue}x_1g_2 & \color{blue}x_1g_3 & \color{blue}x_1g_4 & \color{blue}x_1g_5 & \color{blue}x_2g_1 & \color{blue}x_2g_2 & \color{blue}x_2g_3 & \color{blue}x_2g_4 & \color{blue}x_2g_5 & & & \color{blue}x_3g_1 & \color{blue}x_3g_2 & \color{blue}x_3g_3 & \color{blue}x_3g_4 & \color{blue}x_3g_5 & \color{blue}x_4g_1 & \color{blue}x_4g_2 & \color{blue}x_4g_3 & \color{blue}x_4g_4 & \color{blue}x_4g_5 \smallskip\\\hline
	\color{blue}x_1\cdot x_1x_2 | & \color{red}a_1 &\color{red}b_1 &\color{red}c_1 &\color{red}d_1 & \color{red}e_1 &\color{red}0 &\color{red}0 & \color{red}0 &\color{red}0 &\color{red}0 &\color{red}0 & \color{red} 0 & 0 & 0 & 0 & 0 & 0 & 0 & 0 & 0 & 1 & 0 \\
	\color{blue}x_1\cdot x_1^2 | &\color{red} a_2 & \color{red}b_2 &\color{red}  c_2 &\color{red}  d_2 & \color{red} e_2 & \color{red}0 & \color{red}0 & \color{red}0 &\color{red} 0 &\color{red} 0 &\color{red} 0 & \color{red}0 & 0 & 0 & 0 & 0 & 0 & 0 & 0 & 0 & 0 & 0 \\
	\color{blue}x_1\cdot x_2^2 | & \color{red} a_3 &\color{red} b_3 & \color{red}c_3 & \color{red}d_3 &\color{red} e_3 &\color{red} 0 & \color{red}0 & \color{red}0 &\color{red} 0 &\color{red} 0 &\color{red} 0 & \color{red} 0 & 0 & 0 & 0 & 0 & 0 & 0 & 0 & 0 & 0 & 1 \\
	\color{blue}x_1\cdot x_3^2 | & \color{red} a_4 & \color{red}b_4 & \color{red}c_4 &\color{red} d_4 &\color{red} e_4 &\color{red} 0 &\color{red} 0 &\color{red} 0 &\color{red} 0 &\color{red} 0 &\color{red} 0 & \color{red}0 & 0 & 0 & 0 & 0 & 0 & 0 & 0 & 0 & 0 & 0 \\
	\color{blue}x_1\cdot x_4^2 | &\color{red} a_5 &\color{red} b_5 & \color{red}c_5 &\color{red} d_5 & \color{red}e_5 &\color{red} 0 & \color{red}0 & \color{red}0 &\color{red} 0 &\color{red} 0 &\color{red} 0 & \color{red} 0 & 0 & 0 & 0 & 0 & 0 & 0 & 0 & 0 & 0 & 0 \\
	\color{blue}x_2\cdot x_1x_2 | & \color{red} 0 & \color{red}0 &\color{red} 0 & \color{red}0 &\color{red} 0 & \color{red}a_1 & \color{red}b_1 &\color{red} c_1 &\color{red} d_1 &\color{red} e_1 &  \color{red}0 & \color{red}-1 & 0 & 0 & 0 & 0 & 0 & 0 & 0 & 0 & 0 & 0 \\
	\color{blue}x_2\cdot x_1^2 | &\color{red}  0 & \color{red}0 & \color{red}0 &\color{red} 0 &\color{red} 0 & \color{red}a_2 & \color{red}b_2 &\color{red} c_2 & \color{red}d_2 & \color{red}e_2 & \color{red} -1 &\color{red} 0 & 0 & 0 & 0 & 0 & 0 & 0 & 0 & 0 & 0 & 0 \\
	\color{blue}x_2\cdot x_2^2 | & \color{red} 0 &\color{red} 0 & \color{red}0 & \color{red}0 &\color{red} 0 & \color{red}a_3 &\color{red} b_3 & \color{red}c_3 & \color{red}d_3 &\color{red} e_3 &\color{red}  0 & \color{red} 0 & 0 & 0 & 0 & 0 & 0 & 0 & 0 & 0 & 0 & 0 \\
	\color{blue}x_2\cdot x_3^2 | &\color{red}  0 &\color{red} 0 & \color{red}0 & \color{red}0 & \color{red}0 & \color{red}a_4 &\color{red} b_4 & \color{red}c_4 &\color{red} d_4 &\color{red} e_4 &\color{red}  0 & \color{red}0 & 0 & 0 & 0 & 0 & 0 & 0 & 0 & 0 & 0 & 0 \\
	\color{blue}x_2\cdot x_4^2 | & \color{red} 0 & \color{red} 0 & \color{red}0 & \color{red} 0 & \color{red} 0 & \color{red}a_5 & \color{red}b_5 & \color{red} c_5 &\color{red} d_5 & \color{red} e_5 & \color{red} 0 & \color{red} 0 & 0 & 0 & 0 & 0 & 0 & 0 & 0 & 0 & 0 & 0 \\
	\color{blue}x_3\cdot x_1x_2 | & 0 & 0 & 0 & 0 & 0 & 0 & 0 & 0 & 0 & 0 & 0 & 0 & \color{green} a_1 & \color{green} b_1 &\color{green}  c_1 & \color{green} d_1 & \color{green} e_1 &   0 & 0 & 0 & 0 & 0  \\
	\color{blue}x_3\cdot x_1^2 | & 0 & 0 & 0 & 0 & 0 & 0 & 0 & 0 & 0 & 0 & 0 & 0 & \color{green} a_2 & \color{green} b_2 & \color{green} c_2 & \color{green} d_2 & \color{green} e_2 &   0 & 0 & 0 & 0 & 0 \\
	\color{blue}x_3\cdot x_2^2 | & 0 & 0 & 0 & 0 & 0 & 0 & 0 & 0 & 0 & 0 & 0 & 0 & \color{green} a_3 & \color{green} b_3 & \color{green} c_3 & \color{green} d_3 & \color{green} e_3 & 0 & 0 &0 & 0 & 0  \\
	\color{blue}x_3\cdot x_3^2 | & 0 & 0 & 0 & 0 & 0 & 0 & 0 & 0 & 0 & 0 & 0 & 0 & \color{green} a_4 & \color{green} b_4 & \color{green} c_4 & \color{green} d_4 & \color{green} e_4 &0 & 0 & 0 & 0 & 0\\
	\color{blue}x_3\cdot x_4^2 | & 0 & 0 & 0 & 0 & 0 & 0 & 0 & 0 & 0 & 0 & 0 & 0 & \color{green} a_5 & \color{green} b_5 & \color{green} c_5 & \color{green} d_5 & \color{green}  e_5 &   0 & 0 & 0 & 0 & 0\\
	\color{blue}x_4\cdot x_1x_2 | & 0 & 0 & 0 & 0 & 0 & 0 & 0 & 0 & 0 & 0 & 0 & 0 & 0 & 0 & 0 & 0 & 0 & \color{orange} a_1 &\color{orange} b_1 & \color{orange}c_1 &\color{orange}  d_1 & \color{orange} e_1\\
	\color{blue}x_4\cdot x_1^2 | & 0 & 0 & 0 & 0 & 0 & 0 & 0 & 0 & 0 & 0 & 0 & 0 & 0 & 0 & 0 & 0 & 0 & \color{orange} a_2 & \color{orange} b_2 & \color{orange} c_2 & \color{orange} d_2 & \color{orange} e_2 \\
	\color{blue}x_4\cdot x_2^2 | & 0 & 0 & 0 & 0 & 0 & 0 & 0 & 0 & 0 & 0 & 0 & 0 & 0 & 0 & 0 & 0 & 0 & \color{orange} a_3 & \color{orange} b_3 &\color{orange} c_3 & \color{orange} d_3 & \color{orange} e_3 \\
	\color{blue}x_4\cdot x_3^2 | & 0 & 0 & 0 & 0 & 0 & 0 & 0 & 0 & 0 & 0 & 0 & 0 & 0 & 0 & 0 & 0 & 0 & \color{orange} a_4 &\color{orange}  b_4 &\color{orange} c_4 &\color{orange} d_4 & \color{orange} e_4  \\
	\color{blue}x_4\cdot x_4^2 | & 0 & 0 & 0 & 0 & 0 & 0 & 0 & 0 & 0 & 0 & 0 & 0 & 0 & 0 & 0 & 0 & 0 & \color{orange} a_5 &\color{orange} b_5 & \color{orange}c_5 &\color{orange} d_5 & \color{orange} e_5 \\
	\color{blue} x_1x_3x_4 | & 0 & 0 & 0 & 0 & 0 & 0 & 0 & 0 & 0 & 0 & 0 & 0 & 0 & 0 & 0 & 0 & 0 & 0 & 0 & 0 &    0 & 0 \\
	\color{blue} x_2x_3x_4 | & 0 & 0 & 0 & 0 & 0 & 0 & 0 & 0 & 0 & 0 & 0 & 0 & 0 & 0 & 0 & 0 & 0 & 0 & 0 & 0 &    0 & 0 \\
\end{smallmatrix}	
\right].
$$
\smallskip

The first $10$ rows correspond to the basis elements $q_{u,c}$ where $u \in D_{3,(0,0)}$.
The next $5$ rows correspond to the basis elements $q_{u,c}$ where $u \in D_{3, (1, 0)}$.
The next $5$ rows correspond to the basis elements $q_{u,c}$ where $u \in D_{3, (0, 1)}$.
The last two rows correspond to the two basis elements $q_{x_1x_3x_4}$ and $q_{x_2x_3x_4}$.
The columns indexed by the $x_ig_j$'s correspond to the basis vectors $p_{u,l}$.
The matrix  also shows the direct sum decomposition of $\psi_3={\color{red}\psi_{3, \{0,0\}}} \oplus \color{green} \psi_{3, \{1, 0\}} \oplus \color{orange} \psi_{3, \{0,1\}}$. 
\end{example}

\begin{remark}
\label{rem_bound_gen_sharp}
The fact that $\AAA$ is generated in $\xx$-degrees at most one, proved in \autoref{thm_sublime_h_2}, was already predicted by \autoref{thm_Jou_dual_Sym}(iv).
This also shows that the bound in \autoref{thm_Jou_dual_Sym}(iv) is sharp. The fact that the symmetric algebra and the Rees algebra of $I$ first differ in $\yy$-degree $d-1$ was already proved in \cite{UV93}*{Theorem 2.5}. There it is also shown
that $\AAA_{(\star, d-1)}$ is isomorphic to a shift of $\Ext^d_R(R/I_1(\varphi), R).$ In the setting of \autoref{thm_sublime_h_2}, 
this already implies that $\AAA$ has exactly two minimal generators in $\yy$-degree $d-1.$ 

%In \cite{UV93}*{Theorem 2.5} is where the authors also proved that  $\AAA_{\star, d-1}\cong \Ext^d_R(R/I_1(\varphi), R),$ in particular \cite{UV93}*{Theorem 2.5} shows already that $\AAA$ has exactly two generators in $\yy$-degree  $d-1.$
\end{remark}

\section*{Acknowledgments}
We thank the reviewer for carefully reading our paper and for several comments and corrections.
C.P.~was partially supported by NSF grant DMS-2201110.
B.U.~was partially supported by NSF grant DMS-2201149.
Y.C.R.~is grateful to the mathematics departments of Purdue University and the University of Notre Dame, where the majority of the work was done, for their hospitality and for excellent working conditions.

%
%\begin{theorem}
%	(general case)
%\end{theorem}

\vspace{0.3cm}

\bibliography{references}
\end{document}